\documentclass[11pt]{amsart}

\usepackage{geometry}
\usepackage{setspace}
\usepackage{enumerate}

\usepackage{amsmath,amssymb,graphicx,amscd}
\usepackage{mathtools}
\usepackage{color}
\usepackage{enumitem}
\usepackage{xcolor}
\usepackage{float}
\usepackage[numbers]{natbib}
\usepackage[all]{xy}
\usepackage{amsthm}
\usepackage{tikz}
\usepackage{tikz-cd}
\usepackage{mathtools}
\usepackage{fancyhdr}
\usepackage{hyperref}
\hypersetup{colorlinks=true,citecolor=blue,linkcolor=cyan!100,urlcolor=black,filecolor=black}
\numberwithin{figure}{section}
\numberwithin{equation}{section}
\allowdisplaybreaks
\numberwithin{equation}{section}

\newtheorem{theorem}{Theorem}[section]
\newtheorem{proposition}[theorem]{Proposition}
\newtheorem{corollary}[theorem]{Corollary}
\newtheorem{lemma}[theorem]{Lemma}

\theoremstyle{definition}
\newtheorem{remark}[theorem]{Remark}
\newtheorem{example}[theorem]{Example}
\newtheorem{definition}[theorem]{Definition}
\newtheorem{convention}[theorem]{Convention}

\newcommand{\hooklongrightarrow}{\lhook\joinrel\longrightarrow}

\newcommand{\para}[1]{\medskip\noindent\textbf{#1.}}

\DeclareMathOperator{\Aut}{Aut}
\DeclareMathOperator{\kernel}{ker}

\definecolor{ASTRAL}{RGB}{46,116,181}
\author{Anderson Vera}
\address{Institut de Recherche Math\'ematique Avanc\'ee,  Universit\'e de Strasbourg 7 rue Ren\'e Descartes,  67084 Strasbourg, France}
\email{vera@math.unistra.fr}
\title{Alternative versions of the Johnson homomorphisms and the LMO functor}
\begin{document}
\begin{abstract}
Let $\Sigma$ be a compact connected oriented surface with one boundary component and let $\mathcal{M}$ denote the mapping class group of $\Sigma$. By considering the action of $\mathcal{M}$ on the fundamental group of $\Sigma$ it is possible to define different filtrations of $\mathcal{M}$ together with some homomorphisms on each term of the filtration. The aim of this paper is twofold. Firstly we study a filtration of $\mathcal{M}$ introduced recently by Habiro and Massuyeau, whose definition involves a handlebody bounded by $\Sigma$. We shall call it the \emph{``alternative Johnson filtration"}, and the corresponding homomorphisms are referred to as \emph{``alternative Johnson homomorphisms"}. We provide a comparison between the alternative Johnson filtration and two previously known filtrations: the original Johnson filtration and the Johnson-Levine filtration. Secondly, we study the relationship between the alternative Johnson homomorphisms and the functorial extension of the Le-Murakami-Ohtsuki invariant of $3$-manifolds. We prove that these homomorphisms can be read in the tree reduction of the LMO functor. In particular, this provides a new reading grid for the tree reduction of the LMO functor.
\end{abstract}

\maketitle
\tableofcontents

\section{Introduction}\label{Section1}

Let $\Sigma$ be a compact connected oriented surface with one boundary component and let $\mathcal{M}$ denote the \emph{mapping class group} of $\Sigma$, that is, the group of isotopy classes of orientation-preserving self-homeomorphisms of $\Sigma$ fixing the boundary pointwise.  The group $\mathcal{M}$ is not only an important object in the study of the topology of surfaces but also plays an important role  in the study of $3$-manifolds, Teichm\"uller spaces, topological quantum field theories, among other branches of mathematics. 

A natural way to study $\mathcal{M}$ is to analyse the way  it acts on other objects. For instance, we can consider the action  on the first homology group $H:=H_1(\Sigma;\mathbb{Z})$ of $\Sigma$. This action gives rise to the so-called \emph{symplectic representation}
$$\sigma:\mathcal{M}\longrightarrow \text{Sp}(H,\omega),$$
where $\omega:H\otimes H \rightarrow \mathbb{Z}$ is the intersection form of $\Sigma$. The homomorphism $\sigma$ is surjective but it is far from being injective. Its kernel is known as the \emph{Torelli group} of $\Sigma$,  denoted by $\mathcal{I}$. Hence we have the short exact sequence
\begin{equation}\label{sesSp}
1\longrightarrow \mathcal{I}\xrightarrow{\ \subset \ } \mathcal{M} \xrightarrow{\ \sigma \ } \text{Sp}(H,\omega)\longrightarrow 1.
\end{equation}
We can see that, in order to understand the algebraic structure of $\mathcal{M}$, the Torelli group~$\mathcal{I}$ deserves significant attention  because, in a certain way, it is the part of $\mathcal{M}$ that is beyond linear algebra (at least with respect to the symplectic representation).

More interestingly, we can consider the action of $\mathcal{M}$ on the fundamental group $\pi:=\pi_1(\Sigma,*)$ for a fixed point $*\in\partial \Sigma$. This way we obtain an injective homomorphism 
$$\rho:\mathcal{M}\longrightarrow \text{Aut}(\pi),$$
which is known as the \emph{Dehn-Nielsen-Baer representation} and whose image is the subgroup of automorphisms of $\pi$ that fix the homotopy class of the boundary of $\Sigma$.

\para{Johnson-type filtrations} As stepwise approximations of $\rho$, we can consider the action of $\mathcal{M}$ on the nilpotent quotients of $\pi$
$$\rho_m:\mathcal{M}\longrightarrow\text{Aut}(\pi/\Gamma_{m+1}\pi),$$
where $\Gamma_{1}\pi:=\pi$ and $\Gamma_{m+1}\pi:=[\pi,\Gamma_{m}\pi]$ for $m\geq 1$, define the \emph{lower central series} of~$\pi$. This is the approach pursued by D. Johnson \cite{MR718141} and S. Morita \cite{MR1224104}. This approach allows to define the \emph{Johnson filtration} 
\begin{equation}
\mathcal{M}\supseteq \mathcal{I}=J_1\mathcal{M}\supseteq J_2\mathcal{M}\supseteq J_3\mathcal{M}\supseteq\cdots
\end{equation}
where $J_m\mathcal{M}:=\text{ker}(\rho_m)$.

Now, there is a deep interaction between the study of $3$-manifolds and that of the mapping
class group. For instance through \emph{Heegaard splittings}, that is, by gluing two handlebodies via an element of the mapping class group of their common boundary. Thus, if we are interested in this  interaction, it is natural to consider the surface $\Sigma$ as the boundary of a handlebody $V$. Let $\iota:\Sigma\hookrightarrow V$ denote the induced inclusion and let $B := H_1(V;\mathbb{Z})$ and $\pi':=\pi_1(V,\iota(*))$. Let $A$ and $\mathbb{A}$ be the subgroups $\text{ker}(H\stackrel{\iota_*}{\longrightarrow} B)$ and $\text{ker}(\pi\stackrel{\iota_{\#}}{\longrightarrow}\pi')$, where $\iota_*$ and $\iota_{\#}$ are the induced maps by $\iota$ in homology and homotopy, respectively. The \emph{Lagrangian mapping class group} of $\Sigma$ is the group 
$$\mathcal{L}=\{f\in\mathcal{M}\ |\ f_*(A)\subseteq A\}.$$

By considering a descending series $(K_m)_{m\geq 1}$ of normal subgroups of $\pi$ (different from the lower central series) K. Habiro and G. Massuyeau introduced in \cite{MR3828784} a filtration of the Lagrangian mapping class group $\mathcal{L}$:
\begin{equation}
\mathcal{L}\supseteq \mathcal{I}^{\mathfrak{a}}=J_1^{\mathfrak{a}}\mathcal{M}\supseteq J_2^{\mathfrak{a}}\mathcal{M}\supseteq J_3^{\mathfrak{a}}\mathcal{M}\supseteq\cdots 
\end{equation}
that we  call the \emph{alternative Johnson filtration}.  We call the first term $\mathcal{I}^{\mathfrak{a}}:=J_1^{\mathfrak{a}}\mathcal{M}$ of this filtration   the \emph{alternative Torelli group}. Notice that $\mathcal{I}^{\mathfrak{a}}$ is a normal subgroup of $\mathcal{L}$ but it is not normal in $\mathcal{M}$. Roughly speaking, the group $K_m$ consists of commutators of $\pi$ of \emph{weight} $m$, where the elements of $\mathbb{A}$ are considered to have weight $2$, for instance $K_1=\pi$, $K_2=\mathbb{A}\cdot \Gamma_2\pi$, $K_3= [\mathbb{A},\pi]\cdot \Gamma_3\pi$ and so on. The alternative Johnson filtration will be our main object of study in Section \ref{JTH1}.

Besides, in \cite{MR1823501,MR2265877} J. Levine defined a different filtration of $\mathcal{L}$ by considering the lower central series of $\pi'$, and whose first term is the \emph{Lagrangian Torelli group}  $\mathcal{I}^L=\{f\in\mathcal{L}\ |\ f_*|_A=\text{Id}_A\}$:
\begin{equation}
\mathcal{L}\supseteq \mathcal{I}^L=J_1^L\mathcal{M}\supseteq J_2^L\mathcal{M}\supseteq J_3^L\mathcal{M}\supseteq \cdots
\end{equation}
 we call this filtration the \emph{Johnson-Levine filtration}. The group $\mathcal{I}^L$ is normal in $\mathcal{L}$ but not in $\mathcal{M}$.
 
We refer to the Johnson filtration, the alternative Johnson filtration and the Johnson-Levine filtration as \emph{Johnson-type filtrations}. Notice that unlike the Johnson filtration the alternative Johnson filtration takes into account a handlebody. Besides, the intersection of all terms in the alternative Johnson filtration is the identity of $\mathcal{M}$ as in the case of the Johnson filtration. But this is not the case for the Johnson-Levine filtration. One of the main purposes of this paper is the study of the alternative Johnson filtration and its relation with the other two filtrations. Proposition \ref{prop1altJF} and Proposition \ref{JTHproposition5} give the following result.

\para{Theorem A} \emph{The alternative Johnson filtration satisfies the following properties.
 \begin{enumerate}
\item [\emph{(i)}]  $\bigcap_{m\geq 1}J^{\mathfrak{a}}_m\mathcal{M}=\{\text{\emph{Id}}_{\Sigma}\}$.
\item [\emph{(ii)}] For all $k\geq 1$ the group $J^{\mathfrak{a}}_k\mathcal{M}$ is residually nilpotent, that is, $\bigcap_m\Gamma_mJ_k^{\mathfrak{a}}\mathcal{M}=\{\text{\emph{Id}}_{\Sigma}\}$.
\end{enumerate}
Besides, for every $m\geq 1$, we have
\begin{center}
\begin{tabular}{ccccccc}
\emph{(iii)}\ \  $J_{2m}^{\mathfrak{a}}\mathcal{M}\subseteq J_m\mathcal{M}$. &\ \  & &\emph{(iv)}\ \  $J_m\mathcal{M}\subseteq J_{m-1}^{\mathfrak{a}}\mathcal{M}$.&\ \   & & \emph{(v)}\ \  $J_{m}^{\mathfrak{a}}\mathcal{M}\subseteq J_{m+1}^L\mathcal{M}$.
\end{tabular}
\end{center}
In particular, the Johnson filtration and the alternative Johnson filtration are cofinal.}

\para{Johnson-type homomorphisms} Each term of the Johnson-type filtrations comes  with a homomorphism whose kernel is the next subgroup in the filtration. We refer to these homomorphisms as \emph{Johnson-type homomorphisms}. The \emph{Johnson homomorphisms} are  important tools to understand the structure of the Torelli group and the topology of homology $3$-spheres \cite{MR793179,MR1014464,MR1133875,MR1617632}. Let us review the target spaces of these homomorphisms. For an abelian group $G$, we denote by  $\mathfrak{Lie}(G)=\bigoplus_{m\geq 1}\mathfrak{Lie}_m(G)$  the graded  Lie algebra freely generated by $G$ in degree $1$. 

The \emph{$m$-th Johnson homomorphism} $\tau_m$ is defined on $J_m\mathcal{M}$ and it takes values in the group $\text{Der}_m(\mathfrak{Lie}(H))$ of  degree $m$ derivations of $\mathfrak{Lie}(H)$. Consider the element $\Omega\in\mathfrak{Lie}_2(H)$ determined by the intersection form $\omega:H\otimes H\rightarrow \mathbb{Z}$. A \emph{symplectic} derivation $d$ of $\mathfrak{Lie}(H)$ is a derivation satisfying $d(\Omega)=0$. S. Morita shows in \cite{MR1224104} that for $h\in J_m\mathcal{M}$, the morphism $\tau_m(h)$ defines a symplectic derivation of $\mathfrak{Lie}(H)$. The group of symplectic degree $m$ derivations of $\mathfrak{Lie}(H)$ can be canonically identified with the kernel $D_m(H)$ of the Lie bracket $[\ ,\ ]:H\otimes\mathfrak{Lie}_{m+1}(H)\rightarrow \mathfrak{Lie}_{m+2}(H)$. This way, for $m\geq 1$ we have  homomorphisms 
$$\tau_m: J_m\mathcal{M}\longrightarrow D_m(H).$$

The \emph{$m$-th Johnson-Levine homomorphism} $\tau^L_m: J_m^L\mathcal{M}\rightarrow D_m(B)$ is defined on $J_m^L\mathcal{M}$ and it takes values in the kernel $D_m(B)$ of the Lie bracket $[\ ,\ ]:B\otimes\mathfrak{Lie}_{m+1}(B)\rightarrow \mathfrak{Lie}_{m+2}(B)$. 

For the \emph{alternative Johnson homomorphisms} \cite{MR3828784}, consider the graded  Lie algebra   $\mathfrak{Lie}(B;A)$ freely generated by $B$ in degree $1$ and $A$ in degree $2$. The \emph{$m$-th alternative Johnson homomorphism} $\tau^{\mathfrak{a}}_m: J^{\mathfrak{a}}_m\mathcal{M}\rightarrow \text{Der}_m(\mathfrak{Lie}(B;A))$ is defined on $J_m^{\mathfrak{a}}\mathcal{M}$ and it takes values in the group $\text{Der}_m(\mathfrak{Lie}(B;A))$ of degree $m$ derivations of $\mathfrak{Lie}(B;A)$. Similarly to the case of $\mathfrak{Lie}(H)$, we define a notion of \emph{symplectic derivation} of $\mathfrak{Lie}(B;A)$ by considering the element $\Omega'\in\mathfrak{Lie}_3(B;A)$ defined by the intersection form of the handlebody~$V$. Theorem \ref{JTH2thm1} and Proposition \ref{JTH2prop11} give the following result.

\para{Theorem B} \emph{Let $m\geq 1$ and $h\in J^{\mathfrak{a}}_m\mathcal{M}$. Then
\begin{enumerate}
\item [\emph{(i)}] The morphism $\tau^{\mathfrak{a}}_m(h)$ defines a  degree $m$ symplectic  derivation of $\mathfrak{Lie}(B;A)$.
\item [\emph{(ii)}] The morphism $\tau^{L}_{m+1}(h)$ is determined by the morphism $\tau^{\mathfrak{a}}_m(h)$.
\end{enumerate}}

Property (ii) in Theorem B can be expressed more precisely by the commutativity of the diagram
\begin{equation*}
\xymatrix{  J^{\mathfrak{a}}_m\mathcal{M}\ar[r]^{\subset}\ar[d]_{\tau^{\mathfrak{a}}_m} & J_{m+1}^L\mathcal{M} \ar[d]^{\tau_{m+1}^L} \\
						D_m(B;A)\ar[r]^{\iota_*} & D_{m+1}(B),}
\end{equation*}
for $m\geq 1$, where the inclusion $J^{\mathfrak{a}}_m\mathcal{M}\subseteq J^L_{m+1}\mathcal{M}$ is assured by Theorem A (v). The homomorphism $\iota_*:D_m(B;A)\rightarrow D_{m+1}(B)$ is induced by the map $\iota_*:H\rightarrow B$. Property~(i) in Theorem~B  allows  to define a \emph{diagrammatic version} of the alternative Johnson homomorphisms so  that we are able to study their relation to the \emph{LMO functor}. This is the second main purpose of this paper. Before we proceed with a description of our results in this setting, let us state another result in the context of the alternative Johnson homomorphisms. In \cite{MR3828784}, K. Habiro and G. Massuyeau  consider a group homomorphism $\tau_0^{\mathfrak{a}}:\mathcal{L}\rightarrow \text{Aut}(\mathfrak{Lie}(B;A))$, which we call the \emph{$0$-th alternative Johnson homomorphism}, and whose kernel is the alternative Torelli group $\mathcal{I}^{\mathfrak{a}}$. In subsection \ref{subsection5.3} we prove the following.

\para{Theorem C} \emph{The homomorphism $\tau_0^{\mathfrak{a}}:\mathcal{L}\rightarrow \text{\emph{Aut}}(\mathfrak{Lie}(B;A))$ can be equivalently described as a group homomorphism $\tau^{\mathfrak{a}}_0: \mathcal{L}\longrightarrow \text{\emph{Aut}}(B)\ltimes \text{\emph{Hom}}(A, \Lambda^2 B)$ for a certain action of $\text{\emph{Aut}}(B)$  on $\text{\emph{Hom}}(A, \Lambda^2 B)$. The kernel of $\tau_0^{\mathfrak{a}}$ is the second term $J^L_2\mathcal{M}$ of the Johnson-Levine filtration. In particular we have $\mathcal{I}^{\mathfrak{a}}=J^{\mathfrak{a}}_1\mathcal{M}=J^L_2\mathcal{M}$.}

\medskip

Moreover, we explicitly describe  the image $\mathcal{G}:=\tau_0^{\mathfrak{a}}(\mathcal{L})$ and then we obtain the short exact sequence 
\begin{equation}\label{sesG}
1\longrightarrow \mathcal{I}^{\mathfrak{a}}\xrightarrow{\ \subset\ } \mathcal{L}\xrightarrow{ \ \tau_0^{\mathfrak{a}}\  } \mathcal{G}\longrightarrow 1.
\end{equation}

This short exact sequence  has a similar role, in the context of the alternative Johnson homomorphisms, to that of the short exact sequence (\ref{sesSp}) in the context of the Johnson homomorphisms. This is because in \cite{MR3828784} the authors prove that the alternative Johnson homomorphisms satisfy an equivariant property with respect to the homomorphism $\tau_0^{\mathfrak{a}}$, which  is the analogue of the \text{Sp}-equivariant property of the Johnson homomorphisms. Hence the short exact sequence  (\ref{sesG}) can be  important  for a further development of the study of the alternative Johnson filtration.

\para{Relation with the LMO functor} After the discovery of the Jones polynomial and the advent of many new invariants, the so-called \emph{quantum invariants},  of links and $3$-manifolds, it became necessary to ``organize" these invariants. The theory  of \emph{finite-type (Vassiliev-Goussarov) invariants}  in the case of links and the theory of \emph{finite-type (Goussarov-Habiro) invariants} in the case of $3$-manifolds, provide an efficient way to do this task. An important success was achieved with the introduction of the \emph{Kontsevich integral} for links \cite{MR1237836,MR1318886} and the \emph{Le-Murakami-Othsuki invariant} for $3$-manifolds \cite{MR1604883}, because they are \emph{universal} among rational finite-type invariants. Roughly speaking, this property says that every $\mathbb{Q}$-valued finite-type invariant  is determined by the Kontsevich integral in the case of links or by the LMO invariant in the case of homology $3$-spheres. 

The LMO invariant was extended to a TQFT (Topological quantum field theory) in \cite{MR1473309,MR2304469,MR2403806}. We follow the work of D. Cheptea, K. Habiro and G. Massuyeau in~\cite{MR2403806}, where they extend the LMO invariant  to a functor $\widetilde{Z}:\mathcal{LC}ob_q\rightarrow {}^{ts}\!\!\mathcal{A}$, called the  \emph{LMO functor}, from the category of \emph{Lagrangian cobordisms} (cobordisms satisfying a homological condition)  between bordered surfaces to a category of \emph{Jacobi diagrams} (uni-trivalent colored graphs up to some relations). See Figure \ref{figuraintro1artalt} for some examples of Jacobi diagrams. There is  still   a lack of understanding of the topological information encoded by the LMO functor. One reason for this is that the  construction of the LMO functor takes several steps and  also uses several combinatorial operations on the space of Jacobi diagrams. This motivates the search of topological interpretations of some reductions of the LMO functor through  known invariants, some results in this direction were obtained in \cite{MR2403806,MR2903772,vera1}. The second main purpose of this paper is to give a topological interpretation of the tree reduction of the LMO functor through the alternative Johnson homomorphisms.

 A \emph{homology cobordism} of $\Sigma$ is a homeomorphism class of  pairs $(M,m)$ where $M$ is a compact oriented $3$-manifold and $m:\partial(\Sigma\times[-1,1])\rightarrow \partial M$ is an orientation-preserving homeomorphism such that the \emph{top} and \emph{bottom} restrictions $m_{\pm}|_{\Sigma\times\{\pm 1\}}:\Sigma\times\{\pm1\}\rightarrow M$ of $m$ induce isomorphisms in homology. Denote by $\mathcal{C}$ the monoid of homology cobordisms of $\Sigma$ (or $\mathcal{C}_{g,1}$ where $g$ is the genus of $\Sigma$). In particular, if $h\in\mathcal{M}$, we can consider the homology cobordism $c(h):=(\Sigma\times[-1,1], m^h)$ where $m^h$ is  such that $m^h_+=h$ and $m^h_-=\text{Id}_{\Sigma}$. Moreover,  $h\in \mathcal{L}$ if and only if the cobordism $c(h)$ is a Lagrangian cobordism. Thus $c(h)$ belongs to the source category of the LMO functor and therefore we can compute $\widetilde{Z}(c(h))$.
 
 The alternative Johnson homomorphisms motivate the definition of the \emph{alternative degree}, denoted  $\mathfrak{a}$-deg, for connected tree-like Jacobi diagrams. If $T$ is a tree-like Jacobi diagram colored by $B\oplus A$, then
$$\mathfrak{a}\text{-deg}(T)=2|T_A|+ |T_B| -3,$$
where $|T_A|$ (respectively $|T_B|$) denotes the number of univalent vertices of $T$ colored by $A$ (respectively by $B$). See Figure \ref{figuraintro1artalt} $(a)$ and $(b)$ for some examples.
\begin{figure}[ht!]
										\centering
                        \includegraphics[scale=0.775]{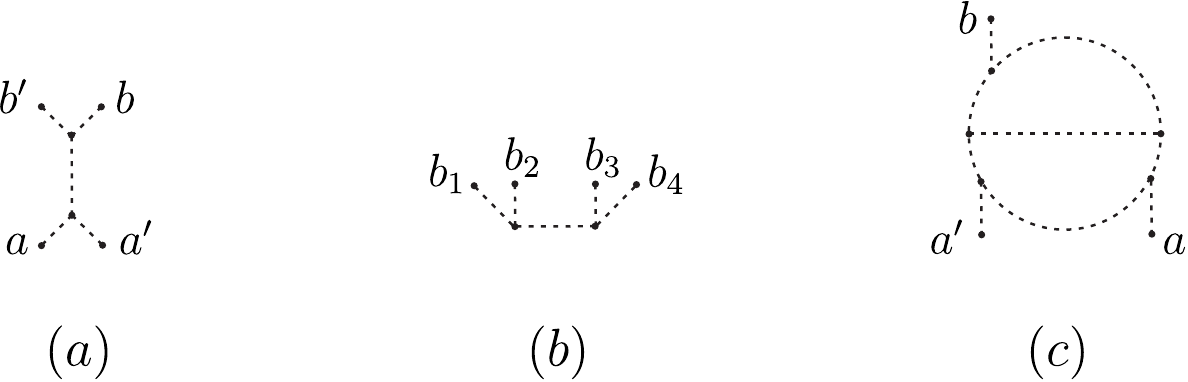}
												\caption{Tree-like Jacobi diagrams of $\mathfrak{a}$-deg $=3$ in $(a)$ and of $\mathfrak{a}$-deg $=1$ in $(b)$. $(c)$ Looped Jacobi diagram. Here $a,a'\in A$ and $b,b',b_1,\ldots,b_4\in B$.}
\label{figuraintro1artalt} 							
\end{figure}

 Denote by $\mathcal{T}^{Y,\mathfrak{a}}_m(B\oplus A)$ the space generated by  tree-like Jacobi diagrams colored by~$B\oplus A$   with at least one trivalent vertex and with $\mathfrak{a}$-deg $=m$. For a Lagrangian cobordism~$M$ let $\widetilde{Z}^t(M)$ denote the reduction of $\widetilde{Z}(M)$ modulo \emph{looped} diagrams, that is, diagrams with a non-contractible connected component. See Figure \ref{figuraintro1artalt} $(c)$ for an example of a looped diagram. This way, $\widetilde{Z}^t(M)$ consists only of tree-like Jacobi diagrams. Finally, let $\widetilde{Z}^{Y,t}(M)$ denote the terms in $\widetilde{Z}^t(M)$ which have at least one trivalent vertex. The first step to relate the alternative Johnson homomorphisms with the LMO functor   is given in  Theorem \ref{JTHLMOprop1} where we prove the following.

\para{Theorem D} \emph{The alternative degree induces a filtration $\{\mathcal{F}^{\mathfrak{a}}_m\mathcal{C}\}_{m\geq1}$  of  $\mathcal{C}$ by submonoids. Consider the map
$$\widetilde{Z}^{Y,\mathfrak{a}}_m : \mathcal{F}^{\mathfrak{a}}_m\mathcal{C}\longrightarrow \mathcal{T}^{Y,\mathfrak{a}}_m(B\oplus A),$$
where, for  $M\in \mathcal{F}^{\mathfrak{a}}_m\mathcal{C}$, the element $\widetilde{Z}^{Y,\mathfrak{a}}_m(M)\in \mathcal{T}^{Y,\mathfrak{a}}_m(B\oplus A) $ is defined as the sum of the terms of~$\mathfrak{a}$-\emph{deg}~$= m$ in~$\widetilde{Z}^{Y,t}(M)$.  Then $\widetilde{Z}^{Y,\mathfrak{a}}_m$ is a monoid homomorphism.}

\medskip

In Theorem \ref{JTHLMOthm6.15} and Theorem \ref{JTHLMOthmvera2} we prove the following.

\para{Theorem E} \emph{Let $m\geq 1$ and $f\in J^{\mathfrak{a}}_m\mathcal{M}$. Then the $m$-th alternative Johnson homomorphism can be read in the tree-reduction of the LMO functor.}

\medskip

More precisely, we prove that for $h\in J_m^{\mathfrak{a}}\mathcal{M}$ with $m\geq 2$,  the value $\widetilde{Z}_m^{Y,\mathfrak{a}}(c(h))$ coincides (up to a sign) with  the diagrammatic version of $\tau_m^{\mathfrak{a}}(h)$. For $h\in J_1^{\mathfrak{a}}\mathcal{M}$, we show that $\tau_1^{\mathfrak{a}}(h)$ is given by $\widetilde{Z}_1^{Y,\mathfrak{a}}(c(h))$ together with the diagrams without trivalent vertices in $\widetilde{Z}(c(h))$ of $\mathfrak{a}$-deg$=1$. The techniques for the proof  of Theorem E in the case $m=1$ (Theorem \ref{JTHLMOthm6.15}) and $m\geq 2$ (Theorem \ref{JTHLMOthmvera2}) are different. For $m=1$ we need to do some explicit computations of the LMO functor and a comparison between the first alternative Johnson homomorphism and the first Johnson homomorphism. For $m\geq 2$, the key point is the fact that the LMO functor defines an \emph{alternative symplectic expansion} of $\pi$. To show this, we use  a result of Massuyeau \cite{MR2903772} where he proves that the LMO functor defines a symplectic expansion of $\pi$. 

Theorem D and Theorem E provide a new reading grid of the tree reduction of the LMO functor by the alternative degree. Theorem E follows the same spirit of a result of D. Cheptea, K. Habiro and G. Massuyeau in \cite{MR2403806} and of the author in \cite {vera1} where they prove that the Johnson homomorphisms and  the Johnson-Levine homormophisms, respectively, can be read in the tree-reduction of the LMO functor.

 Notice that Theorem D holds in the context of homology cobordisms, as do the results that we use to prove Theorem E. This suggests that the alternative Johnson homomorphisms and  Theorem E could be generalized to the setting of homology cobordisms, but we have not explored this issue so far.

The organization of the paper is as follows. In Section \ref{section2} we review the definition of several spaces of  Jacobi diagrams and some operations on them as well as some explicit computations.    Section \ref{section3} deals with the Kontsevich integral and the LMO functor, in particular we do some explicit computations that are needed in the following sections. Section \ref{JTH1} and  Section \ref{JTH2} provide a detailed exposition of the alternative Johnson filtration and the alternative Johnson homomorphisms, in particular we prove Theorem A, B and~C. Finally, Section \ref{section6} is devoted to the topological interpretation of the LMO functor  through the alternative Johnson homomorphisms, in particular we prove Theorem D and Theorem E.

\para{Reading guide} The reader more interested in the mapping class group could skip Section~\ref{section2} and Section~\ref{section3} and go directly to Section~\ref{JTH1} and Section~\ref{JTH2} (skipping subsection~\ref{subsection5.4}) referring to the previous sections only when needed. The reader familiar with  the LMO functor and more interested in the topological interpretation of its  tree reduction  through the alternative Johnson homomorphisms  can go directly to Section~\ref{section3}. Then go to subsection~\ref{subsection4.3} and  subsection~\ref{sub5.2} to the necessary definitions  to read Section~\ref{section6}.

\para{Notations and conventions} All subscripts appearing in this work are non-negative integers. When we write $m\geq 0$ or $m\geq 1$ we always mean that $m$ is an integer.  We use the blackboard framing convention on all drawings of knotted objects. We usually abbreviate  simple closed curve as scc. By a Dehn twist we mean a left-handed Dehn twist.

\para{Acknowledgements} I am deeply grateful to my advisor Gw\'ena\"el Massuyeau for his encouragement, helpful advice and careful reading. I thank sincerely  Takuya Sakasai for helpful and stimulating discussions, in particular for explaining to me Remark \ref{rksak}. I would like to thank Jean-Baptiste Meilhan and Jun Murakami for their comments on a previous version of this paper. Thanks are also due to the referee for their careful reading and comments.

\section{Spaces of Jacobi diagrams and their operations}\label{section2}

In this section we review several spaces of diagrams which are the target spaces of the Kontsevich integral, LMO functor and Jonhson-type homomorphisms. We refer to \cite{MR1318886,MR1881401} for a detailed discussion on the subject.  Throughout this section let $X$ denote a compact oriented $1$-manifold (possibly empty) whose connected components are ordered and let $C$ denote a finite  set (possibly empty).

\subsection{Generalities} A \emph{vertex-oriented unitrivalent graph} is a finite graph whose vertices are univalent (\emph{legs}) or trivalent, and such that for each trivalent vertex the set of half-edges incident to it is cyclically ordered.

 A \emph{Jacobi diagram} on $(X,C)$ is a vertex-oriented unitrivalent graph whose legs are either embedded in the interior of $X$ or are colored by the $\mathbb{Q}$-vector space generated by $C$. Two Jacobi diagrams are considered to be the same if there is an orientation-preserving homeomorphism between them respecting the order of the connected components, the vertex orientation of the trivalent vertices and the colorings of the legs. For drawings of Jacobi diagrams we use solid lines to represent $X$, dashed lines to represent the unitrivalent graph and we assume that the orientation of trivalent vertices is counterclockwise. See Figure \ref{figuraJD1artalt} for some examples.

\begin{figure}[ht!]
										\centering
                        \includegraphics[scale=0.8]{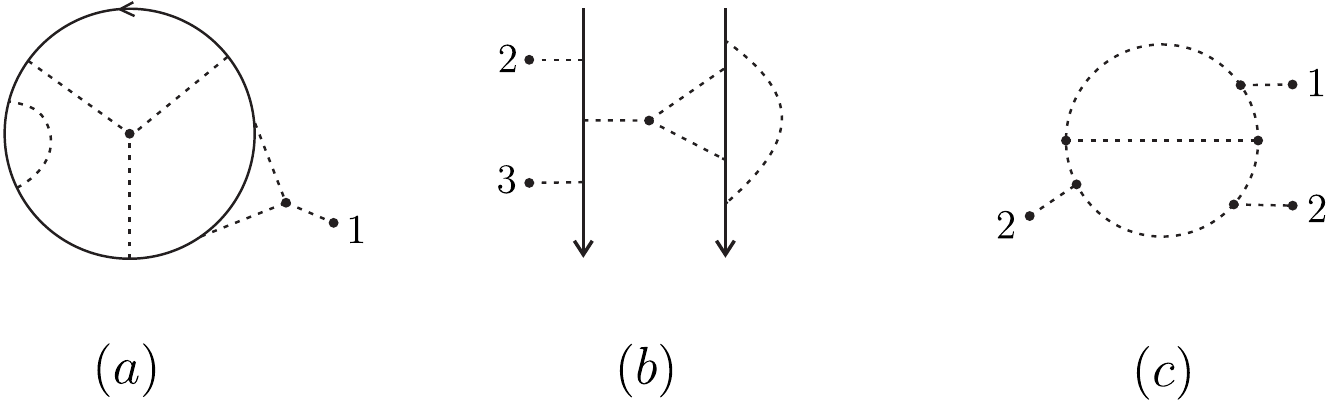}
												\caption{Jacobi diagrams with $X=\ \circlearrowleft$ in $(a)$, $X=\ \downarrow\ \downarrow$ in $(b)$ and $X=\varnothing$ in $(c)$. Here $C=\{1,2,3\}$.}
\label{figuraJD1artalt} 										
\end{figure}

The \emph{space of Jacobi diagrams on} $(X,C)$ is the $\mathbb{Q}$-vector space:

\[\mathcal{A}(X,C)=\frac{\text{Vect}_{\mathbb{Q}}\{\text{Jacobi diagrams on } (X,C)\}}{\text{STU, AS, IHX, $\mathbb{Q}$-multilinearity}},\]

\medskip

\noindent where the relations STU, AS, IHX  are local and the multilinearity relation applies to the colored legs. See Figure \ref{figuraJD2}. 

\begin{figure}[ht!]
										\centering
                        \includegraphics[scale=0.8]{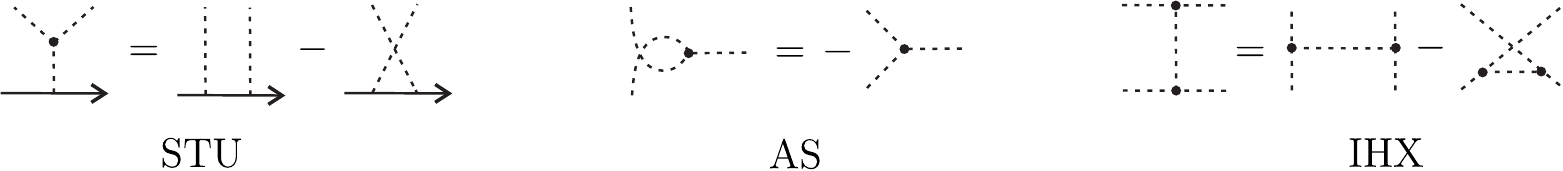}
												
\caption{Relations on Jacobi diagrams.}
\label{figuraJD2} 												
\end{figure}

If $X$ is not empty, it is well known that, for diagrams $D\in \mathcal{A}(X,C)$ such that every connected component of $D$ has at least one leg attached to $X$,  the STU relation implies the AS and IHX relations, see \cite[Theorem 6]{MR1318886}. We  can also define the space $\mathcal{A}(X,G)$ for any finitely generated  free abelian group $G$ as $\mathcal{A}(X,G)=\mathcal{A}(X,C)$, where $C$ is any finite set of free generators of $G$. If $X$ or $C$ is empty we drop it  from the notation. For  $D\in\mathcal{A}(X,C)$ we define the \emph{internal degree}, the \emph{external degree} and \emph{total degree}; denoted $\text{i-deg}(D)$, $\text{e-deg}(D)$ and $\text{deg}(D)$ respectively, as
\begin{align*}
\text{i-deg}(D):=& \text{ number of trivalent vertices of }D,\\
\text{e-deg}(D):=&\text{ number of legs of } D,\\
\text{deg}(D):= &\  \frac{1}{2}(\text{i-deg}(D)  + \text{e-deg}(D)).
\end{align*}
\noindent This way, the space  $\mathcal{A}(X,C)$ is graded with the total degree.  We still denote by $\mathcal{A}(X,C)$ its degree completion.

\begin{example}\label{ejemplo0JD} A connected Jacobi diagram in $\mathcal{A}(C)$ without trivalent vertices is called a \emph{strut}. See Figure \ref{figuraJD3artalt} $(a)$. For a matrix $\Lambda=(l_{ij})$ with entries indexed by a finite set~$C$, we define the element $\left[\Lambda\right]$ in $\mathcal{A}(C)$ by
\medskip

\centerline{\includegraphics[scale=0.85]{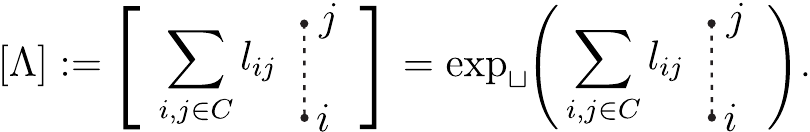}}
\end{example}

\begin{example}\label{ejemplo1JD} For a positive integer $n$, denote by $\left\lfloor n \right\rceil^*$ the set $\{1^*,\ldots, ,n^* \}$, where $*$ is one of the symbols  $+$, $-$ or $*$ itself. For instance the morphisms in the  target category of the LMO functor  are subspaces of the spaces $\mathcal{A}(\left\lfloor g\right\rceil^+\sqcup \left\lfloor f\right\rceil^-)$ for $g$ and $f$ positive integers. See Figure \ref{figuraJD3artalt} $(b)$.
\end{example}

\begin{example}\label{ejemplo2JD}  A Jacobi diagram in $\mathcal{A}(C)$ is \emph{looped} if it has a non-contractible component, see Figure \ref{figuraJD3artalt} $(b)$. The space of \emph{tree-like Jacobi diagrams} colored by $C$, denoted by $\mathcal{A}^t(C)$, is the quotient of  $\mathcal{A}(C)$ by the subspace generated by looped diagrams. The space of \emph{connected tree-like Jacobi diagrams} colored by $C$, denoted by $\mathcal{A}^{t,c}(C)$, is the subspace of $\mathcal{A}^t(C)$ spanned by connected Jacobi diagrams in $\mathcal{A}^t(C)$. For instance the spaces $\mathcal{A}^{t,c}(G)$, for $G$ some particular abelian groups, are the target of the diagrammatic versions of the Johnson-type homomorphisms. See Figure \ref{figuraJD3artalt} $(c)$ for an example of a  connected tree-like Jacobi digram.
\end{example}

\begin{figure}[ht!]
										\centering
                        \includegraphics[scale=0.8]{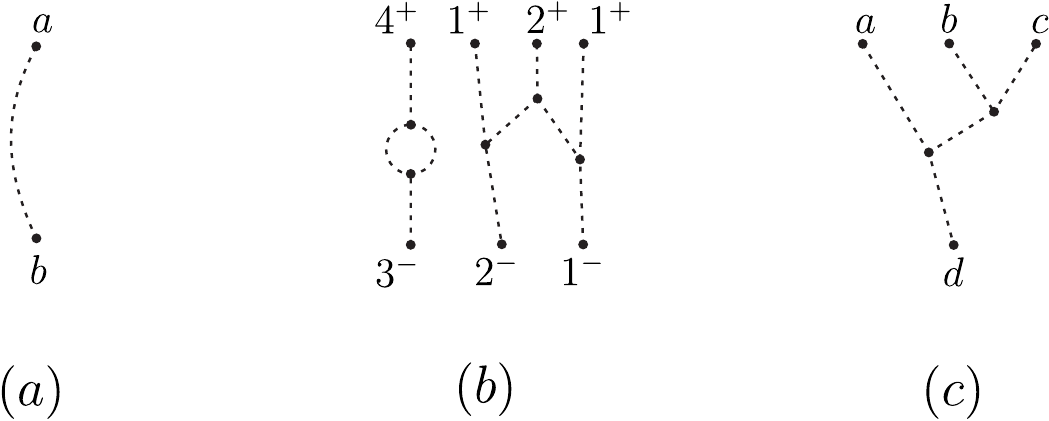}
												
\caption{$(a)$ Strut, $(b)$ Jacobi diagram in $\mathcal{A}(\left\lfloor 4\right\rceil^+\sqcup \left\lfloor 3\right\rceil^-)$,  $(c)$ Tree-like Jacobi diagram. Here $a,b,c,d\in C$ where $C$ is any finite set.}
\label{figuraJD3artalt} 										
\end{figure}

\subsection{Operations on Jacobi diagrams}\label{operationsonJD}
Let us recall some  operations on the spaces of Jacobi diagrams.

\para{Hopf algebra structure} There is a product in $\mathcal{A}(C)$ given by disjoint union, with unit the empty diagram, and a coproduct defined by $\Delta(D)=\sum D'\otimes D''$ where the sum ranges over pairs of  subdiagrams $D',D''$ of $D$ such that $D'\sqcup D''=D$. For instance:

\bigskip

\centerline{\includegraphics[scale=0.8]{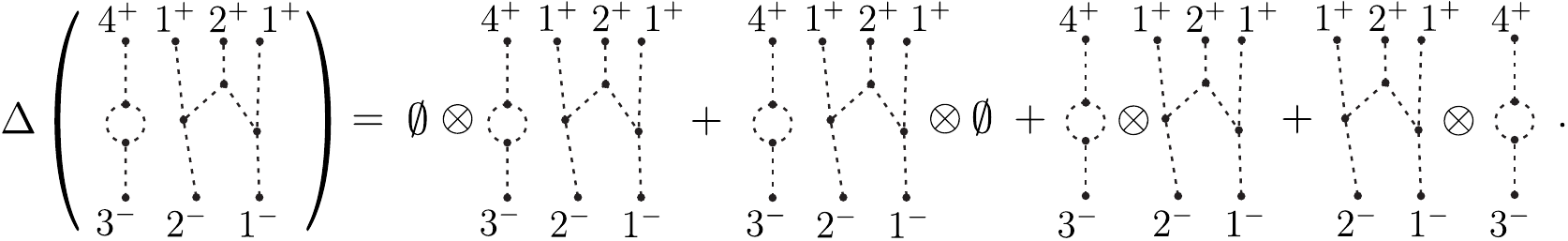}}

\medskip

\noindent With these structures $\mathcal{A}(C)$ is a  co-commutative Hopf algebra with counit the linear map $\epsilon:\mathcal{A}(C)\rightarrow \mathbb{Q}$ defined by $\epsilon(\varnothing)=1$ and $\epsilon(D)=0$ for $D\in\mathcal{A}(C)\setminus\{\varnothing\}$ and with antipode the linear map $S:\mathcal{A}(C)\rightarrow\mathcal{A}(C)$ defined by $S(D)=(-1)^{k_D} D$ where $k_D$ denotes the number of connected components of $D\in\mathcal{A}(C)$. It follows from the definition of the coproduct that the primitive part of $\mathcal{A}(C)$ is the subspace $\mathcal{A}^c(C)$ spanned by connected Jacobi diagrams.

\para{Doubling and orientation-reversal operations} Suppose  that we can decompose the $1$-manifold~$X$ as $X=\downarrow\sqcup X'=\downarrow X'$, here $X'$ can be empty. Then given a Jacobi diagram $D$ on $\downarrow X'$ it is possible to obtain new Jacobi diagrams $\Delta(D)$  on $\downarrow\sqcup X = \downarrow \downarrow X'$ and $S(D)$ on $\uparrow X'$. Let us represent the Jacobi diagram $D$ as 

\centerline{\includegraphics[scale=0.8]{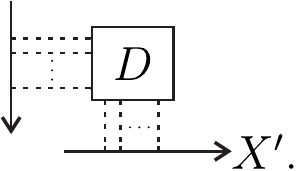}}

\noindent Then $\Delta(D)$ is defined in Figure \ref{figuraJD6artalt}, where we use the \emph{box notation} to denote the sum over all the possible ways of gluing the legs of $D$ attached to the grey box to the two  intervals involved in the grey box, in particular if there are $k$ legs attached to the grey box, there will be $2^k$ terms in the sum.
\begin{figure}[ht!]
										\centering
                        \includegraphics[scale=0.8]{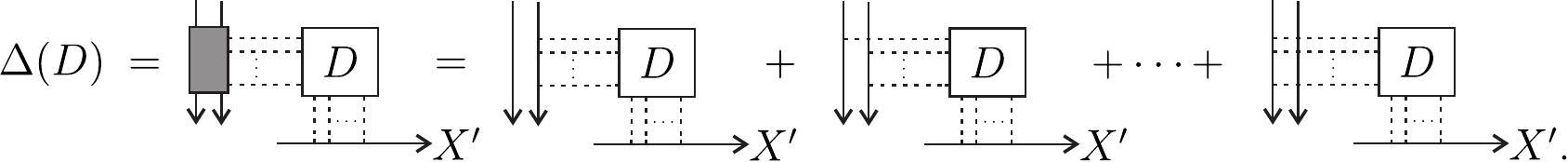}
												
\caption{Definition of the doubling map and box notation.}
\label{figuraJD6artalt} 										
\end{figure}

Besides, the Jacobi digram $S(D)$ is given in Figure \ref{figuraJD7artalt}. 
\begin{figure}[ht!]
										\centering
                        \includegraphics[scale=0.8]{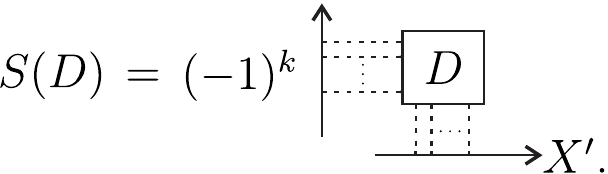}
												
\caption{Definition of orientation-reversal map. Here we suppose that there are $k$ legs attached to the chosen interval.}
\label{figuraJD7artalt} 										
\end{figure}

To sum up, we have maps
\begin{equation}\Delta:\mathcal{A}(\downarrow X')\longrightarrow \mathcal{A}(\downarrow\downarrow X')\text{\ \ \ \  and \ \ \ \ } S:\mathcal{A}(\downarrow X')\longrightarrow \mathcal{A}(\uparrow  X'), 
\end{equation}   
called \emph{doubling map} and \emph{orientation reversal map}, respectively. Observe that even if we use the same notation for the doubling map and the coproduct, the respective meaning can be deduced from the context.

\para{Symmetrization map} Let us recall the diagrammatic version of the Poincar\'e-Birkhoff-Witt isomorphism. We follow \cite{MR1318886,MR2962302} in our exposition. Let $D$ be a Jacobi diagram on $(X,C\sqcup \{s\})$. We could glue all the $s$-colored legs of $D$ to an interval $\uparrow_s$ (labelled by $s$) in order to obtain a Jacobi diagram on $(X\uparrow_s,C)$, \emph{i.e.}  there would not be any $s$-colored leg left. But there are many ways of doing this gluing, so we consider the arithmetic mean of all the possible ways of gluing the $s$-colored legs of $D$ to the interval $\uparrow_s$. This way we obtain a well defined vector space isomorphism 
\begin{equation}\label{symmap}
\chi_s:\mathcal{A}(X,C\sqcup\{s\})\longrightarrow \mathcal{A}(X\uparrow_s, C),
\end{equation}
called \emph{symmetrization map}. It is not difficult to show that the map (\ref{symmap}) is well defined, but it is more laborious to show that it is bijective, see \cite[Theorem 8]{MR1318886}. If $S=\{s_1,\ldots,s_l\}$, it is possible to define, in a similar way, a vector space isomorphism
\begin{equation*}
\chi_S:\mathcal{A}(X,C\sqcup S)\longrightarrow \mathcal{A}(X\uparrow_S, C),
\end{equation*}
where $\uparrow_S=\uparrow_{s_1}\cdots\uparrow_{s_l}$. More precisely, $\chi_S = \chi_{s_l}\circ\cdots\circ \chi_{s_1}$. 

\begin{example}\label{ejemploJDhomotopy} Fix $r\in S$. Denote by $\mathcal{H}(r)$ the subspace of $\mathcal{A}(S)$ generated by Jacobi diagrams with at least one component that is looped or that possesses at least two $r$-colored legs. Similarly, denote 
by $\mathcal{H}(\uparrow_r)$ the subspace of $\mathcal{A}(\uparrow_S)$ generated by Jacobi diagrams with at least one dashed component that is looped or that possesses at least two legs attached to $\uparrow_r$. Bar-Natan shows in \cite[Theorem 1]{MR1321289} that $\chi(\mathcal{H}(r))=\mathcal{H}(\uparrow_r)$.
\end{example}

The inverse of the symmetrization map is constructed recursively. Since we will use this inverse, let us review the definition. Let $D$ be a Jacobi diagram on $(X\uparrow_s,C)$ with~$n$ legs attached to~$\uparrow_s$. Label these legs from $1$ to $n$  following the orientation of $\uparrow_s$. For a permutation $\varsigma\in S_n$, there is a way of obtaining a Jacobi digram $\varsigma D$ on $(X\uparrow_s,C)$ by acting on the legs. For instance if  $\varsigma=(123)$ we have:

\bigskip

\centerline{\includegraphics[scale=0.8]{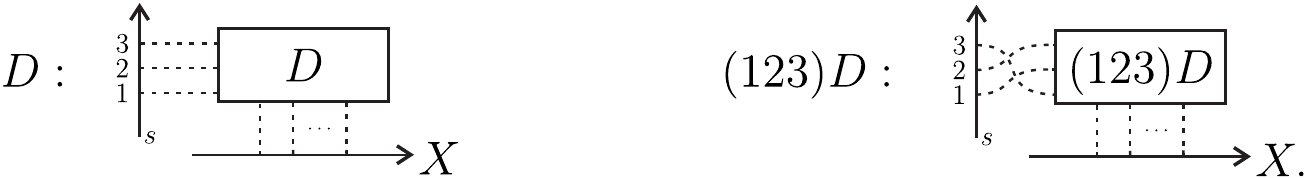}}

\medskip

\begin{theorem}\cite[Theorem 8]{MR1318886}  Let $n\geq 1$ and let $D$ be a Jacobi diagram on $(X\uparrow_s,C)$ with at most $l\leq n$ legs attached to $\uparrow_s$. Denote by $\tilde{D}$ the Jacobi diagram on $(X,C\sqcup\{s\})$ obtained from $D$ by erasing $\uparrow_s$ and coloring with $s$ all the legs that were attached to $\uparrow_s$. Set $\sigma_1(D)=\tilde{D}$ and for $n> 1$
\begin{equation*}
\sigma_n(D)=
\begin{cases}
\tilde{D}+\frac{1}{n!}\sum_{\varsigma\in S_n}\sigma_{n-1}(D-\varsigma D), \text{\ \ if\ \ }l=n, \\
\ \\
\sigma_{n-1}(D), \text{\ \ if\ \ } l<n.
\end{cases}
\end{equation*} 
Then the map 
\[\sigma:\mathcal{A}(X\uparrow_s,C)\longrightarrow \mathcal{A}(X,C\sqcup\{s\}),\]
defined by $\sigma(D)=\sigma_n(D)$ is well-defined and it is the inverse of the symmetrization map.
\end{theorem}

\begin{example}\label{exampleJD2.6}
\begin{align*}
\includegraphics[scale=1]{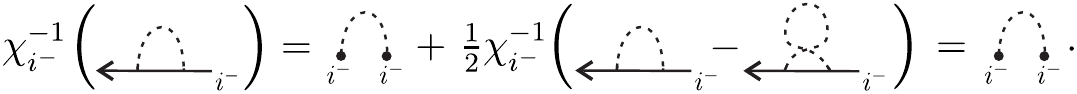}
\end{align*}
\end{example}

\begin{example}\label{exampleJD2.7}
\begin{align*}
\includegraphics[scale=1]{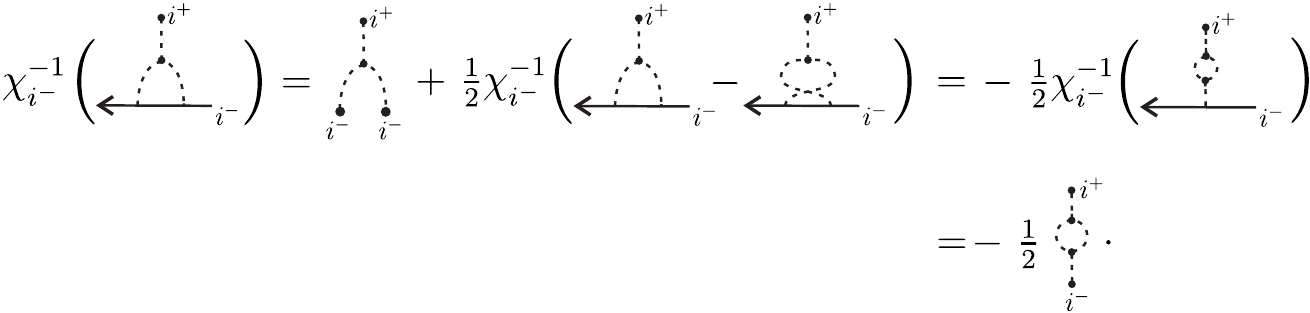}
\end{align*}
\end{example}

\begin{example}\label{exampleJD2.8}
\begin{align*}
\includegraphics[scale=1]{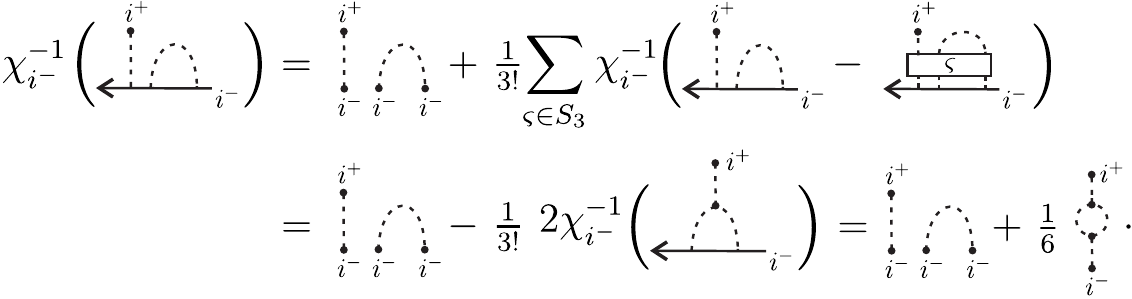}
\end{align*}
In the last equality we used Example \ref{exampleJD2.7}.
\end{example}

\begin{example}\label{exampleJD2.9} We are usually interested in the reduction modulo looped diagrams. We use the symbol $\equiv$ to indicate an equality modulo looped diagrams. Using the previous examples, it is possible to show 
\begin{align*}
\includegraphics[scale=1]{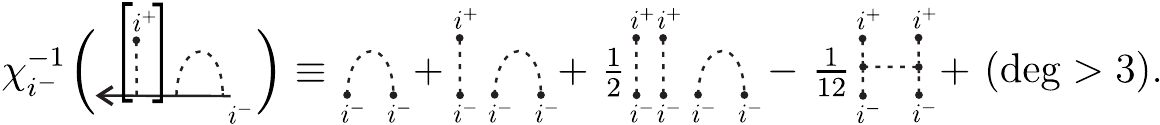}
\end{align*}
Here the square brackets stand for an exponential, more precisely
\begin{align*}
\includegraphics[scale=1]{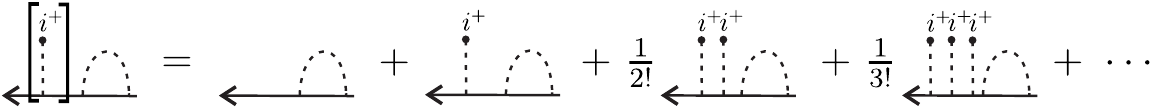}
\end{align*}
\end{example}

\section{The Kontsevich integral and the LMO functor}\label{section3}
In this section we review the combinatorial definition of the Kontsevich integral from~\cite{MR1881401,lenotes}. We also recall the construction of the LMO functor following \cite{MR2403806}. We focus on particular examples, which will play an important role in the next sections,  rather that in a detailed exposition on the subject. 

\subsection{Kontsevich Integral}\label{KI}
Let us start by recalling some basic notions. Consider the cube $[-1,1]^3\subseteq\mathbb{R}^3$ with coordinates $(x,y,z)$. A \emph{framed tangle} in $[-1,1]^3$ is a compact oriented framed $1$-manifold $T$ properly embedded in  $[-1,1]^3$ such that the boundary~$\partial T$ (the \emph{endpoints} of $T$) is uniformly distributed along $\{0\}\times [-1,1]\times\{\pm 1\}$ and the framing on the endpoints of $T$ is the vector $(0,1,0)$. We draw diagrams of framed tangles using the blackboard framing convention.
Let $T$  be a framed tangle. Denote by~$\partial_t T$ the endpoints of $T$ lying in $\{0\}\times [-1,1]\times\{+1\}$, we call $\partial_t T$ the \emph{top boundary} of $T$. Similarly, $\partial_bT$ of $T$ denotes  the \emph{bottom boundary}. 

We can associate words $w_t(T)$ and $w_b(T)$ on $\{+,-\}$ to $\partial_tT$ and $\partial_bT$ as follows. To an endpoint of $T$ we associate  $+$  if the orientation of $T$ goes downwards at that endpoint,  and $-$ if the orientation of $T$ goes upwards at that endpoint. The words $w_t(T)$ and $w_b(T)$ are obtained by reading the corresponding signs in the positive direction of the~$y$ coordinate. See Figure (\ref{figuraKI2artalt}) $(a)$ for an example of a tangle with its corresponding words.

 We consider \emph{non-associative words} on $\{+,-\}$, that is, words on $\{+,-\}$ together with a parenthesization (formally an element of the \emph{free magma} generated by $\{+,-\}$). For instance $((+-)+)$ and $(+(-+))$ are the two possible non-associative words obtained from the word $+-+$.  From now on we omit the outer parentheses. A $q$\emph{-tangle} is a framed tangle whose top and bottom words are endowed with a parenthesization. See Figure (\ref{figuraKI2artalt}) $(b)$ and $(c)$ for two different parenthesizations of the same  framed tangle.

\begin{figure}[ht!]
										\centering
                        \includegraphics[scale=0.9]{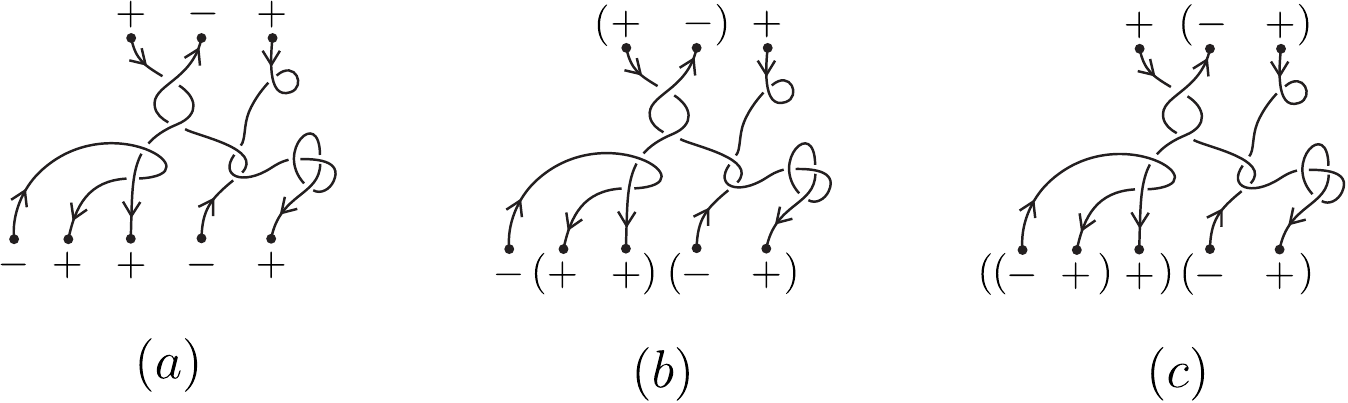}
												
\caption{A framed tangle and two different $q$-tangles obtained from it.}
\label{figuraKI2artalt} 								
\end{figure}

To define the Kontsevich integral it is  necessary to fix a particular element $\Phi\in\mathcal{A}(\downarrow\downarrow\downarrow)$ called an \emph{associator}. The element $\Phi$ is an exponential series of Jacobi diagrams satisfying several conditions, among these, one ``pentagon" and two ``hexagon" equations; see \cite[(6.11)--(6.13)]{MR1881401}.  From now on we fix an \emph{even Drinfeld associator} $\Phi$, see \cite[Corollary 4.2]{MR1669949} for the definition and existence. In low degree we have:
\begin{align*}
\includegraphics[scale=1]{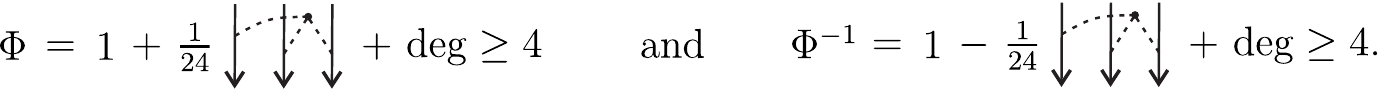}
\end{align*}

\noindent Here $1$ means $\downarrow\downarrow\downarrow$. The Kontsevich integral is defined so that:
\begin{equation}
\begin{aligned}\label{functorialityZ}
Z(T_1\circ T_2)&= Z(T_1)\circ Z(T_2), \\
Z(T_1\otimes T_2)&=Z(T_1)\otimes Z(T_2);
\end{aligned}
\end{equation}
where the composites $T_1\circ T_2$ and $Z(T_1)\circ Z(T_2)$ and the tensor products $T_1\otimes T_2$ and  $Z(T_1)\otimes Z(T_2)$ are defined by vertical and horizontal juxtaposition of $q$-tangles and Jacobi diagrams, respectively. Now every $q$-tangle can be expressed as the composition of tensor products of some \emph{elementary $q$-tangles}, so it is enough to define the Kontsevich integral on these $q$-tangles. Set
\begin{align*}
\includegraphics[scale=1]{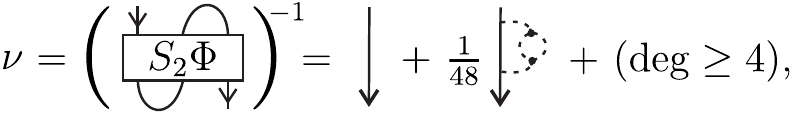}
\end{align*}
where $S_2$ is the orientation-reversal map applied to the second interval. The Kontsevich integral is defined on the elementary $q$-tangles as follows:
\begin{align*}
\includegraphics[scale=1]{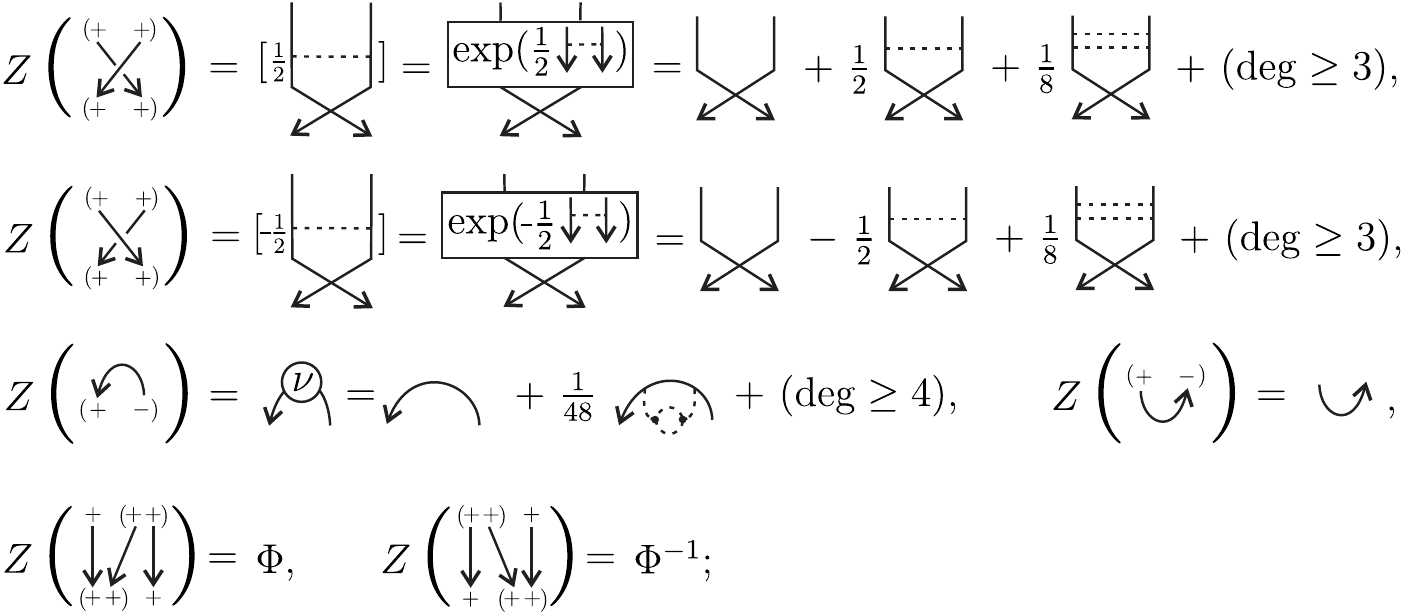}
\end{align*}
and for elementary $q$-tangles of the form
\begin{align*}
\includegraphics[scale=1]{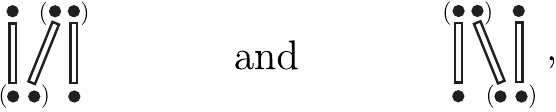}
\end{align*}
\noindent where the thick lines represent a trivial tangle and the black dots some non-associative words on $\{+,-\}$, the Kontsevich integral is defined by using the doubling and orientation reversal maps, see subsection \ref{operationsonJD}, for instance 
\begin{align*}
\includegraphics[scale=1]{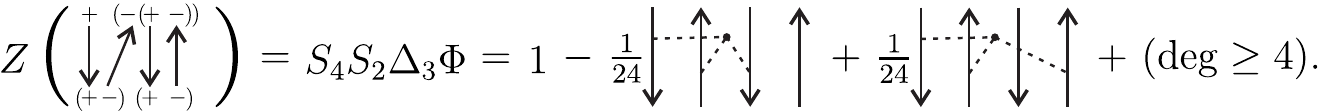}
\end{align*}
Here the subscripts indicate the interval to which the operation is applied. It is known that $Z$ is well defined and is an isotopy invariant of $q$-tangles, see \cite{MR1328252,MR1394520}. For a $q$-tangle~$T$, we denote by $Z^t(T)$ the reduction of $Z(T)$ modulo looped diagrams, see Example \ref{ejemplo2JD}.

\begin{example}\label{ejemplo1KI} 
\begin{align*}
\includegraphics[scale=1]{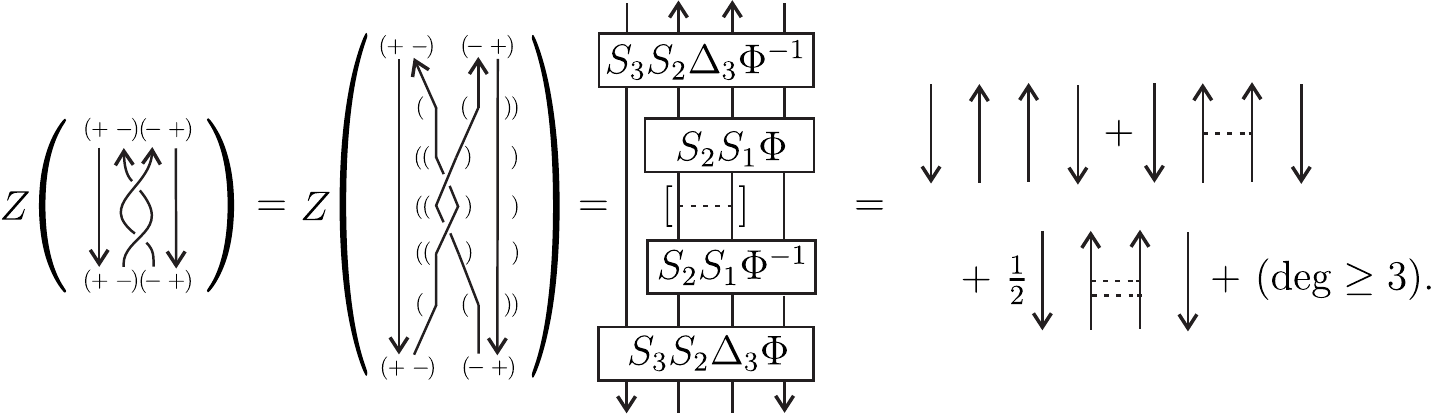}
\end{align*}
\end{example}

\begin{example}\label{ejemplo2KI} 
\begin{align*}
\includegraphics[scale=1]{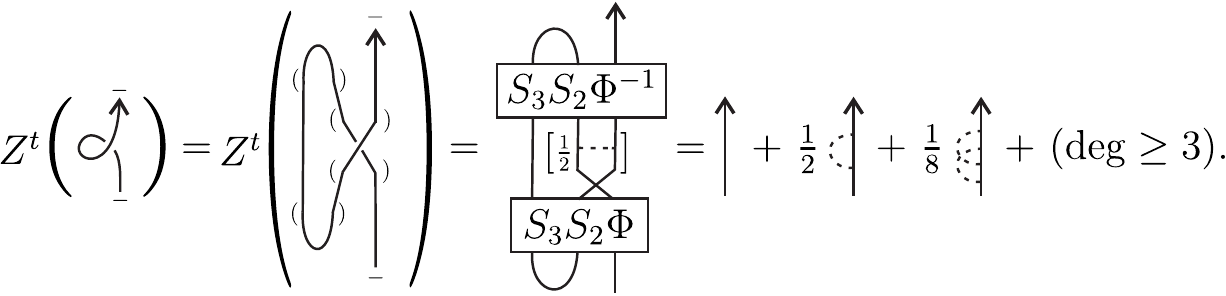}
\end{align*}
\end{example}

\begin{example}\label{ejemplo3KI} Using Examples \ref{ejemplo1KI}, \ref{ejemplo2KI} and Equations (\ref{functorialityZ}) we have
\begin{align*}
\includegraphics[scale=1]{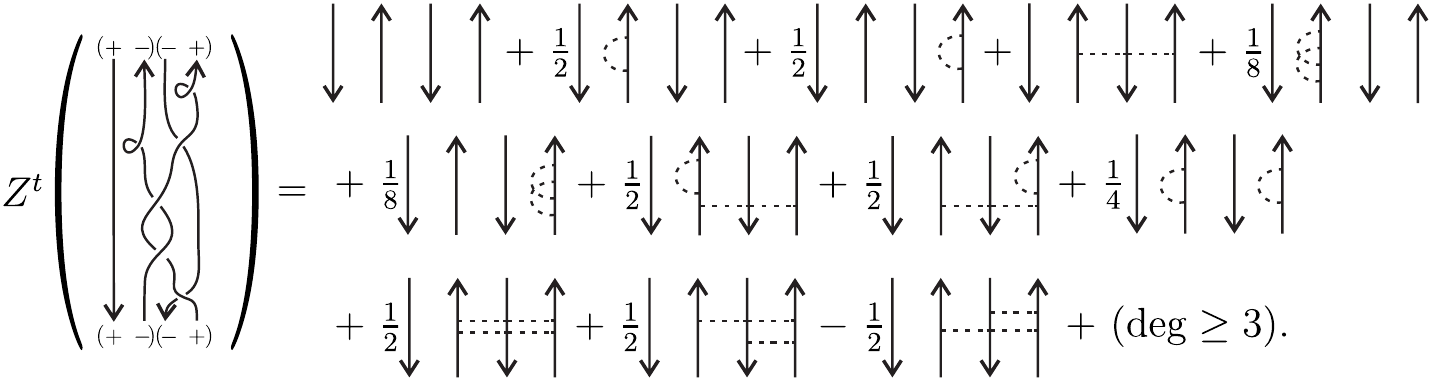}
\end{align*}
\end{example}

\begin{example}\label{ejemplo4KI} Recall the space $\mathcal{H}(\uparrow_r)$ defined in Example \ref{ejemploJDhomotopy}. We have
\begin{align*}
\includegraphics[scale=1]{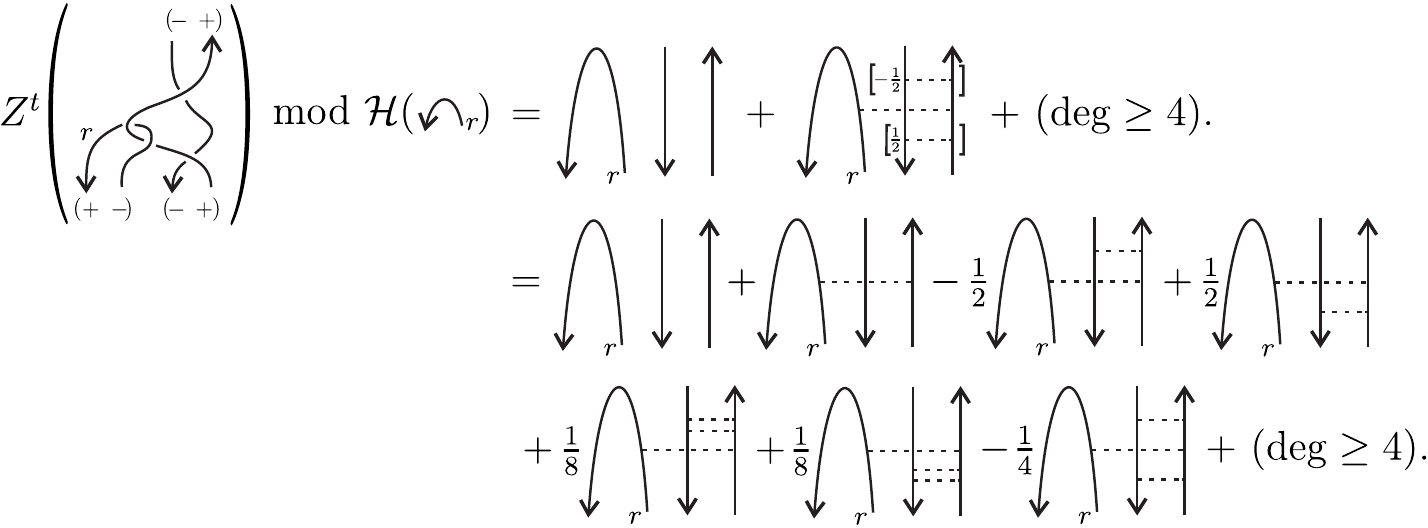}
\end{align*}
\end{example}

\medskip

\subsection{The LMO functor}\label{subsectionLMOfunctor}
This subsection is devoted to a brief description of the LMO functor $\widetilde{Z}:\mathcal{LC}ob_q\rightarrow {}^{ts}\!\!\mathcal{A}$ and principally to explicit computations which will be useful in the following sections. We refer to \cite{MR2403806} for more details. Throughout this subsection we denote by  $\Sigma_{g,1}$  a compact connected oriented surface of genus $g$ with one boundary component for each non-negative integer $g$, see Figure \ref{figuraLMO2artalt}. 

\para{Homology cobordisms and their bottom-top tangle presentation} Let us start with some preliminaries.
A \emph{homology cobordism} of $\Sigma_{g,1}$ is the equivalence class of a pair $M=(M,m)$, where $M$ is a compact connected oriented $3$-manifold and $m:\partial(\Sigma_{g,1}\times[-1,1])\rightarrow\partial M$ is an orientation-preserving homeomorphism, such that the \emph{bottom} and \emph{top} inclusions
$m_{\pm}(\cdot):=m(\cdot,\pm1):\Sigma_{g,1}\rightarrow M$ induce isomorphisms in homology. Two pairs $(M,m)$ and $(M',m')$ are \emph{equivalent} if there exists an orientation-preserving homeomorphism $\varphi:M\rightarrow M'$ such that $\varphi\circ m=m'$. The \emph{composition} $(M,m)\circ (M',m')$ of two homology cobordisms $(M,m)$ and $(M',m')$ of $\Sigma_{g,1}$ is the equivalence class of the pair $(\widetilde{M},m_-\cup m'_+)$, where $\widetilde{M}$ is obtained by gluing the two $3$-manifolds $M$ and $M'$ by using the map $m_+\circ(m'_-)^{-1}$. This composition is associative and has as identity element the equivalence class of the trivial cobordism $( \Sigma_{g,1}\times[-1,1],  \text{Id})$.  Denote by $\mathcal{C}_{g,1}$  the \emph{monoid of homology cobordisms} of $\Sigma_{g,1}$. This notion plays an important role in the theory of finite-type invariants as shown independently by M. Goussarov in~\cite{MR1715131} and K. Habiro in \cite{MR1735632}.

\begin{example}\label{ejemploLMO1} Denote by $\mathcal{M}_{g,1}$ the \emph{mapping class group} of  $\Sigma_{g,1}$, \emph{i.e.} the group of isotopy classes of orientation-preserving homeomorphisms of $\Sigma_{g,1}$ that fix the boundary $\partial\Sigma_{g,1}$ pointwise. This group can be embedded into $\mathcal{C}_{g,1}$ by associating to any $h\in\mathcal{M}_{g,1}$ the homology cobordism, called \emph{mapping cylinder}, $c(h)=(\Sigma_{g,1}\times[-1,1], m^h)$, where $m^h:\partial(\Sigma_{g,1}\times[-1,1])\rightarrow\partial (\Sigma_{g,1}\times[-1,1])$ is the orientation-preserving homeomorphism defined by $m^h(x,1)=(h(x),1)$ and  $m^h(x,t)=(x,t)$ for $t\not=1$. This way we have an injective map $c:\mathcal{M}_{g,1}\rightarrow \mathcal{C}_{g,1}$. The submonoid $c(\mathcal{M}_{g,1})$ is precisely the group of  invertible elements of $\mathcal{C}_{g,1}$, see \cite[Proposition 2.4]{MR2952770}. 
\end{example}

There is a more general notion of cobordism. For $g,f\geq 0$ let $C^g_f$ denote the compact oriented $3$-manifold obtained from $[-1,1]^3$ by adding $g$ (respectively $f$) $1$-handles  along $[-1,1]\times [-1,1]\times \{+1\}$ (respectively along $[-1,1]\times [-1,1]\times \{-1\}$), uniformly in the~$y$ direction. A \emph{cobordism} from $\Sigma_{g,1}$ to $\Sigma_{f,1}$ is the homeomorphism class relative to the boundary of a pair $(M,m)$, where $M$ is a compact connected oriented $3$-manifold and $m:\partial C^g_f\rightarrow \partial M$  is an orientation-preserving homeomorphism.

Given  a homology cobordism $(M,m)$ of $\Sigma_{g,1}$; or more generally a cobordism from $\Sigma_{g,1}$ to $\Sigma_{f,1}$. We can associate a particular kind of tangle whose components split in $f$ bottom components and $g$ top components (they are called \emph{bottom-top tangles} in~\cite{MR2403806}). The association is defined as follows. First fix a system of meridians and parallels $\{\alpha_i,\beta_i\}$ on $\Sigma_{g,1}$ for each non-negative integer $g$  as shown in Figure \ref{figuraLMO2artalt}. 

\begin{figure}[ht!] 
										\centering
                        \includegraphics[scale=0.7]{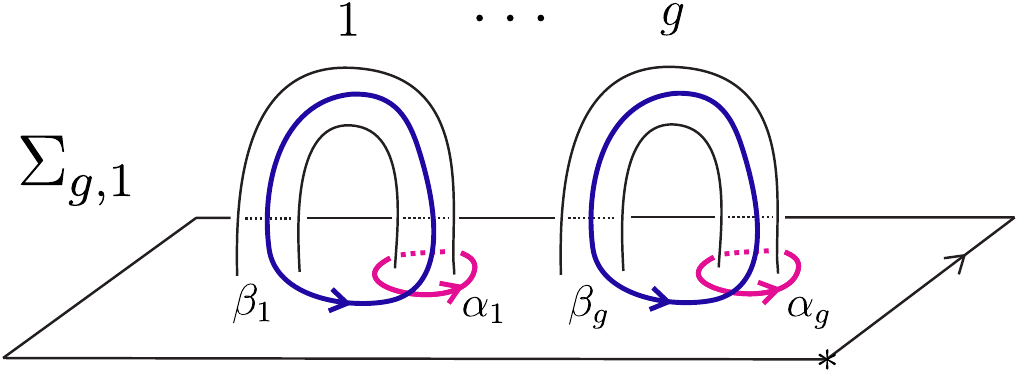}
												\caption{System of meridians and parallels $\{\alpha_i,\beta_i\}$ on $\Sigma_{g,1}$.}								\label{figuraLMO2artalt}
\end{figure}

\noindent Then  attach $g$  $2$-handles (or $f$ in the case of a cobordism from $\Sigma_{g,1}$ to $\Sigma_{f,1}$) on the bottom surface of $M$ by sending the cores of the $2$-handles to the curves $m_-(\alpha_i)$. In the same way,  attach $g$ $2$-handles on the top surface of $M$ by sending the cores to the curves $m_+(\beta_i)$. This way we obtain a compact connected oriented $3$-manifold $B$ and an orientation-preserving homeomorphism $b:\partial([-1,1]^3)\rightarrow\partial B$. The pair $B=(B,b)$ together with the cocores of the $2$-handles, determine a tangle $\gamma$ in $B$. We call the homeomorphism class relative to the boundary of the pair $(B,\gamma)$, still denoted in the same way,  the \emph{bottom-top tangle presentation} of $(M,m)$. Following the positive direction of the $y$ coordinate, we label the bottom components of $\gamma$ with $1^-$, $\ldots$, $f^-$ and the top components with $1^+$, $\ldots$, $g^+$, respectively. This procedure is sketched in Example \ref{ejemploLMO0}.
 
\begin{example}\label{ejemploLMO0} In Figure\ref{figuraLMO3artalt} we illustrate the procedure to obtain  the bottom-top tangle presentation of the trivial cobordism $\Sigma_{g,1}\times[-1,1]$.
\begin{figure}[ht!] 
										\centering
                        \includegraphics[scale=0.7]{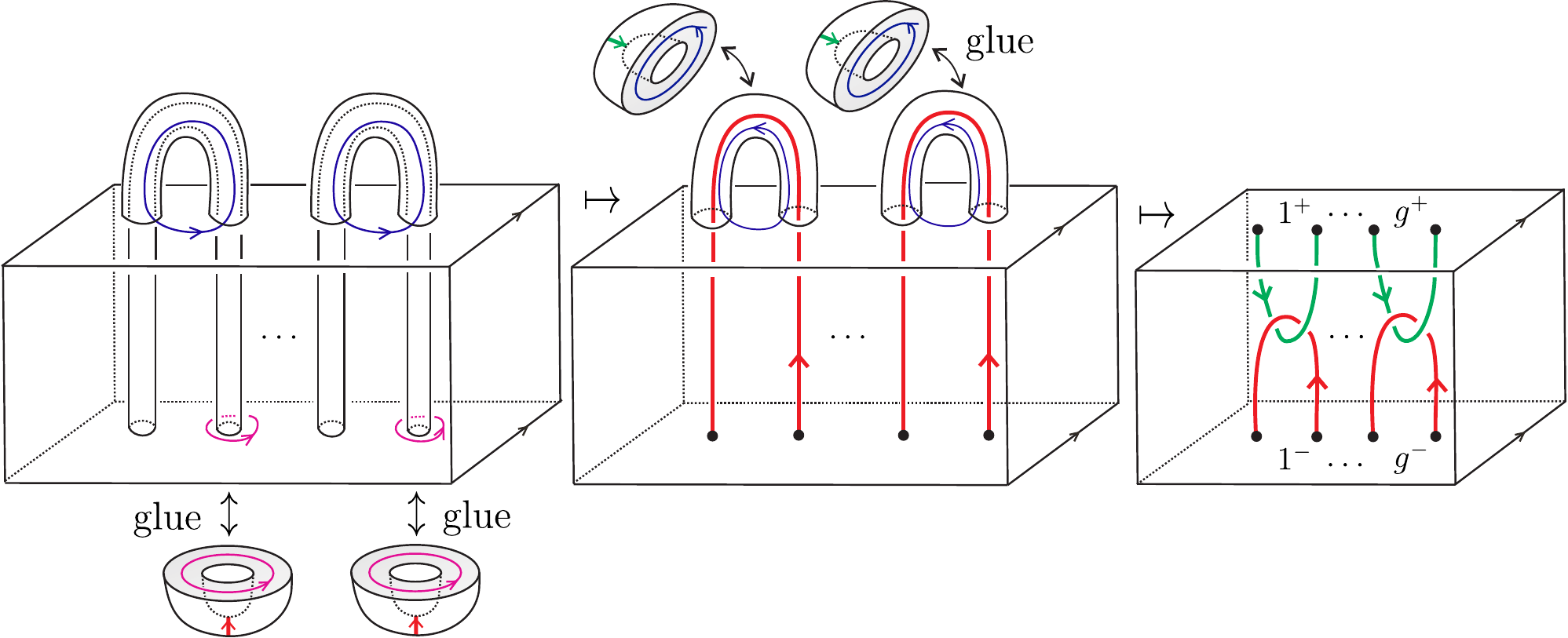}
												\caption{Obtaining the bottom-top tangle presentation of the trivial cobordism $\Sigma_{g,1}\times[-1,1]$.}
\label{figuraLMO3artalt}
\end{figure}
\end{example}

\para{Lagrangian Cobordisms} Let us now roughly describe the source category $\mathcal{LC}ob$ of the LMO functor.  For each non-negative integer $g$, let $H_g=H_1(\Sigma_{g,1};\mathbb{Z})$ be the first homology group of $\Sigma_{g,1}$ with integer coefficients, and $\omega:H_g\otimes H_g\rightarrow \mathbb{Z}$ the intersection form.  Denote by $A_g$ the subgroup of $H_g$ generated by the homology classes of the meridians $\{\alpha_i\}$. This is a Lagrangian subgroup of $H_g$ with respect to the intersection form. Let $V_g$ be a handlebody of genus $g$ obtained from $\Sigma_{g,1}$ by attaching $g$ $2$-handles by sending the cores of the $2$-handles  to the meridians $\alpha_i$'s, in particular the curves $\alpha_i$ bound pairwise disjoint disks in $V_g$. We also see $V_g$ as a cobordism from $\Sigma_{g,1}$ to $\Sigma_{0,1}$,  see Figure \ref{figuraLMO4artalt}. Thus we can also see $A_g$ as  $A_g=\text{ker}(H_g\rightarrow H_1(V_g;\mathbb{Z}))$.

\begin{figure}[ht!] 
\centering
\includegraphics[scale=0.75]{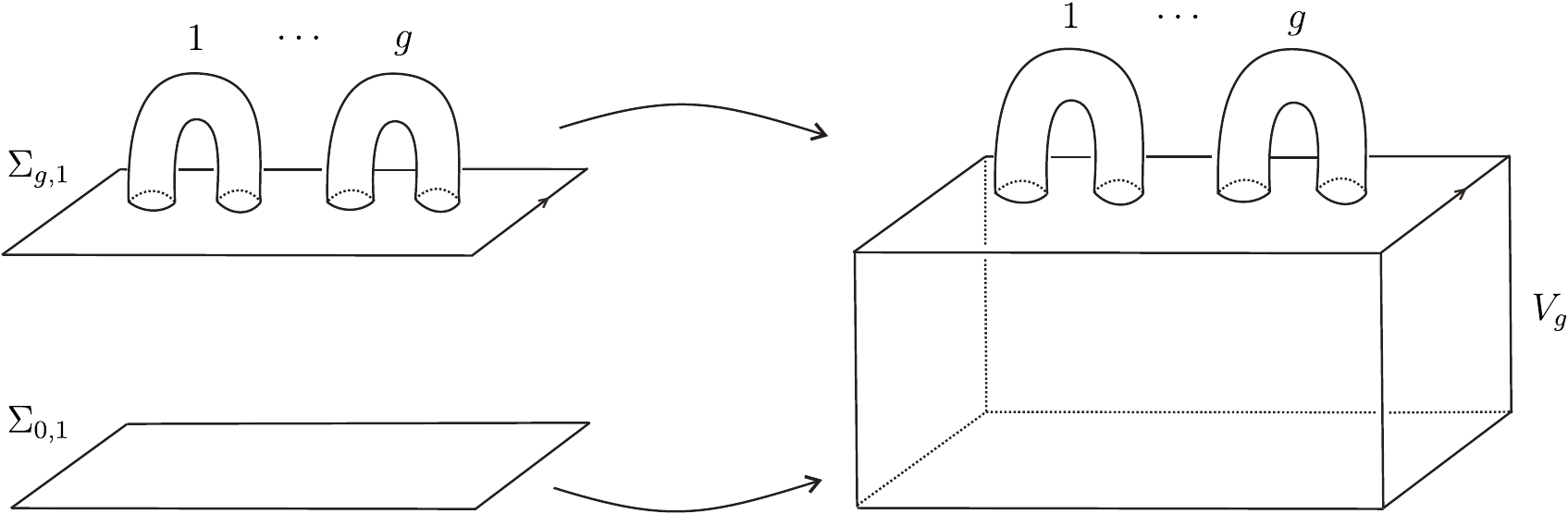}
\caption{Handlebody  $V_g$ as a cobordism from $\Sigma_{g,1}$ to $\Sigma_{0,1}$.}											\label{figuraLMO4artalt}
\end{figure}

\begin{definition}\cite[Definitions 2.4 and 2.6]{MR2403806}
A cobordism $(M,m)$ from $\Sigma_{g,1}$ to $\Sigma_{f,1}$ is said to be \emph{Lagrangian} if it satisfies:
\begin{itemize}
\item $H_1(M;\mathbb{Z})=m_{-,*}(A_f)+m_{+,*}(H_g)$,
\item $m_{+,*}(A_g)\subseteq m_{-,*}(A_f)$ in $H_1(M;\mathbb{Z})$.
\end{itemize}
Moreover, $(M,m)$ is said to be \emph{special Lagrangian} if it additionally satisfies $V_f\circ M=V_g$ as cobordisms.
\end{definition}

Let $M$ be a Lagrangian cobordism and $(B,\gamma)$ its bottom-top tangle presentation. It follows, from a Mayer-Vietoris argument, that $B$ is a \emph{homology cube}, \emph{i.e.} $B$ has the same homology groups as the standard cube $[-1,1]^3$, see \cite[Lemma 2.12]{MR2403806}. Notice that the definition of $q$-tangle in $[-1,1]^3$ given in subsection \ref{KI} extends naturally to $q$-tangles in  homology cubes.

Let us now define the category $\mathcal{LC}ob$. The objects of  $\mathcal{LC}ob$ are the non-negative integers and the set of morphisms $\mathcal{LC}ob(g,f)$ from $g$ to $f$ are Lagrangian cobordisms from $\Sigma_{g,1}$ to $\Sigma_{f,1}$. Denote by ${}^{s}\!\mathcal{LC}ob(g,f)$ the morphisms from $g$ to $f$ which are special Lagrangian.

\begin{example}\label{ejemploLMO2} Let $h\in\mathcal{M}_{g,1}$. Then the mapping cylinder $c(h)$ is Lagrangian if and only if $h(A_g)\subseteq A_g$. Moreover, $c(h)$ is special Lagrangian if and only if $h$ can be extended to a self-homeomorphism of the handlebody $V_g$.
\end{example}

 Let us consider some particular cases of the mapping cylinders described in Example~\ref{ejemploLMO2}. Let $\gamma$ be a simple closed curve on $\Sigma_{g,1}$ and denote by $t_{\gamma}$ the (left) Dehn twist along~$\gamma$. Recall that the mapping cylinder $c(t_{\gamma})$ can be obtained from the trivial cobordism $\Sigma_{g,1}\times[-1,1]$  by performing a surgery along a $(-1)$-framed knot in a neighbourhood of a push-off of the curve $\gamma$ in $\Sigma_{g,1}\times[-1,1]$, see for instance \cite[Lemma 8.5]{MR1881401}. In particular we can obtain the bottom-top tangle presentation of $c(t_{\gamma})$ from that of $\Sigma_{g,1}\times[-1,1]$, see Examples \ref{ejemploLMO3}, \ref{ejemploLMO4} and \ref{ejemploLMO4.5}.

\begin{example}\label{ejemploLMO3} Let $t_{\alpha_i}$ be the  Dehn twist along a meridian curve $\alpha_i$. Then $c(t_{\alpha_i})\in{}^s\!\mathcal{LC}ob(g,g)$.  Figure \ref{figuraLMO5artalt} $(a)$ shows the bottom-top tangle presentation of the trivial cobordism $\Sigma_{g,1}\times[-1,1]$ (in thin line) together with a $(-1)$-framed knot (in thick line)  such that the surgery along this knot gives the bottom-top tangle presentation of $c(t_{\alpha_i})$ showed in Figure \ref{figuraLMO5artalt} $(b)$. Notice that going from Figure \ref{figuraLMO5artalt} $(a)$   to  Figure \ref{figuraLMO5artalt} $(b)$ is exactly a Fenn-Rourke move.

\begin{figure}[ht!]
										\centering
                        \includegraphics[scale=0.9]{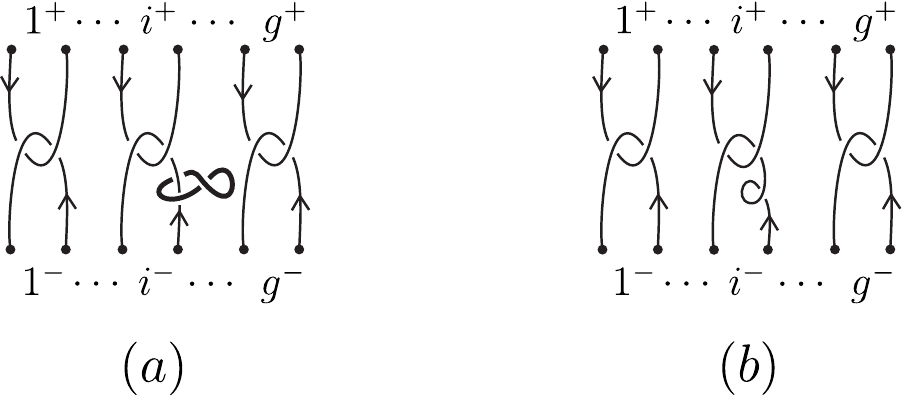}									\caption{Bottom-top tangle presentation of $c(t_{\alpha_i})$.}
\label{figuraLMO5artalt} 										
\end{figure}
\end{example}

\begin{example}\label{ejemploLMO4} Let $\alpha_{12}$ be the curve shown in Figure  \ref{figuraLMO6artalt} $(a)$ and let $t_{\alpha_{12}}$ be the  Dehn twist along $\alpha_{12}$. We have $c(t_{\alpha_{12}})\in{}^s\!\mathcal{LC}ob(g,g)$.   As in Example \ref{ejemploLMO3},  Figure \ref{figuraLMO6artalt} $(c)$  shows the bottom-top tangle presentation of $c(t_{\alpha_{12}})$ obtained by surgery along the thick component in Figure \ref{figuraLMO6artalt} $(b)$. 

\begin{figure}[ht!]
										\centering
                        \includegraphics[scale=0.9]{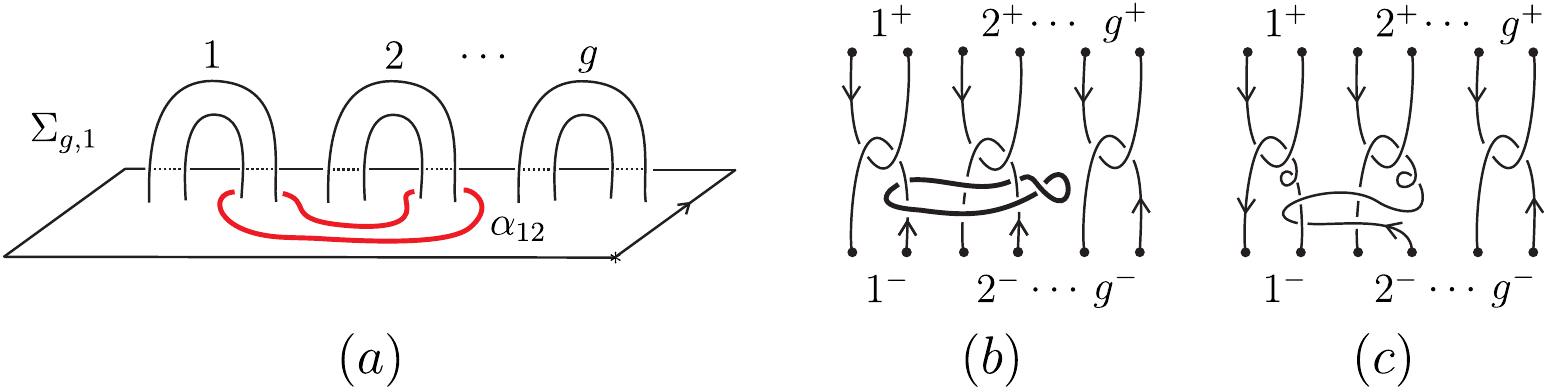}
												
\caption{$(a)$ Curve $\alpha_{12}$ and $(c)$ bottom-top tangle presentation of $c(t_{\alpha_{12}})$.}
\label{figuraLMO6artalt} 								
\end{figure}
\end{example}

\begin{example}\label{ejemploLMO4.5} Example \ref{ejemploLMO4} can be generalized. Consider two  integers $k$ and $l$ with   $1\leq k < l \leq g$. Let  $\alpha_{kl}$ be the simple closed curve which turns around the $k$-th handle and the $l$-th handle as shown in Figure \ref{figuraJTHLMO8artalt} $(a)$. Consider the Dehn twist  $t_{\alpha_{kl}}$ along $\alpha_{kl}$. We have $c(t_{\alpha_{kl}})\in{}^s\!\mathcal{LC}ob(g,g)$.   Figure \ref{figuraJTHLMO8artalt} $(b)$  shows the bottom-top tangle presentation of $c(t_{\alpha_{kl}})$.

\begin{figure}[ht!]
										\centering
                        \includegraphics[scale=0.9]{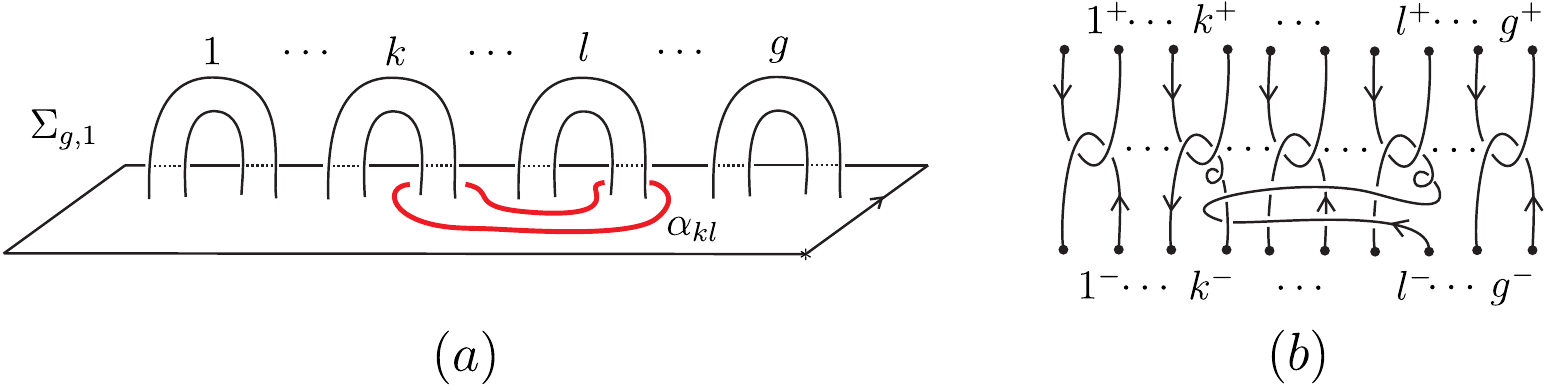}
												
\caption{$(a)$ Curve $\alpha_{kl}$ and $(b)$ bottom-top tangle presentation of $c(t_{\alpha_{kl}})$.}
\label{figuraJTHLMO8artalt} 						
\end{figure}
\end{example}

\begin{example}\label{ejemploLMO5} Let $N_i$ be  the cobordism from $\Sigma_{g,1}$ to $\Sigma_{g+1,1}$ with the bottom-top tangle presentation shown in Figure \ref{figuraLMO7artalt}. Then $N_i$ is a special Lagrangian cobordism. The label  $r$ on the first (from left to right) bottom component stands for \emph{root}. This is because from these cobordisms we will obtain, via the LMO functor, rooted trees with root $r$ that we will interpret as Lie commutators. See subsection \ref{section6.3}.
\begin{figure}[ht!]
										\centering
                        \includegraphics[scale=0.9]{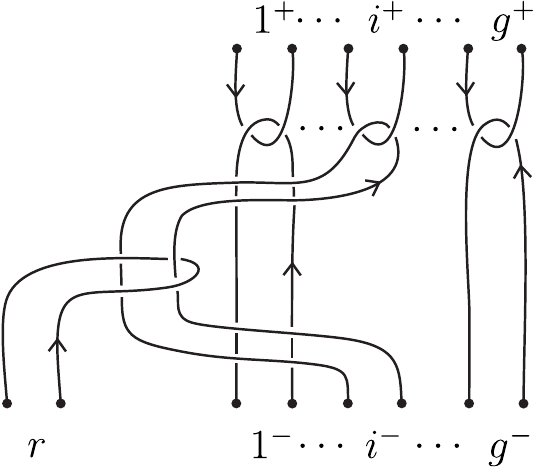}
												
\caption{Bottom-top tangle presentation of $N_i\in{}^s\!\mathcal{LC}ob(g,g+1)$.}
\label{figuraLMO7artalt} 										
\end{figure}
\end{example}

\para{Top-substantial Jacobi diagrams} Let us now describe the target category  ${}^{ts}\!\!\mathcal{A}$ of the LMO functor. The objects of the category ${}^{ts}\!\!\mathcal{A}$ are the non-negative integers. The set of morphisms from $g$ to $f$ is the subspace ${}^{ts}\!\!\mathcal{A}(g,f)$ of diagrams in $\mathcal{A}(\left\lfloor g\right\rceil^+\sqcup \left\lfloor f\right\rceil^-)$ (see Example \ref{ejemplo1JD}) without struts whose both ends are colored by elements of $\left\lfloor g\right\rceil^+$. These kind of Jacobi diagrams are called \emph{top-substantial}. If $D\in {}^{ts}\!\!\mathcal{A}(g,f)$ and $E\in {}^{ts}\!\!\mathcal{A}(h,g)$ the composition
 $$D\circ E =\left\langle D_{|j^+\mapsto j^*},\ E_{|j^-\mapsto j^*}\right\rangle_{\left\lfloor g\right\rceil^*}$$
is the element in ${}^{ts}\!\!\mathcal{A}(h,f)$ given by the sum  of Jacobi diagrams obtained by considering all the possible ways of gluing the $\left\lfloor g\right\rceil^+$-colored legs of $D$ with the $\left\lfloor g\right\rceil^-$-colored legs of $E$. A schematic description is shown in Figure \ref{figuraLMO1artalt} $(a)$. The identity morphism in ${}^{ts}\!\!\mathcal{A}(g,g)$ is shown in Figure \ref{figuraLMO1artalt} $(b)$.
\begin{figure}[ht!]
										\centering
                        \includegraphics[scale=0.85]{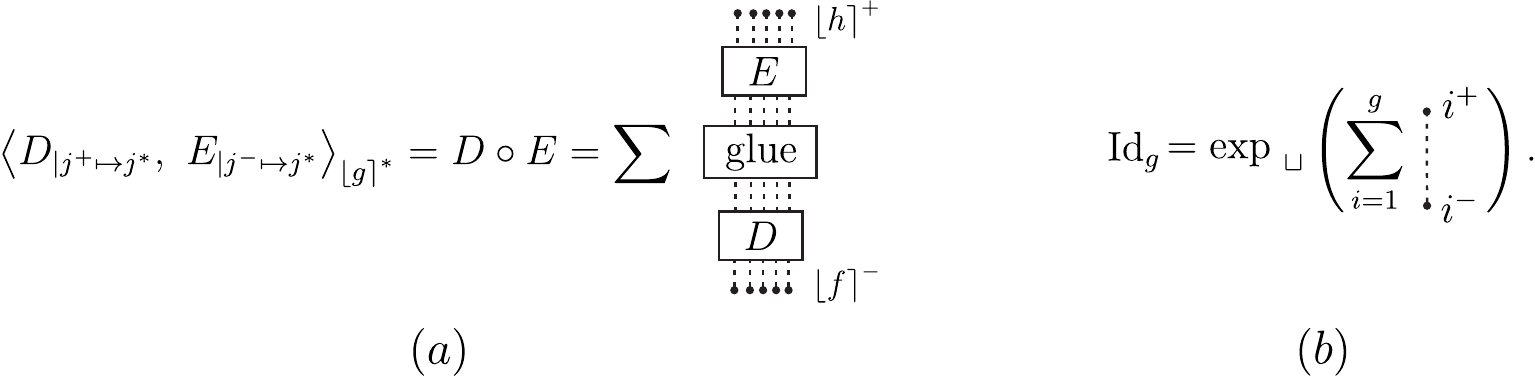}												
\caption{$(a)$ Composition in ${}^{ts}\!\!\mathcal{A}$ and  $(b)$ identity morphism in ${}^{ts}\!\!\mathcal{A}(g,g)$.}
\label{figuraLMO1artalt} 										
\end{figure}

\para{Sketch of the construction of the LMO functor}
The definition of the LMO functor uses the Kontsevich integral which is defined for $q$-tangles. Because of this, it is necessary to modify the objects of $\mathcal{LC}ob$ to obtain the category $\mathcal{LC}ob_q$: instead of non-negative integers, the objects of  $\mathcal{LC}ob_q$ are non-associative words in the single letter~$\bullet$. If $u$ and $v$ are non-associative words in $\bullet$ of length $g$ and $f$ respectively, a morphism from $u$ to $v$ is a Lagrangian cobordism from $\Sigma_{g,1}$ to $\Sigma_{f,1}$.

Roughly speaking, the LMO functor $\widetilde{Z}:\mathcal{LC}ob_q\rightarrow {}^{ts}\!\!\mathcal{A}$ is defined as follows. Let $M\in\mathcal{LC}ob_q(u,v)$, where $u$ and $v$ are two non-associative words in $\bullet$. Let  $(B,\gamma')$ be the bottom-top tangle presentation of $M$.  By performing the change $\bullet\mapsto (+-)$ in~$u$ and~$v$ we obtain words $w_t(\gamma')$ and $w_b(\gamma')$ on $\{+,-\}$ together with some parenthesizations. Hence $\gamma'$ is a $q$-tangle in the homology cube $B$. Next, take a \emph{surgery presentation} of $(B,\gamma')$, that is, a framed link $L\subseteq \text{int}([-1,1]^3)$ and a  tangle $\gamma$ in $[-1,1]^3\setminus L$ such that surgery along $L$ carries $([-1,1]^3,\gamma)$ to $(B,\gamma')$. Set $w_t(\gamma)=w_t(\gamma')$ and $w_b(\gamma)=w_b(\gamma')$. Hence $L\cup\gamma$ is a $q$-tangle in $[-1,1]^3$. Now, consider the Kontsevich integral of $L\cup\gamma$, which gives a series of a kind of Jacobi diagrams. To get rid of the ambiguity in the surgery presentation, it is necessary to use some combinatorial operations on the space of diagrams. Among these operations there is the so-called \emph{Aarhus integral} (see \cite{MR1931167,MR1931168}), which is a kind of formal Gaussian integration on the space of diagrams. We then arrive  to ${}^{ts}\!\!\mathcal{A}$. Finally, to obtain the functoriality, it is necessary to do a normalization. 

Recall that the definition of the Kontsevich integral requires the choice of a \emph{Drinfeld associator}, and the bottom-top tangle presentation requires the choice of a system of meridians and parallels.  Thus, the LMO functor also depends on these choices.

We are especially interested in the LMO functor for  special Lagrangian cobordisms. For these kind of cobordisms the LMO functor can be computed from the Kontsevich integral and the symmetrization map as is assured by a result of Cheptea, Habiro and Massuyeau. We state the result for our particular case.

\begin{convention}\label{convention1} From now on, we endow Lagrangian cobordisms  with the right-handed non-associative word  $(\bullet\cdots(\bullet(\bullet \bullet))\cdots)$ in the letter $\bullet$ unless we say otherwise. This way we will always be in the context of the category $\mathcal{LC}ob_q$.
\end{convention}
\begin{lemma}\cite[Lemma 5.5]{MR2403806}\label{lemmaslc} Let $M\in\mathcal{LC}ob_q(u,v)$, where $u$ and $v$ are non-associative words in the letter $\bullet$ of length $g$ and $f$, respectively. Suppose that the bottom-top tangle presentation of $M$ is as in Figure \ref{figuraLMO8artalt}, 
\begin{figure}[ht!]
										\centering
                        \includegraphics[scale=0.9]{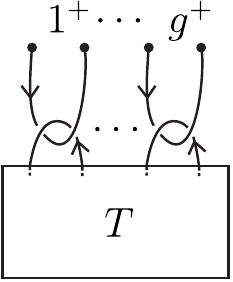}
												
\caption{Bottom-top tangle presentation of $M$.}
\label{figuraLMO8artalt} 										
\end{figure}
where $T$ is a tangle in $[-1,1]^3$. Endow $T$ with the non-associative words $w_t(T)=u_{/\bullet\mapsto (+-)}$  and $w_b(T)=v_{/\bullet\mapsto (+-)}$. Then the value of the LMO functor $\widetilde{Z}(M)$ can be computed from the value of the Kontsevich integral $Z(T)$ as  shown in Figure \ref{figuraLMO9artalt}.
\begin{figure}[ht!]
										\centering
                        \includegraphics[scale=0.9]{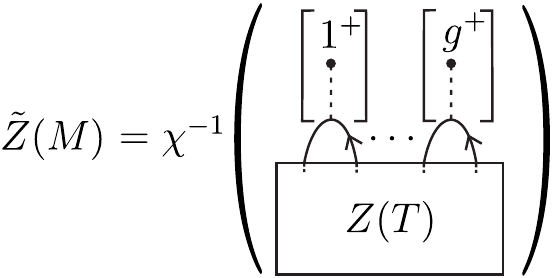}
												
\caption{Value of $\widetilde{Z}(M)$ in terms of $Z(T)$.}
\label{figuraLMO9artalt} 										
\end{figure}
\end{lemma}

Let $(M,m)$ be a homology cobordism and $(B,\gamma)$  its bottom-top tangle presentation. Define the \emph{linking matrix} of $(M,m)$, denoted $\text{Lk}(M)$,  as the linking matrix of the link~$\hat{\gamma}$ in $B$ obtained from $\gamma$ by identifying the two endpoints on each of the top and bottom components of $\gamma$.

For any Lagrangian cobordism $M$, denote by $\widetilde{Z}^s(M)$ the strut part of $\widetilde{Z}(M)$, that is, the reduction of $\widetilde{Z}(M)$  modulo diagrams with at least one trivalent vertex. Denote by $\widetilde{Z}^Y(M)$ the reduction of  $\widetilde{Z}(M)$ modulo struts. Denote by      $\widetilde{Z}^t(M)$ the reduction of $\widetilde{Z}(M)$ modulo looped diagrams. Finally denote by $\widetilde{Z}^{Y,t}(M)$ the reduction of   $\widetilde{Z}^t(M)$ modulo struts.

\begin{lemma}\cite[Lemma 4.12]{MR2403806}
\label{lemmasplitstruty} Let $M\in\mathcal{LC}ob_q(u,v)$ where $u$ and $v$ are non-associative words in the letter $\bullet$. Then $\widetilde{Z}(M)$ is group-like. Moreover $\widetilde{Z}(M)= \widetilde{Z}^s(M)\sqcup \widetilde{Z}^Y(M)$ and 
\begin{equation}\label{struteqn}
 \widetilde{Z}^s(M)=\left[\frac{\text{\emph{Lk}}(M)}{2}\right]. 
\end{equation}
\end{lemma}

The colors $1^+,\ldots,g^+$ and $1^-,\ldots,f^-$ in the series of Jacobi diagrams $\widetilde{Z}(M)$ refer to the curves $m_+(\beta_1)$,$\ldots$, $m_+(\beta_g)$ and $m_-(\alpha_1),\ldots, m_-(\alpha_f)$ on the top and bottom surfaces of $M$ respectively.

\begin{example}\label{ejemploLMO6} Let us consider the special Lagrangian cobordism $c(t_{\alpha_i})$, from Example~\ref{ejemploLMO3}, equipped with non-associative words as in Convention \ref{convention1}. By Lemma \ref{lemmaslc} and the functoriality of $Z$  (see Equation (\ref{functorialityZ})), to compute $\widetilde{Z}^t(c(t_{\alpha_i}))$ in low degrees  we need to first compute 
\begin{align*}
\includegraphics[scale=1]{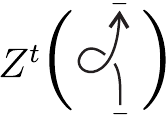}
\end{align*}
in low degrees, which we already computed in Example \ref{ejemplo2KI}. Therefore
\begin{align*}
\includegraphics[scale=0.9]{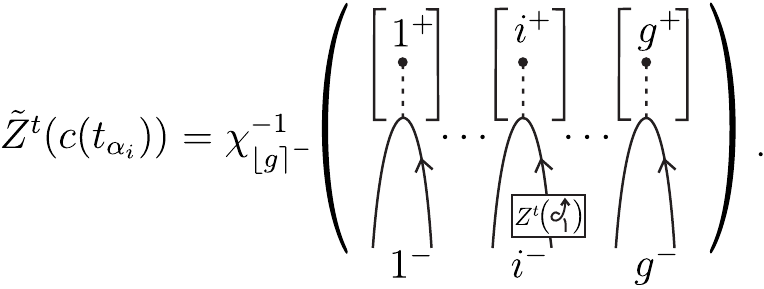}
\end{align*}											From Example \ref{exampleJD2.9}, we conclude 	
\begin{center}
\includegraphics[scale=0.74]{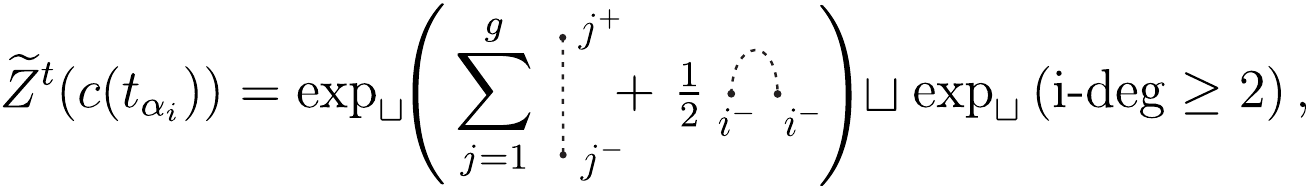}	
\end{center}
which shows that there are no terms of i-deg $=1$ in $\widetilde{Z}^{Y,t}(c(t_{\alpha_i}))$.
\end{example}

\begin{example}\label{ejemploLMO7} Consider the special Lagrangian cobordism $c(t_{\alpha_{12}})$ from Example \ref{ejemploLMO4}, equipped with non-associative words as in Convention \ref{convention1}. By Lemma \ref{lemmaslc},  to compute $\widetilde{Z}^t(c(t_{\alpha_{12}}))$ in low degrees,  we need to first compute   the tree-like part in the Kontsevich integral of the $q$-tangle 
\begin{align*}
\includegraphics[scale=1]{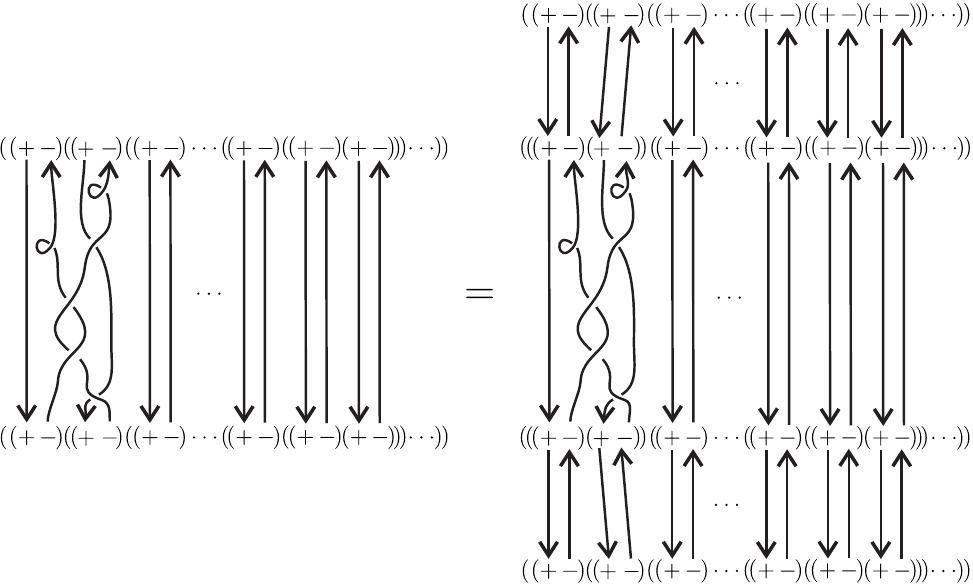}
\end{align*}
by the functoriality of $Z$, we have to compute the low degree terms of 
\begin{align*}
\includegraphics[scale=1]{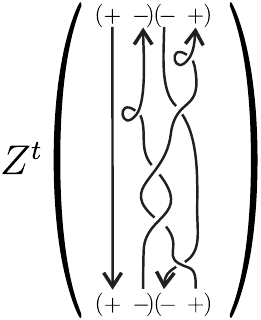}
\end{align*}
which was computed in Example \ref{ejemplo3KI}. Now, by a straightforward but long computation we obtain
\begin{center}
\includegraphics[scale=0.74]{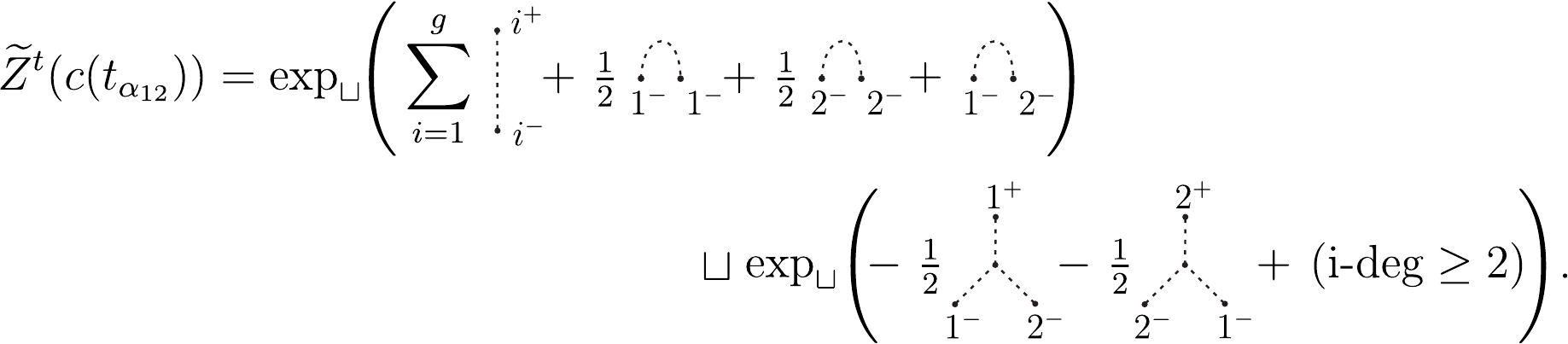}	
\end{center}
\end{example}

\begin{example}\label{ejemploLMO8} Example \ref{ejemploLMO7} can be generalized to the cobordism $c(t_{\alpha_{kl}})$ from Example~\ref{ejemploLMO4.5}. In this case we obtain
\begin{center}
\includegraphics[scale=0.74]{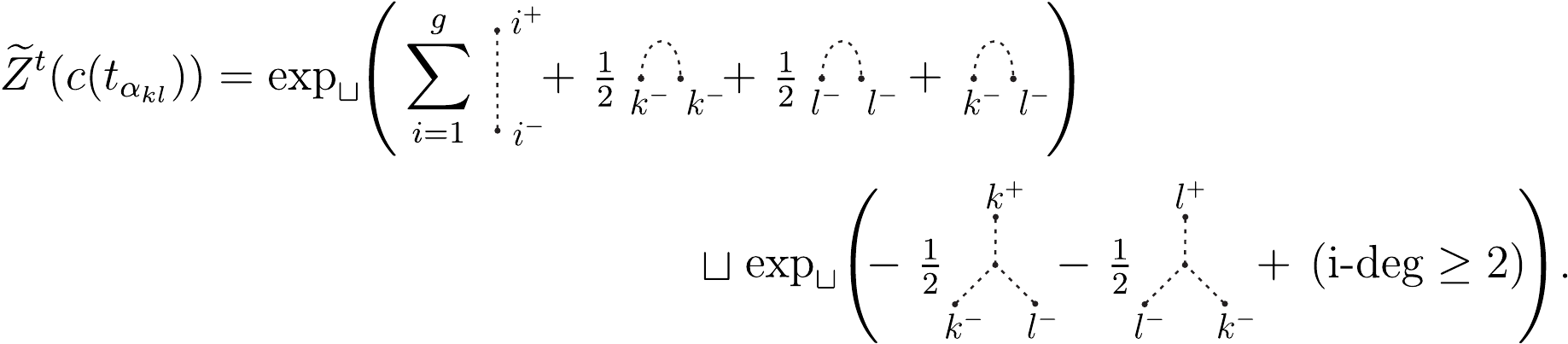}	
\end{center}
\end{example}

\begin{example}\label{ejemploLMO9} Consider the special Lagrangian cobordism $N_1$ from Example \ref{ejemploLMO5}, equipped with non-associative words as in Convention \ref{convention1}. Denote by $w$ the right-handed non-associative word in $\bullet$ of length $g-1$. Denote by $P_{\bullet,\bullet,w}$ the $q$-cobordism $((\bullet\bullet)w)\rightarrow (\bullet(\bullet w))$ whose underlying cobordism is the identity $\mathcal{LC}ob(g+1,g+1)$. Thus we can decompose $N_1$ as $N_1=P_{\bullet,\bullet,w}\circ (T\otimes \text{Id}_w)$, where $\text{Id}_w$  is the identity cobordism  equipped with $w$ on the top and bottom, and  $T$ is the special Lagrangian cobordism whose bottom-top tangle presentation is shown in Figure \ref{figuraLMO18artalt}.
\begin{figure}[ht!]
										\centering
                        \includegraphics[scale=0.9]{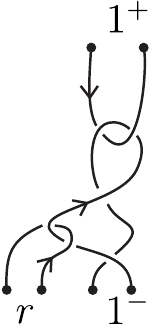}
												
\caption{Bottom top-tangle presentation of $T$.}
\label{figuraLMO18artalt} 								
\end{figure}

\noindent Hence, $\widetilde{Z}^t(N_1) = \widetilde{Z}^t(P_{\bullet,\bullet,w})\circ (\widetilde{Z}^t(T)\otimes \text{Id}_{g-1})$. Now, by the functoriality of $\widetilde{Z}$ we have
$$\left(\widetilde{Z}(P_{\bullet,\bullet,w})_{|r\mapsto 0}\right) = \varnothing\otimes\text{Id}_{g} \ \ \ \ \ \ \text{and}\ \ \ \ \ \ \left(\widetilde{Z}(P_{\bullet,\bullet,w})_{|1^-\mapsto 0}\right) = \text{Id}_1\otimes \varnothing\otimes\text{Id}_{g-1},$$
therefore 
$$\left(\widetilde{Z}^Y(P_{\bullet,\bullet,w})_{|r\mapsto 0}\right) = \varnothing \ \ \ \ \ \ \text{and}\ \ \ \ \ \ \left(\widetilde{Z}^Y(P_{\bullet,\bullet,w})_{|1^-\mapsto 0}\right) = \varnothing.$$
This way, each one of the connected diagrams appearing in $\widetilde{Z}^Y(P_{\bullet,\bullet,w})$ has at least one  $r$-colored leg and at least one  $1^-$-colored leg. Hence, each one of the connected diagrams in $\widetilde{Z}^t(N_1)$ coming  from  $\widetilde{Z}^Y(P_{\bullet,\bullet,w})$  has at least one $r$-colored leg and at least one  $1^-$-colored leg. 

We are interested in the low degree terms of $\widetilde{Z}^t(N_1) \text{ mod } \mathcal{H}(r)$. By Lemma \ref{lemmaslc}, we need to compute the low degree terms of
\begin{align*}
\includegraphics[scale=1]{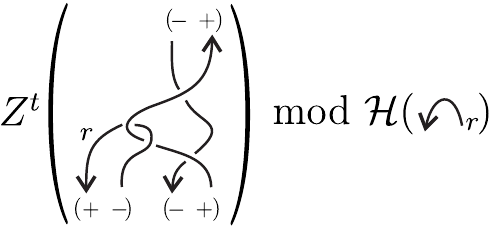}
\end{align*}
which we already computed in Example \ref{ejemplo4KI}. Whence we obtain
\begin{center}
\includegraphics[scale=0.74]{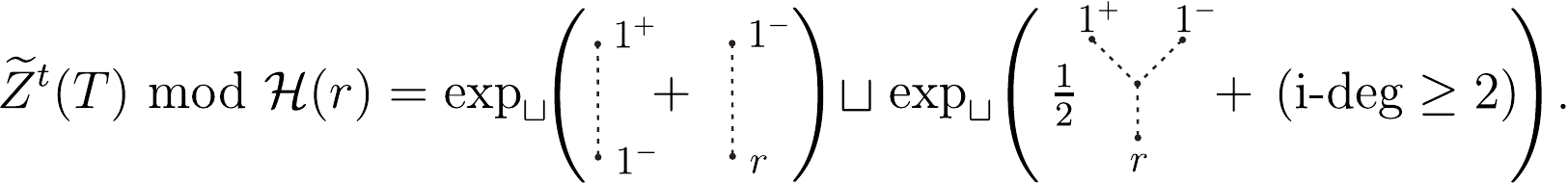}	
\end{center}
We conclude that each of the terms with i-deg $=1$ in $\widetilde{Z}(N_1) \text{ mod } \mathcal{H}(r)$ has one $r$-colored  and one $1^-$-colored leg. In a similar way, it can be shown for $1\leq i\leq g$ that each of the terms with i-deg $=1$ in $\widetilde{Z}(N_i) \text{ mod } \mathcal{H}(r)$ has one $r$-colored leg and one $i^-$-colored leg.
\end{example}

\section{Johnson-type filtrations}\label{JTH1}

 As in subsection \ref{subsectionLMOfunctor}, we denote by $\Sigma_{g,1}$ a compact connected oriented surface of genus~$g$ with one  boundary component. Let $\mathcal{M}_{g,1}$ denote  the mapping class group of $\Sigma_{g,1}$. We will often omit the subscripts $g$ and $1$ of our notation unless there is ambiguity, then  we will usually write $\Sigma$ and $\mathcal{M}$ instead of $\Sigma_{g,1}$ and $\mathcal{M}_{g,1}$.

\subsection{Preliminaries}
 
Let us fix a base point $*\in\partial\Sigma$ and set $\pi=\pi_1(\Sigma,*)$ and $H=H_1(\Sigma;\mathbb{Z})$, finally denote by $\text{ab}:\pi\rightarrow H$ the abelianization map. Notice that the intersection form $\omega:H\otimes H \rightarrow \mathbb{Z}$ is a symplectic form on $H$. The elements of $\mathcal{M}$ preserve $\partial\Sigma$, in particular they preserve $*$, therefore we have a well defined group homomorphism:
\begin{equation}\label{JTHequ1}
\rho:\mathcal{M}\longrightarrow \text{Aut}(\pi),
\end{equation} 
which sends $h\in\mathcal{M}$ to the induced map $h_{\#}$ on $\pi$. It is well known that the map $\rho$  is  injective and it is called the \emph{Dehn-Nielsen-Baer representation} of $\mathcal{M}$. On the other hand, since the elements of $\mathcal{M}$ are orientation-preserving, their induced maps on $H$ preserve the intersection form. This way we have a well defined surjective group homomorphism:
\begin{equation}\label{JTHequ2}
\sigma:\mathcal{M}\longrightarrow\text{Sp}(H)=\{f\in\text{Aut}(H)\ |\ \forall x,y\in H, \ \omega(f(x),f(y))=\omega(x,y)\},
\end{equation}
that sends $h\in\mathcal{M}$ to the induced map $h_*$ on $H$. The map $\sigma$ is called the \emph{symplectic representation} of $\mathcal{M}$ and it is far from being injective, its kernel is known as the \emph{Torelli group} of $\Sigma$,  which is denoted by $\mathcal{I}$ (or $\mathcal{I}_{g,1}$), so
\begin{equation}\label{JTHequ3}
\mathcal{I}=\mathcal{I}_{g,1}=\text{ker}(\sigma)=\{h\in\mathcal{M}\ | \ h_*=\text{Id}_H\}.
\end{equation}

\subsection{Alternative Torelli group}

Let $V$ (or $V_g$) be a handlebody of genus $g$. Consider a disk $D$ on $\partial V$ such that $\partial V=\Sigma\cup D$, where $D$ and $\Sigma$ are glued along their boundaries. Let  $\iota:\Sigma \hookrightarrow V$ be the inclusion of $\Sigma$ into $\partial V\subseteq V$, see Figure \ref{figuraJTH1}.
\begin{figure}[ht!]
										\centering
                        \includegraphics[scale=0.75]{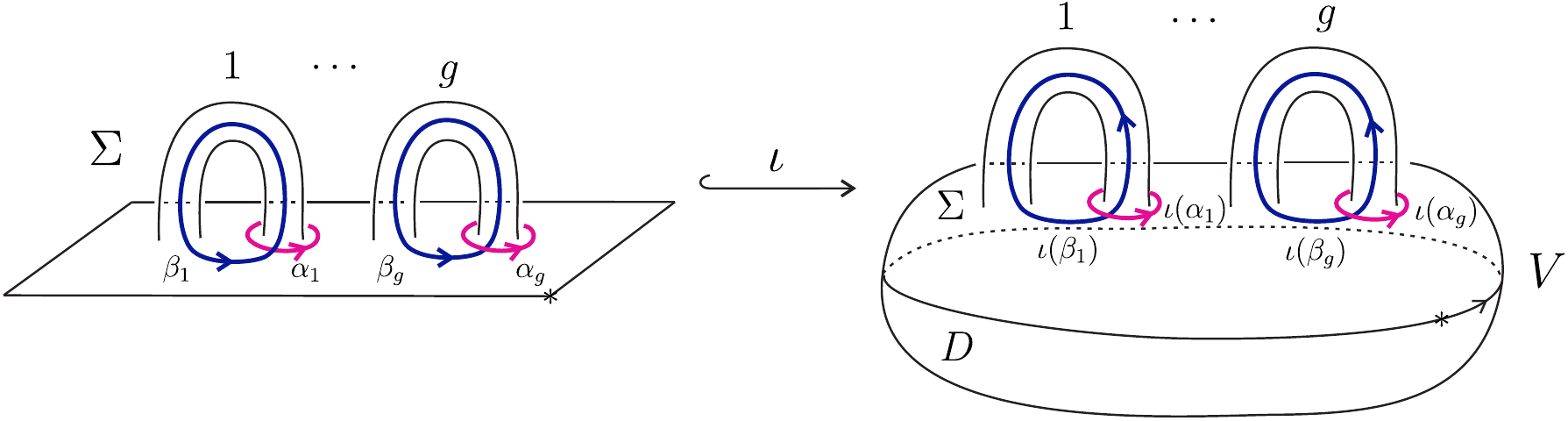}
												\caption{The inclusion $\Sigma\stackrel{\iota}{\hooklongrightarrow}V$.}											\label{figuraJTH1}
\end{figure}

Figure \ref{figuraJTH1}  also shows the fixed system of meridians and parallels of $\Sigma$ used in subsection \ref{subsectionLMOfunctor}. Moreover we suppose that the images $\iota(\alpha_i)$ of the meridians $\alpha_i$, under the embedding $\iota$, bound pairwise disjoint disks in $V$. Set $H'=H_1(V;\mathbb{Z})$ and $\pi'=\pi_1(V,\iota(*))$ and denote by $\text{ab}':\pi'\rightarrow H'$ the abelianization map. Consider the following subgroups of $\pi$ and $H$ that arise when looking at the induced maps by $\iota$ in homotopy and in homology:
\begin{equation}\label{JTHequ4}
A=\kernel(\iota_*:H\rightarrow H') \text{\ \ \ \ \ \ and\ \ \ \ \ \ } \mathbb{A}=\kernel(\iota_{\#}:\pi\rightarrow\pi').
\end{equation}

\noindent We also consider the following subgroup of $\pi$:
\begin{equation}\label{JTHequ5}
K_2=\text{ker}(\pi\stackrel{\iota_{\#}}{\longrightarrow}\pi'\stackrel{\text{ab'}}{\longrightarrow}H')=\mathbb{A}\cdot\Gamma_2\pi.
\end{equation}

\noindent The subgroup $A\leq H$ is  a Lagrangian subgroup of $H$ with respect to the intersection form on $H$ and it is the group that appears in the definition of Lagrangian cobordisms in the previous section. We may think of $K_2$ as the subgroup of $\pi$ generated by commutators of \emph{weight} 2, where the elements of $\pi$, not belonging to $\mathbb{A}$, are considered to have weight $1$ while the elements in $\mathbb{A}$ are considered to have weight $2$. The subgroups $A$, $\mathbb{A}$ and $K_2$ allow us to define some important subgroups of the mapping class group~$\mathcal{M}$.

\begin{definition} The \emph{Lagrangian mapping class group} of $\Sigma$, denoted by $\mathcal{L}$ (or $\mathcal{L}_{g,1}$) is defined as follows:
\begin{equation}\label{JTHequ6}
\mathcal{L}=\mathcal{L}_{g,1}=\{f\in\mathcal{M}_{g,1}\ \ |\ \ f_{*}(A)\subseteq A\}.  
\end{equation}
\end{definition}

We are mainly interested in three particular subgroups of $\mathcal{L}$, one of these is the Torelli group, see Equality (\ref{JTHequ3}).

\begin{definition} The \emph{Lagrangian Torelli group} of $\Sigma$, denoted by $\mathcal{I}^L$ (or $\mathcal{I}_{g,1}^L$), is defined as follows:
\begin{equation}\label{JTHequ7}
\mathcal{I}^L=\mathcal{I}_{g,1}^L=\{h\in\mathcal{L}\ \ | \ \ h_*|_A=\text{Id}_A\}.
\end{equation}
\end{definition} 

The groups $\mathcal{L}$ and $\mathcal{I}^L$ appear in the works \cite{MR1823501,MR2265877} of J. Levine  in connection with the theory of finite-type invariants of homology $3$-spheres. From an algebraic point of view these groups were studied by S. Hirose in \cite{MR2261393}, where he found a generating system for $\mathcal{L}$ and by T. Sakasai in \cite{MR2916276}, where he computed $H_1(\mathcal{L};\mathbb{Z})$ and $H_1(\mathcal{I}^L;\mathbb{Z})$.

\begin{definition}\label{defalttorelli} The \emph{alternative Torelli group} of $\Sigma$, denoted by $\mathcal{I}^{\mathfrak{a}}$ (or $\mathcal{I}_{g,1}^{\mathfrak{a}}$), is defined as follows:
\begin{equation}\label{JTHequ8}
\mathcal{I}^{\mathfrak{a}}=\mathcal{I}_{g,1}^{\mathfrak{a}}=\left\{\begin{array}{l|l}
         &\text{for } x\in\pi: \ \ \ h_{\#}(x)x^{-1}\in K_2  \\
        h\in\mathcal{L}&\ \ \ \ \ \  \text{and for } y\in K_2:\\
        & h_{\#}(y)y^{-1}\in \Gamma_3\pi\cdot [\pi,\mathbb{A}]=:K_3 
        \end{array}\right\}.
\end{equation}
\end{definition} 

Notice that  the definition of $\mathcal{I}^{\mathfrak{a}}$ involves the group $K_3=\Gamma_3\pi\cdot[\pi,\mathbb{A}]=[[\pi,\pi],\pi]\cdot[\pi,\mathbb{A}]$, which we see as the subgroup of $\pi$ generated by commutators of weight $3$. Like the Lagrangian Torelli group, the group $\mathcal{I}^{\mathfrak{a}}$ appears in  \cite{MR1823501,MR2265877,MR1601612} in connection with the theory of finite-type invariants but with  a  different definition: the second term of the Johnson-Levine filtration. Definition \ref{defalttorelli} comes from \cite{MR3828784}. In Proposition~\ref{JTH2prop14} we show the equivalence of the two definitions. J. Levine shows in \cite[Proposition 4.1]{MR1823501} that~$\mathcal{I}^{\mathfrak{a}}$ is generated by Dehn twists along simple closed curves (scc) whose homology class belongs to $A$. Equivalently, $\mathcal{I}^{\mathfrak{a}}$ is generated by Dehn twists along scc's which bound a surface in the handlebody $V$. This is the definition of $\mathcal{I}^{\mathfrak{a}}$ given in \cite{MR2265877,MR1601612}.

From the above definitions it follows that $\mathcal{I}\subseteq \mathcal{I}^L\subseteq\mathcal{L}$ and $\mathcal{I}^{\mathfrak{a}}\subseteq \mathcal{I}^L\subseteq\mathcal{L}$. But $\mathcal{I}^{\mathfrak{a}}\not\subseteq \mathcal{I}$ and $\mathcal{I}\not\subseteq \mathcal{I}^{\mathfrak{a}}$.  We shall call here the groups $\mathcal{I}$, $\mathcal{I}^L$ and $\mathcal{I}^{\mathfrak{a}}$ \emph{Torelli-type groups}.  In contrast with  $\mathcal{I}$, the groups $\mathcal{I}^L$ and $\mathcal{I}^{\mathfrak{a}}$ are not normal in $\mathcal{M}$, but  they are normal in $\mathcal{L}$.

\begin{example}\label{ejemploJTH1} The Dehn twists $t_{\alpha_i}$ and $t_{\alpha_{kl}}$ from Examples \ref{ejemploLMO3} and \ref{ejemploLMO4.5} are elements of the alternative Torelli group which do not belong to the Torelli group.
\end{example}

\begin{example}\label{ejemploJTH2} Consider the parallel $\beta_1$ and the curve $\gamma$ as shown in Figure \ref{figuraJTH2artalt}. These curves form a \emph{bounding pair}. Consider the Dehn twists $t_{\beta_1}$ and $t_{\gamma}$ along these curves. It can be shown that the homeomorphism $t_{\gamma}t^{-1}_{\beta_1}$ belongs to $\mathcal{I}\cap\mathcal{I}^{\mathfrak{a}}$.
\begin{figure}[ht!]
										\centering
                        \includegraphics[scale=0.9]{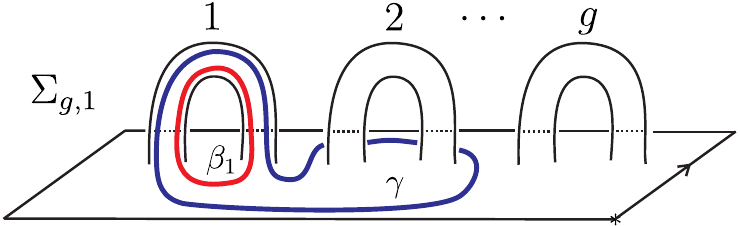}
												
\caption{Curves $\beta_{1}$ and $\gamma$.}
\label{figuraJTH2artalt} 							
\end{figure}
\end{example}

More generally we have the following lattice of subgroups:
$$\xymatrix{
  & {\ \mathcal{I}^{\mathfrak{a}}\ \ \ } \ar@{^{(}->}[rd] &   &  \\
 {\mathcal{I}\cap\mathcal{I}^{\mathfrak{a}}\ \ \ \ \ } \ar@{^{(}->}[ru] \ar@{^{(}->}[rd]&   & {\ \mathcal{I}^{L}\ \ } \ar@{^{(}->}[r] & \mathcal{L}\\
  & {\ \mathcal{I}\ \ \ } \ar@{^{(}->}[ru] &   & }
$$

\noindent  where all the inclusions are proper. Besides, J. Levine  proved in \cite[Theorem 2]{MR1705580} that 
\begin{equation}\label{JTHequ9}
\mathcal{I}\cap\mathcal{I}^{\mathfrak{a}}=\mathcal{K}\cdot [\mathcal{I},\mathcal{I}^{\mathfrak{a}}],
\end{equation}
where $\mathcal{K}$ is the \emph{Johnson kernel}. D. Johnson proved in \cite{MR793178} that $\mathcal{K}$ is generated by BSCC  maps (bounding scc's), that is, Dehn twists along scc's which are null-homologous in $\Sigma$.

\subsection{Alternative Johnson filtration}\label{subsection4.3}
This subsection is devoted to the study of a filtration of the alternative Torelli group introduced in \cite{MR3828784} which we shall call here the \emph{alternative Johnson filtration}. We compare this filtration with the \emph{Johnson filtration} and the \emph{Johnson-Levine filtration}. Let us start by recalling some terminology.

An \emph{$N$-series} $(G_m)_{m\geq 1}$ of a group $G$ is a decreasing sequence
$$G=G_1\geq G_2\geq\cdots\geq G_m\geq G_{m+1}\geq\cdots$$
of subgroups of $G$ such that $[G_i,G_j]\subseteq G_{i+j}$ for $i,j\geq 1$. We are interested in $N$-series of the group $\pi=\pi(\Sigma,*)$. A first example of an $N$-series of $\pi$ is the lower central series $(\Gamma_k \pi)_{k\geq1}$. We consider an $N$-series of $\pi$ in which the subgroup $\mathbb{A}$ plays a special role.

Set $K_1=\pi$ and $K_2=\mathbb{A}\cdot\Gamma_2\pi$  as defined in  (\ref{JTHequ5}). Let $(K_m)_{m\geq1}$ be the smallest $N$-series of $\pi$ starting with these $K_1$ and $K_2$, that is, if $(G_i)_{m\geq1}$ is any $N$-series of $\pi$ with $G_1=K_1$ and $G_2=K_2$ then $K_m\subseteq G_m$ for every $m\geq1$.  More precisely, for every $m\geq 3$ we have 
\begin{equation}\label{JTHequ11}
K_m=[K_{m-1}, K_1]\cdot[K_{m-2},K_2].
\end{equation}
In particular $K_3=\Gamma_3\pi\cdot [\pi,\mathbb{A}]$ is the group that we used in the definition of the alternative Torelli group, see (\ref{JTHequ8}). We can think of $K_m$ as the subgroup of $\pi$ generated by commutators of weight $m$, where the elements of $\pi\setminus \mathbb{A}$ have weight $1$ and the elements of $\mathbb{A}$ have weight $2$. By induction on $m\geq 1$ we have
\begin{equation}\label{JTHequ12}
\Gamma_m\pi\subseteq K_m\subseteq \Gamma_{\left\lceil m/2\right\rceil}\pi,
\end{equation}
where $\left\lceil m/2\right\rceil$ denotes the least integer greater than or equal to $m/2$. 

Restricting the Dehn-Nielsen-Baer representation (\ref{JTHequ1}) to the Lagrangian mapping class group we get an action of $\mathcal{L}$ on $K_1=\pi$. We denote the action of $h\in\mathcal{L}$ on $x\in\pi$ by $^hx$. Hence $^hx=\rho(h)(x)=h_{\#}(x)$.

\begin{lemma}\label{JTHlemma2}
For every $h\in\mathcal{L}$ we have $^h(K_2)=K_2$.
\end{lemma}
\begin{proof}
It is enough to show $^h(K_2)\subseteq K_2$ for every $h\in\mathcal{L}$.  Let $h\in\mathcal{L}$ and $x\in K_2=\text{ker}(\text{ab}'\iota_{\#})$. Hence $0=\text{ab}'\iota_{\#}(x)=\iota_*(\text{ab}(x))$, so $\text{ab}(x)\in A$ and then $h_*(\text{ab}(x))\in A$. Therefore
$$\text{ab}'\iota_{\#}h_{\#}(x)=\iota_*(\text{ab}(h_{\#}(x)))=\iota_*(h_*(\text{ab}(x)))=0,$$
that is, $h_{\#}(x)\in K_2$.
\end{proof}

It follows from Equality (\ref{JTHequ11}) and Lemma \ref{JTHlemma2}, by induction, that $^h(K_m)=K_m$ for every $m\geq1$ and $h\in\mathcal{L}$. From the general setting in \cite[Section 3.4 and Section 10.2]{MR3828784}  we have a decreasing sequence
\begin{equation}\label{JTHequ13}
\mathcal{L}=J^{\mathfrak{a}}_0\mathcal{M}\supseteq J^{\mathfrak{a}}_1\mathcal{M}\supseteq J^{\mathfrak{a}}_2\mathcal{M}\supseteq\cdots\supseteq J^{\mathfrak{a}}_m\mathcal{M}\supseteq J^{\mathfrak{a}}_{m+1}\mathcal{M}\supseteq\cdots
\end{equation}
of subgroups of $\mathcal{M}$ satisfying:
\begin{equation}\label{JTHequ14}
[J^{\mathfrak{a}}_l\mathcal{M},J^{\mathfrak{a}}_m\mathcal{M}]\subseteq J^{\mathfrak{a}}_{l+m}\mathcal{M}\ \ \ \ \text{for all } l,m\geq 0.
\end{equation}
\noindent In our case, the $m$-th term in this decreasing sequence is given by
 
\begin{equation}\label{JTHequ15}
J^{\mathfrak{a}}_m\mathcal{M}=J^{\mathfrak{a}}_m\mathcal{M}_{g,1}=\left\{\begin{array}{l|l}
         &\text{for } x\in\pi: \ \ \ h_{\#}(x)x^{-1}\in K_{1+m}  \\
        h\in\mathcal{L}&\ \ \ \ \ \  \text{and for } y\in K_2:\\
        & \ \ \ h_{\#}(y)y^{-1}\in K_{2+m}
        \end{array}\right\}.
\end{equation}
 
\medskip
 
\begin{definition} The \emph{alternative Johnson filtration} of $\mathcal{M}$ is the descending chain of subgroups  $\{J^{\mathfrak{a}}_m\mathcal{M}\}_{m\geq 0}$.
\end{definition}

\begin{example}\label{JTHexample3} 
Consider the curves  $\delta,\epsilon$ and the meridian $\alpha_g$ as show in Figure \ref{figuraJTH3artalt}. It can be show that $t_{\delta}$ and $t_{\epsilon}t^{-1}_{\alpha_g}$ belong to $J_2^{\mathfrak{a}}\mathcal{M}$. We will show this explicitly in Examples~\ref{JTH2example3} and~\ref{JTH2ejemplo4}. In particular  $t_{\delta}$ and $t_{\epsilon}t^{-1}_{\alpha_g}$ belong to $\mathcal{I}\cap\mathcal{I}^{\mathfrak{a}}$.
\begin{figure}[ht!]
\centering   	
	\includegraphics[scale=0.7]{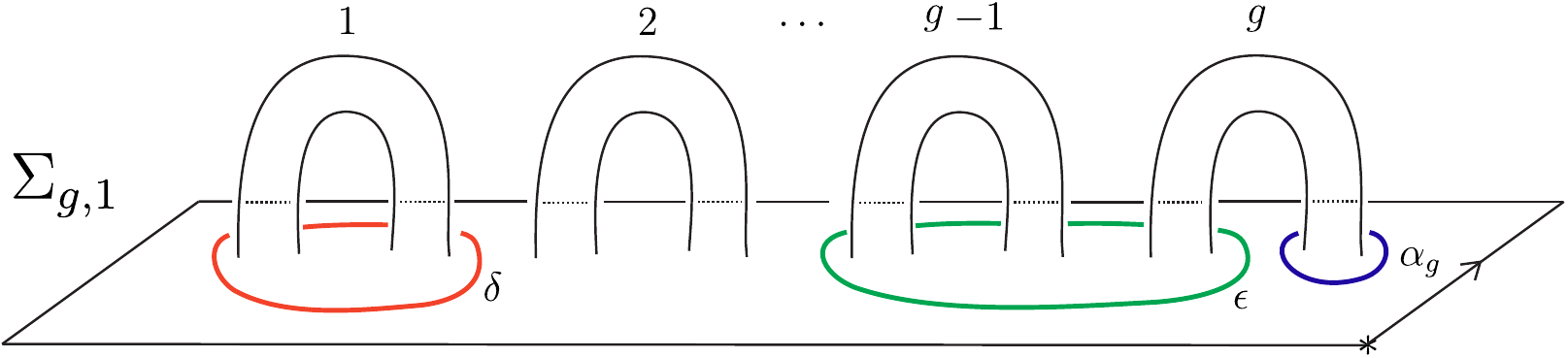}
	\caption{Curves $\delta, \epsilon$ and $\alpha_g$.}
\label{figuraJTH3artalt}
\end{figure}
\end{example}

\begin{proposition}\label{prop1altJF} The alternative Johnson filtration satisfies the following properties.
 \begin{enumerate}
\item [\emph{(i)}]  $\bigcap_{m\geq 0}J^{\mathfrak{a}}_m\mathcal{M}=\{\text{\emph{Id}}_{\Sigma}\}$.
\item [\emph{(ii)}] For all $k\geq 1$ the group $J^{\mathfrak{a}}_k\mathcal{M}$ is residually nilpotent, that is, $\bigcap_m\Gamma_mJ_k^{\mathfrak{a}}\mathcal{M}=\{\text{\emph{Id}}_{\Sigma}\}$.
\end{enumerate}
\end{proposition}

\begin{proof}
In order to prove (i), recall that 
$K_m\subseteq \Gamma_{\left\lceil m/2\right\rceil}\pi$ for $m\geq 1$. Consider  $h\in\mathcal{L}$ such that $h\in J^{\mathfrak{a}}_m\mathcal{M}$ for all $m\geq 0$. Let $x\in \pi$, thus
$$\forall m\geq 1,\ h_{\#}(x)x^{-1}\in K_{m+1}\subseteq \Gamma_{\left\lceil (m+1)/2\right\rceil}\pi.$$
Therefore $h_{\#}(x)x^{-1}\in \Gamma_k\pi$ for all $k\geq 1$. Since $\pi$ is residually nilpotent, we have that $h_{\#}=\rho(h)=\text{Id}_{\pi}$. In view of  the injectivity of the Dehn-Nielsen-Baer representation $\rho$ we conclude that $h=\text{Id}_{\Sigma}$, so we have (i). Now, let us see (ii). Fix $k\geq 1$, from (\ref{JTHequ14}) it follows, by induction on $m$, that $\Gamma_mJ^{\mathfrak{a}}_k\mathcal{M}\subseteq J^{\mathfrak{a}}_m\mathcal{M}$ for all $m\geq 1$. Therefore by (i) we obtain $\bigcap_{m\geq 1}\Gamma_m J^{\mathfrak{a}}_k\mathcal{M}=\{\text{Id}_{\Sigma}\}$.
\end{proof}

The Johnson filtration satisfies similar properties  to those stated in the above proposition. Let us briefly recall  the Johnson filtration and the Johnson-Levine filtration in order to compare them with each other.

\para{Johnson filtration} The lower central series of $\pi$ is preserved by the Dehn-Nielsen-Baer representation $\rho$, so for every $k\geq 1$ there is a group homomorphism
\begin{equation}\label{JTHequ16}
\rho_k:\mathcal{M}\longrightarrow \Aut(\pi/\Gamma_{k+1}\pi),
\end{equation}
defined as the composition
$$\mathcal{M}\stackrel{\rho}{\longrightarrow}\text{Aut}(\pi)\longrightarrow \text{Aut}(\pi/\Gamma_{k+1}\pi).$$
Notice that $\kernel(\rho_1)$ is the Torelli group $\mathcal{I}$. The \emph{Johnson filtration} of $\mathcal{M}$ is the descending chain of subgroups 
\begin{equation}\label{JTHequ17}
\mathcal{M}\supseteq\mathcal{I}=J_1\mathcal{M} \supseteq J_2\mathcal{M} \supseteq J_3\mathcal{M} \supseteq \cdots
\end{equation}
defined by $J_k\mathcal{M}:=\kernel(\rho_k)$ for $k\geq 1$. Equivalently for $k\geq 1$, 
\begin{equation}\label{JTHequ18}
J_k\mathcal{M}=\{h\in\mathcal{M}\ | \ \text{for all } x\in\pi:\  h_{\#}(x)x^{-1}\in\Gamma_{k+1}\pi \}.
\end{equation}

\begin{proposition}\cite[Corollary 3.3]{MR1133875} The Johnson filtration satisfies the following properties.
 \begin{enumerate}
 \item [\emph{(i)}] $[J_k\mathcal{M},J_m\mathcal{M}]\subseteq J_{k+m}\mathcal{M}$ for all $k,m\geq 1$.
\item [\emph{(ii)}]  $\bigcap_{k\geq 1}J_k\mathcal{M}=\{\text{\emph{Id}}_{\Sigma}\}$.
\item [\emph{(iii)}] For all $k\geq 1$ the group $J_k\mathcal{M}$ is residually nilpotent.
\end{enumerate}
\end{proposition}

\para{Johnson-Levine filtration} J. Levine introduced in \cite{MR1823501,MR2265877} a different filtration of the mapping class group by means of the embedding $\iota:\Sigma\hookrightarrow V$, see Figure \ref{figuraJTH1}, and the lower central series of $\pi'=\pi_1(V,\iota(*))$.

The \emph{Johnson-Levine filtration} of $\mathcal{M}$ is the descending chain of subgroups 
\begin{equation}\label{JTHequ19}
\mathcal{I}^L=J^L_1\mathcal{M} \supseteq J^L_2\mathcal{M} \supseteq J^L_3\mathcal{M} \supseteq \cdots
\end{equation}
defined by 
\begin{equation}\label{JTHequ20}
J^L_k\mathcal{M}:=\{h\in\mathcal{I}^L\ | \  \iota_{\#}h_{\#}(\mathbb{A})\subseteq\Gamma_{k+1}\pi' \}
\end{equation}
for $k\geq 1$. 

Let $\mathcal{H}$  be the subgroup of $\mathcal{M}$ consisting of the elements that can be extended to the handlebody $V$. In Example \ref{ejemploLMO2} we used these kind of homeomorphisms to give examples of special Lagrangian cobordisms.  It is well known that 
\begin{equation}\label{handlebodygroup}
\mathcal{H}=\{h\in\mathcal{M}\ |\ h_{\#}(\mathbb{A})\subseteq \mathbb{A}\},
\end{equation}
see \cite[Theorem 10.1]{MR0159313}. The group $\mathcal{H}$ is called the \emph{handlebody group} because it is isomorphic to the mapping class group of $V$.

\begin{proposition}{(Levine \cite{MR1823501,MR2265877})}\label{JTHprop3} The Johnson-Levine filtration satisfies the following properties.
\begin{enumerate}
\item [\emph{(i)}] For $k\geq1$,  $J^L_k\mathcal{M}$ is a subgroup of $\mathcal{M}$. 
\item [\emph{(ii)}] $\bigcap_{k\geq1} J_k^L\mathcal{M}=\mathcal{H}\cap\mathcal{I}^L$.
\item [\emph{(iii)}] $J_k\mathcal{M}\subseteq J_k^L\mathcal{M}$ for every $k\geq 1$.
\item [\emph{(iv)}] $J_2^L\mathcal{M}$ is generated by simple closed curves which bound in $V$, equivalently by scc's whose homology class belongs to $A$.
\item [\emph{(v)}] $\mathcal{I}^L=\mathcal{I}\cdot(\mathcal{H}\cap\mathcal{I}^L)$ and $J_2^L\mathcal{M}=J_2\mathcal{M}\cdot(\mathcal{H}\cap\mathcal{I}^L)$.
\end{enumerate}
\end{proposition}

We refer to the alternative Johnson filtration, the Johnson filtration and the Johnson-Levine filtration as \emph{Johnson-type filtrations}. 

\para{Comparison between Johnson-type filtrations}  Proposition \ref{JTHprop3} gives a first comparison between the three filtrations. Let us give a more general comparison.

\begin{lemma}\label{JTHlemma3}
For every $m\geq 1$ there exists a normal subgroup $N_m$ of $\mathbb{A}$ such that  $K_m=\Gamma_m\pi\cdot N_m$.
\end{lemma}

\begin{proof} The argument is by strong induction on $m$. Taking $N_1=\{1\}$ and $N_2=\mathbb{A}$, clearly we have $K_1=\Gamma_1\pi\cdot N_1$ and $K_2=\Gamma_2\pi\cdot N_2$. Suppose $m\geq 3$ and let $N_{m-2},N_{m-1}$ be normal subgroups of $\mathbb{A}$ such that $K_{m-1}=\Gamma_{m-1}\pi\cdot N_{m-1}$ and $K_{m-2}=\Gamma_{m-2}\pi\cdot N_{m-2}$. Thus

\begin{equation*}
\begin{split}
K_m& = [K_{m-1}, K_1]\cdot [K_{m-2}, K_2]\\					   & = [\Gamma_{m-1}\pi\cdot N_{m-1}, \pi]  \cdot[\Gamma_{m-2}\pi\cdot N_{m-2}, K_2]\\							 & = [\Gamma_{m-1}\pi, \pi]\cdot[N_{m-1}, \pi]\cdot[\Gamma_{m-2}\pi, \Gamma_2\pi\cdot\mathbb{A}]\cdot[N_{m-2}, K_2]\\
& = \Gamma_m\pi\cdot [N_{m-1}, \pi]\cdot[\Gamma_{m-2}\pi, \Gamma_2\pi]\cdot[\Gamma_{m-2}\pi, \mathbb{A}]\cdot[N_{m-2}, K_2]\\
& = \Gamma_m\pi\cdot N_m,
\end{split}
\end{equation*}
where $N_{m}=[N_{m-1}, \pi]\cdot[\Gamma_{m-2}\pi, \mathbb{A}]\cdot[N_{m-2}, K_2]$ is a normal subgroup of $\mathbb{A}$.
\end{proof}

\begin{proposition}\label{JTHproposition5} For every $m\geq 1$, we have
\begin{enumerate}
\item [\emph{(i)}] $J_{2m}^{\mathfrak{a}}\mathcal{M}\subseteq J_m\mathcal{M}$.
\item [\emph{(ii)}] $J_m\mathcal{M}\subseteq J_{m-1}^{\mathfrak{a}}\mathcal{M}$.
\item [\emph{(iii)}] $J_{m}^{\mathfrak{a}}\mathcal{M}\subseteq J_{m+1}^L\mathcal{M}$.
\end{enumerate}
In particular the Johnson filtration and the alternative Johnson filtration are cofinal.
\end{proposition}

\begin{proof}
Let $m\geq 1$. Let $h\in J^{\mathfrak{a}}_{2m}\mathcal{M}$, then for every $x\in\pi$ we have
$$h_{\#}(x)x^{-1}\in K_{2m+1}\subseteq\Gamma_{\left\lceil (2m+1)/2\right\rceil}\pi=\Gamma_{m+1}\pi,$$
that is, $h\in J_m\mathcal{M}$ so  (i) holds.  Let $h\in J_m\mathcal{M}$, then for every $x\in\pi$ we have
$$h_{\#}(x)x^{-1}\in \Gamma_{m+1}\pi\subseteq K_{m+1}.$$
In particular, $h_{\#}(x)x^{-1}\in K_m$ for every $x\in\pi$ and $h_{\#}(y)y^{-1}\in K_{m+1}$ for every $y\in K_2$. That is, $h\in J^{\mathfrak{a}}_{m-1}\mathcal{M}$, hence (ii). Finally for (iii) we use Lemma \ref{JTHlemma3} to write $K_{m+2}=\Gamma_{m+2}\pi\cdot N$ with $N$ a normal subgroup of $\mathbb{A}$. Let $h\in J^{\mathfrak{a}}_{m}\mathcal{M}$. It follows that for every $\alpha\in\mathbb{A}\subseteq K_2$, $h_{\#}(\alpha)\alpha^{-1}\in K_{m+2}$. Write $h_{\#}(\alpha)\alpha^{-1}=xn$ with $x\in\Gamma_{m+2}\pi$ and $n\in N$. Therefore
$$\iota_{\#}(h_{\#}(\alpha))=\iota_{\#}(h_{\#}(\alpha)\alpha^{-1})=\iota_{\#}(x)\iota_{\#}(n)=\iota_{\#}(x)\in\Gamma_{m+2}\pi',$$
whence $\iota_{\#}h_{\#}(\mathbb{A})\subseteq\Gamma_{m+2}\pi'$. Hence $h\in J^L_{m+1}\mathcal{M}$.
\end{proof}

\begin{remark} We expect that the subscripts of the relations on Proposition \ref{JTHproposition5} are the best possibles. 
\end{remark}

\begin{remark}\label{rksak}
D. Johnson proved in \cite{MR727699} that the Torelli group $\mathcal{I}_{g,1}$ is finitely generated for $g\geq 3$. This result together with the short exact sequence 
$$1\longrightarrow \mathcal{I}\longrightarrow \mathcal{I}^L \xrightarrow{\ \sigma\ }\sigma(\mathcal{I}^L)\longrightarrow 0,$$
where $\sigma$ is the symplectic representation, imply that the Lagrangian Torelli group $\mathcal{I}^L_{g,1}$ is finitely generated for $g\geq 3$. Notice that 
$$\sigma(\mathcal{I}^L)=\sigma(\mathcal{I}^{\mathfrak{a}}) = \left\{\left( \begin{smallmatrix} \text{Id}_g &\Delta\\ 0& \text{Id}_g\end{smallmatrix} \right) \  \bigg\rvert \ \ \Delta \text{\ is symmetric}  \right\}\subseteq\text{Sp}(2g,\mathbb{Z}),$$
see Lemma \ref{lemmaJTHLMO2} and Equation (\ref{JTHLMOequ10}). Hence $\sigma(\mathcal{I}^L)$ and $\sigma(\mathcal{I}^{\mathfrak{a}})$ are finitely generated. Recently T. Church, M. Ershov and A. Putnam proved in \cite{ChurchPutman} several results concerning the finite generation of the Johnson filtration. In particular they proved \cite[Theorem A]{ChurchPutman} that the Johnson kernel $\mathcal{K}_{g,1}$ is finitely generated for $g\geq 4$. This result together with the short exact sequences
$$1\longrightarrow \mathcal{K}\longrightarrow \mathcal{I}\cap\mathcal{I}^{\mathfrak{a}} \xrightarrow{\ \tau_1\ }\tau_1(\mathcal{I}\cap\mathcal{I}^{\mathfrak{a}})\longrightarrow 0    \ \ \ \ \text{and} \ \ \ \ 1\longrightarrow \mathcal{I}\cap\mathcal{I}^{\mathfrak{a}}\longrightarrow \mathcal{I}^{\mathfrak{a}} \xrightarrow{\ \sigma\ }\sigma(\mathcal{I}^{\mathfrak{a}})\longrightarrow 0,$$
where $\tau_1$ is the \emph{first Johnson homomorphism}, imply that the alternative Torelli group $\mathcal{I}^{\mathfrak{a}}_{g,1}$ is finitely generated for $g\geq 4$. Besides, it follows from the general result \cite[Theorem~B]{ChurchPutman} that   $J^{\mathfrak{a}}_m\mathcal{M}_{g,1}$ is finitely generated for $m\geq 2$ and $g\geq 2m+1$. We also have that $\mathcal{I}_{1,1}$ and $\mathcal{I}_{1,1}^{\mathfrak{a}}$ are (abelian) finitely generated groups. It would be interesting to study the group $\mathcal{I}_{2,1}^{\mathfrak{a}}$.
\end{remark}

\section{Johnson-type homomorphisms}\label{JTH2}
Throughout this section we use the same notations and conventions from Section \ref{JTH1}. The aim of this section is the study  of a sequence of group homomorphisms $\{\tau^{\mathfrak{a}}_m\}_{m\geq 0}$ introduced in \cite{MR3828784}, which are defined on each term of the alternative Johnson filtration and taking values in some abelian groups. These abelian groups  can be described by means of the first homology group $H$ of the surface $\Sigma$, the first homology group $B:=H'$ of the handlebody $V$ and the subgroup $A=\text{ker}(\iota_*)$. 

\subsection{Preliminaries}\label{sec5.1prel} Since $[J^{\mathfrak{a}}_l\mathcal{M},J^{\mathfrak{a}}_m\mathcal{M}]\subseteq J^{\mathfrak{a}}_{l+m}\mathcal{M}$ for all $l,m\geq 0$, the quotient group $J^{\mathfrak{a}}_m\mathcal{M}/J^{\mathfrak{a}}_{m+1}\mathcal{M}$ is an abelian group for $m\geq 1$ and we can endow
\begin{equation}\label{JTH2equ1}
\text{Gr}(J^{\mathfrak{a}}_{\bullet}\mathcal{M})=\bigoplus_{m\geq 1}\frac{J^{\mathfrak{a}}_m\mathcal{M}}{J^{\mathfrak{a}}_{m+1}\mathcal{M}}
\end{equation}
with a structure of graded Lie algebra with Lie bracket induced by the commutator operation. K. Habiro and G. Massuyeau show in \cite{MR3828784} that the Lie algebra (\ref{JTH2equ1}) embeds into a Lie algebra of derivations. To achieve this, they define  group homomorphisms on each term of the alternative Johnson filtration, even on the $0$-th term $J^{\mathfrak{a}}_0\mathcal{M}$ which is the Lagrangian mapping class group $\mathcal{L}$. We shall call these homomorphisms the \emph{alternative Johnson homomorphisms}.  In order to define them, let us start with some preliminaries.

\para{Free Lie algebra associated to the $N$-series $(K_m)_{m\geq1}$} In the definition of the alternative Johnson filtration we use the $N$-series $(K_m)_{m\geq1}$ defined in Equality (\ref{JTHequ11}). The graded Lie algebra associated to this $N$-series is given by
\begin{equation}\label{JTH2equ2}
\text{Gr}(K_{\bullet})=\bigoplus_{m\geq 1} \frac{K_m}{K_{m+1}}=\frac{K_1}{K_2}\oplus \frac{K_2}{K_3}\oplus\cdots
\end{equation}
It follows from \cite[Proposition 1]{MR0251111} that this graded Lie algebra is freely generated in degree $1$ and $2$, see also \cite[Lemma 10.9]{MR3828784}.  More precisely, by  Hopf's formula we have $(\Gamma_2\pi\cap\mathbb{A})/[\pi,\mathbb{A}]\cong H_2(\pi/\mathbb{A})\cong H_2(\pi')$ and $H_2(\pi')=0$ because $\pi'$ is a free group. Hence~$[\pi,\mathbb{A}]=\Gamma_2\pi\cap\mathbb{A}$. Consider the injective homomorphism $j:A\rightarrow K_2/K_3$ given by the composition 
\begin{equation}\label{JTH2equ3}
A \mathrel{\mathop{\longleftarrow}^{\mathrm{\cong}}_{\text{ab}}} (\mathbb{A}\cdot\Gamma_2\pi)/\Gamma_2\pi\cong \mathbb{A}/(\Gamma_2\pi\cap\mathbb{A})=\mathbb{A}/[\pi,\mathbb{A}]\lhook\joinrel\xrightarrow{\ \ }\frac{K_2}{K_3}.
\end{equation}
Identify $B=H'$ with $K_1/K_2$. Denote by $\mathfrak{Lie}(B;A)$ the graded free Lie algebra (over~$\mathbb{Z}$) generated by $B$ in degree $1$ and $A$ in degree $2$:
\begin{equation}\label{JTH2equ4}
\mathfrak{Lie}(B;A)=\bigoplus_{m\geq1}\mathfrak{Lie}_m(B;A)=B\oplus(\Lambda^2B\oplus A)\oplus\cdots
\end{equation}
Therefore we have 
\begin{equation}\label{JTH2equ5}
\text{Gr}(K_{\bullet})\cong\mathfrak{Lie}(B;A).
\end{equation}

\para{Positive symplectic derivations of $\mathfrak{Lie}(B;A)$} Recall that a derivation of $\mathfrak{Lie}(B;A)$  is a linear map
$d:\mathfrak{Lie}(B;A)\rightarrow \mathfrak{Lie}(B;A)$ such that $d([x,y])=[d(x),y]+[x,d(y)]$ for every $x,y\in\mathfrak{Lie}(B;A)$. The set $\text{Der}(\mathfrak{Lie}(B,A))$ of derivations of $\mathfrak{Lie}(B;A)$ is a Lie algebra with Lie bracket $[d,d']=dd'-d'd$.

From the long exact sequence associated to the pair $(V,\partial V)$ we obtain the short exact sequence
\begin{equation}\label{wthbasis1}
0\longrightarrow H_2(V, \partial V; \mathbb{Z})\xrightarrow{\ \delta_*\ } H \xrightarrow{\ \iota_*\ } B\longrightarrow 0,
\end{equation}
whence $H_2(V, \partial V; \mathbb{Z})\cong A$. Besides, by Poincar\'e-Lefschetz duality there is a canonical isomorphism $H_2(V, \partial V; \mathbb{Z})\cong H^1(V;\mathbb{Z})$ which allows to define the intersection form of the handlebody
\begin{equation}\label{wthbasis2}
\omega':B\times H_2(V, \partial V; \mathbb{Z})\longrightarrow \mathbb{Z}.
\end{equation}
Consider the identifications $B\cong A^*$  given by the isomorphism $H/A\xrightarrow{\ \iota_*\ } B$ and by sending $x+A\in H/A$ to $\omega(x,\cdot)\in A^*$; and $A\cong B^*$ given by the isomorphism $H_2(V,\partial V;\mathbb{Z})\cong A$ and by sending  $a\in A$ to  $\omega'(\cdot , a)\in B^*$. This way, the intersection form $\omega'$ determines an element $\Omega'\in\mathfrak{Lie}_3(B;A)$. The intersection form $\omega:H\otimes H\rightarrow\mathbb{Z}$ determines an element $\Omega\in\mathfrak{Lie}_2(H)\subseteq\mathfrak{Lie}(H)$, where $\mathfrak{Lie}(H)$ is the graded  Lie algebra freely generated by $H$ in degree $1$. The relation between the intersection form $\omega$ of $\Sigma$ and $\omega'$ of $V$ is given by the commutativity of the diagram:
\begin{equation*}
\xymatrix{  \ \ \ \ \ \ \ \ B\times H_2(V,\partial V)\ar[r]^{\ \ \ \ \ \ \ \ \ \ \ \ \omega'}\ar@<3.7ex>[d]_{\cong}^{\delta_*} & \mathbb{Z}.  \\
H/A\times A \ar@<3ex>[u]^{\iota_*}_{\cong}\ar@/_1pc/[ru]_{\ \ \ \omega} & }
\end{equation*}

\begin{definition}
Let $d$ be a derivation of $\mathfrak{Lie}(B;A)$.
\begin{enumerate}
\item [(i)] We say $d$ is a \emph{positive} derivation if $d(B) \subseteq \mathfrak{Lie}_{\geq 2}(B;A)$ and $d(A) \subseteq \mathfrak{Lie}_{\geq 3}(B;A).$
\item [(ii)] Let $m\geq 1$. We say that $d$ is a derivation of \emph{degree} $m$ if $d(B) \subseteq \mathfrak{Lie}_{m+1}(B;A)$ and $d(A) \subseteq \mathfrak{Lie}_{m+2}(B;A).$
\item [(iii)] We say that $d$ is a \emph{symplectic} derivation if $d(\Omega')=0$.
\end{enumerate}
\end{definition}

 Denote by $\text{Der}^{+,\omega}(\mathfrak{Lie}(B;A))$ the set of positive symplectic derivations of $\mathfrak{Lie}(B;A)$. This  set is a Lie subalgebra of $\text{Der}(\mathfrak{Lie}(B;A))$. Let $m\geq 1$, denote by  $\text{Der}_m(\mathfrak{Lie}(B;A))$ the subgroup  of derivations of  $\mathfrak{Lie}(B;A)$ of degree $m$. Notice that a derivation $d$ of $\mathfrak{Lie}(B;A)$ of degree $m$ is a family $d=(d_i)_{i\geq1}$ of group homomorphisms 
$$d_i:\mathfrak{Lie}_i(B;A)\longrightarrow \mathfrak{Lie}_{i+m}(B;A),$$
satisfying $d_{i+j}[x,y]=[d_i(x),y] + [x, d_j(y)]$ for $x\in \mathfrak{Lie}_i(B;A)$ and $y\in \mathfrak{Lie}_j(B;A)$.  Set
\begin{equation}\label{JTH2equ7}
D_m(\mathfrak{Lie}(B;A))=\text{Hom}_{\mathbb{Z}}(B,\mathfrak{Lie}_{m+1}(B;A))\oplus \text{Hom}_{\mathbb{Z}}(A,\mathfrak{Lie}_{m+2}(B;A)).
\end{equation}
The following is a classical result, see for instance  \cite[Lemma 0.7]{MR1231799}.
\begin{proposition}\label{JTH2prop1}
For every $m\geq 1$, there is a bijection
$$\text{\emph{Der}}_m(\mathfrak{Lie}(B;A))\xrightarrow{\ \Psi \ } D_m(\mathfrak{Lie}(B;A)),$$
defined by $\Psi(d)= d_{|B}+d_{|A}$ for $d\in\text{\emph{Der}}_m(\mathfrak{Lie}(B;A))$.
\end{proposition}

By using the identifications $B\cong A^*$ and $A\cong B^*$, we have 
\begin{equation}\label{JTH2equ8}
\begin{split}
D_m(\mathfrak{Lie}(B;A)) & \cong (B^*\otimes\mathfrak{Lie}_{m+1}(B;A))\oplus(A^*\otimes\mathfrak{Lie}_{m+2}(B;A))\\
																 & \cong (A\otimes\mathfrak{Lie}_{m+1}(B;A))\oplus(B\otimes\mathfrak{Lie}_{m+2}(B;A)).\\
\end{split}
\end{equation}
Hence we can see the map $\Psi$ from Proposition \ref{JTH2prop1} as taking values in the space on the left-hand side of Equation (\ref{JTH2equ8}). For $m\geq 1$, consider the Lie bracket map
\begin{equation}\label{JTH2equ9}
\Xi_m:(A\otimes\mathfrak{Lie}_{m+1}(B;A))\oplus(B\otimes\mathfrak{Lie}_{m+2}(B;A))\longrightarrow \mathfrak{Lie}_{m+3}(B;A).
\end{equation}

Set $D_m(B;A):=\text{ker}(\Xi_m)$. Let us denote by $\text{Der}^{+}_m(\mathfrak{Lie}(B;A))$ (respectively by $\text{Der}^{+,\omega}_m(\mathfrak{Lie}(B;A))$) the subgroup of positive (respectively positive symplectic) derivations of $\mathfrak{Lie}(B;A)$ of degree $m$.
\begin{proposition}\label{propwithbs}
Let $d\in \text{\emph{Der}}^{+}_m(\mathfrak{Lie}(B;A))$. Then  $\Xi_m\Psi(d)=0$ if and only if $d(\Omega')=0$. That is
$$\text{\emph{Der}}^{+,\omega}_m(\mathfrak{Lie}(B;A))\cong D_m(B;A).$$
\end{proposition}

\begin{proof}
Consider the symplectic basis $\{a_i,b_i\}$ induced by the systems of meridians and parallels $\{\alpha_i,\beta_i\}$ on $\Sigma$ shown in Figure \ref{figuraLMO2artalt} and identify $\iota_*(b_i)\in B$ with $b_i\in H$. Then, in this basis, the element $\Omega'$ is given by
\begin{equation}\label{JTH2equ6}
\Omega'=\sum_{i=1}^g[a_i,b_i]\in\mathfrak{Lie}_3(B;A).
\end{equation}
Using the identification (\ref{JTH2equ8}) we obtain
\begin{equation}\label{JTH2equ10}
\Psi(d)=\sum_{i=1}^g b^*_i\otimes d(b_i) + \sum_{i=1}^g a^*_i\otimes d(a_i) = \sum_{i=1}^g a_i\otimes d(b_i)-\sum_{i=1}^g b_i\otimes d(a_i).
\end{equation}
Hence
\begin{equation}\label{JTH2equ11}
\Xi_m\Psi(d)=\sum_{i=1}^g [a_i, d(b_i)]-\sum_{i=1}^g [b_i, d(a_i)]=d\left(\sum_{i=1}^g [a_i, b_i]\right)=d(\Omega')=0.
\end{equation}
\end{proof}

\subsection{Alternative Johnson homomorphisms}\label{sub5.2} Let $m$ be a positive integer and consider the space $D_m(\mathfrak{Lie}(B;A))$ defined in (\ref{JTH2equ7}).
\begin{definition}
The $m$-th \emph{alternative Johnson homomorphism} is the group homomorphism
\begin{equation}\label{JTH2equ12}
\tau^{\mathfrak{a}}_m:J^{\mathfrak{a}}_m\mathcal{M}\longrightarrow D_m(\mathfrak{Lie}(B;A)),
\end{equation}
that maps $h\in J^{\mathfrak{a}}_m\mathcal{M}$ to $\tau^{\mathfrak{a}}_m(h)=\left(\tau^{\mathfrak{a}}_m(h)_1,\ \tau^{\mathfrak{a}}_m(h)_2\right)$, where
$$\tau^{\mathfrak{a}}_m(h)_1(xK_2)=h_{\#}(x)x^{-1}K_{m+2}  \text{\ \ \  and \ \ \  } \tau^{\mathfrak{a}}_m(h)_2(a)=h_{\#}(y)y^{-1}K_{m+3}$$
for all $x\in\pi,a\in A$, here  $y\in \mathbb{A}\subseteq K_2$ is any lift of $a$, see (\ref{JTH2equ3}).
\end{definition}

We refer to \cite[Proposition 6.2]{MR3828784} for a proof of the homomorphism property. From the definition of the alternative Johnson homomorphisms it follows that for $m\geq 1$
\begin{equation}\label{JTH2equ13}
\text{ker}(\tau^{\mathfrak{a}}_m)=J^{\mathfrak{a}}_{m+1}\mathcal{M}.
\end{equation}

Consider the bases as in Proposition \ref{propwithbs} of $H$ and $B$ and keep the notation $\{\alpha_i,\beta_i\}$ for a free basis of $\pi$. Using the identification (\ref{JTH2equ8}) we have that the $m$-th alternative Johnson homomorphism of $h\in J^{\mathfrak{a}}_m\mathcal{M}$ is given by
\begin{equation}\label{JTH2equ14}
\begin{split}
\tau^{\mathfrak{a}}_m(h)& =\sum_{i=1}^g a_i\otimes \left(\tau^{\mathfrak{a}}_m(h)_1(\beta_iK_2)\right) - \sum_{i=1}^g b_i\otimes \left(\tau^{\mathfrak{a}}_m(h)_2(a_i)\right)\\
  & = \sum_{i=1}^g a_i\otimes \left(h_{\#}(\beta_i)\beta^{-1}_iK_{m+2}\right) - \sum_{i=1}^g b_i\otimes \left(h_{\#}(\alpha_i)\alpha^{-1}_iK_{m+3}\right).
\end{split}
\end{equation}

\begin{example}\label{JTH2example1} Consider the Dehn twist $h=t_{\alpha_i}$ from Example \ref{ejemploJTH1}, we know that $h\in J^{\mathfrak{a}}_1\mathcal{M}=\mathcal{I}^{\mathfrak{a}}$. Let us compute its first alternative Johnson homomorphism. We have $h_{\#}(\alpha_j)=\alpha_j$ for $1\leq j\leq g$, $h_{\#}(\beta_j)=\beta_j$ for $1\leq j\leq g$ with $j\not=i$ and $h_{\#}(\beta_i)=\alpha^{-1}_i\beta_i$. Hence
$$\tau^{\mathfrak{a}}_1(t_{\alpha_i}) = - a_i\otimes a_i.$$
\end{example}

\begin{example}\label{JTH2example2}
Consider the Dehn twist $h=t_{\alpha_{12}}$ from Examples \ref{ejemploJTH1} and \ref{ejemploLMO4}, which is an element of $\mathcal{I}^{\mathfrak{a}}$. The homotopy class of the curve $\alpha_{12}$ is represented by $\lambda=\alpha_2\alpha_1\in K_2$. We have $h_{\#}(\alpha_j)=\alpha_j$ and $h_{\#}(\beta_j)=\beta_j$ for $3\leq j\leq g$ and $h_{\#}(\alpha_1)=\lambda^{-1}\alpha_1\lambda$, $h_{\#}(\alpha_2)=\lambda^{-1}\alpha_2\lambda$, $h_{\#}(\beta_1)=\lambda^{-1}\beta_1$ and $h_{\#}(\beta_2)=\lambda^{-1}\beta_2$. Hence
$$\tau^{\mathfrak{a}}_1(t_{\alpha_{12}})= -(a_1\otimes a_1) - (a_2\otimes a_2) - (a_1\otimes a_2) - (a_2\otimes a_1).$$
Similarly for the Dehn twist $t_{\alpha_{kl}}$ from Example \ref{ejemploLMO4.5} we have
$$\tau^{\mathfrak{a}}_1(t_{\alpha_{kl}})= - (a_k\otimes a_k) - (a_l\otimes a_l) - (a_k\otimes a_l) - (a_l\otimes a_k).$$
\end{example}

\begin{example}\label{JTH2example3}
Consider the Dehn twist $h=t_{\delta}$, from Example \ref{JTHexample3}. We have that $t_{\delta}\in\mathcal{I}^{\mathfrak{a}}$. The homotopy class of the curve $\delta$ is represented by the commutator $\lambda=[\alpha_1,\beta^{-1}_1]\in K_3$. We have $h_{\#}(\alpha_j)=\alpha_j$ and $h_{\#}(\beta_j)=\beta_j$ for $2\leq j\leq g$ and $h_{\#}(\alpha_1)=\lambda^{-1}\alpha_1\lambda$ and  $h_{\#}(\beta_1)=\lambda^{-1}\beta_1\lambda$. Hence
$$\tau^{\mathfrak{a}}_1(t_{\delta})=0.$$
In particular $t_{\delta}\in J^{\mathfrak{a}}_2\mathcal{M}$.
\end{example}

\begin{example}\label{JTH2ejemplo4} Let $h=t_{\epsilon}t_{\alpha_g}^{-1}$ from Example \ref{JTHexample3}. The homotopy class of the curve $\epsilon$ is represented by $\lambda=\beta_g^{-1}\alpha_g^{-1}\beta_g[\alpha_{g-1},\beta_{g-1}]$. We have $h_{\#}(\alpha_i)=\alpha_i$ and $h_{\#}(\beta_i)=\beta_i$ for $1\leq i\leq g-2$, and $h_{\#}(\alpha_g)=\alpha_g$, $h_{\#}(\alpha_{g-1})=\lambda^{-1}\alpha_{g-1}\lambda$, $h_{\#}(\beta_{g-1})=\lambda^{-1}\beta_{g-1}\lambda$ and $h_{\#}(\beta_g)=\alpha_g\beta_g\lambda$. By a direct calculation we obtain
$$\tau^{\mathfrak{a}}_1(t_{\epsilon}t_{\alpha_g}^{-1})=0.$$
In particular $t_{\epsilon}t_{\alpha_g}^{-1}\in J^{\mathfrak{a}}_2\mathcal{M}$.
\end{example}

Notice that in Examples \ref{JTH2example1} and \ref{JTH2example2}, we have $\Xi_1\tau^{\mathfrak{a}}_1(t_{\alpha_i})=0$ and $\Xi_1\tau^{\mathfrak{a}}_1(t_{\alpha_{12}})=0$. This is a more general fact.

\begin{theorem}\label{JTH2thm1}
Let $m\geq 1$. For $h\in J^{\mathfrak{a}}_m\mathcal{M}$ we have $\Xi_m\tau^{\mathfrak{a}}_m(h)=0$, that is $$\tau^{\mathfrak{a}}_m(h)\in D_m(B;A)\cong \text{\emph{Der}}^{+,\omega}_m(\mathfrak{Lie}(B;A)).$$
In other words, $\tau^{\mathfrak{a}}_m(h)$ is a positive symplectic derivation of $\mathfrak{Lie}(B;A)$.
\end{theorem}

\begin{proof}
The proof is similar to the proof of \cite[Corollary 3.2]{MR1224104}. Consider the free basis $\{\alpha_i,\beta_i\}$ of $\pi$ induced by the system of meridians and parallels in  Figure \ref{figuraLMO2artalt}. Let $h\in J^{\mathfrak{a}}_m\mathcal{M}$. Since $h$ preserves the boundary $\partial \Sigma$ of $\Sigma$, then $h_{\#}$ fixes the inverse of the homotopy class $[\partial\Sigma]$ of $\Sigma$. So $h_{\#}([\partial\Sigma]^{-1})=[\partial\Sigma]^{-1}$, that is,
\begin{equation}\label{JTH2equ15}
h_{\#}\left(\prod_{i=1}^g [\beta^{-1}_i, \alpha_i]\right) = \prod_{i=1}^g [\beta^{-1}_i, \alpha_i].
\end{equation} 
For $1\leq i\leq g$ we have
$$\beta^{-1}_ih_{\#}(\beta_i)=\delta_i\in K_{1+m}\ \ \ \ \ \text{ and }\ \ \ \ \ h_{\#}(\alpha_i)\alpha^{-1}_i=\gamma_i\in K_{2+m}.$$

\medskip

\noindent Whence $h_{\#}(\beta^{-1}_i)=\delta^{-1}_i\beta^{-1}_{i}$. Hence

\begin{equation*}
\begin{split}
[h_{\#}(\beta^{-1}_i),h_{\#}(\alpha_i)] & = [\delta^{-1}_i\beta^{-1}_i, \gamma_i\alpha_i]\\
 & = \delta^{-1}_i\beta^{-1}_i\gamma_i\alpha_i\beta_i\delta_i\alpha^{-1}_i\gamma^{-1}_i\\
 & = \left(\delta^{-1}_i[\beta^{-1}_i, \gamma_i]\delta_i\right)\left(\delta^{-1}_i\gamma_i[\beta^{-1}_i, \alpha_i]\gamma^{-1}_i\delta_i\right)[\delta^{-1}_i,\gamma_i]\left(\gamma_i[\delta^{-1}_i,\alpha_i]\gamma^{-1}_i\right).
\end{split}
\end{equation*}

\medskip

\noindent It follows from Equality (\ref{JTH2equ15}) that
\begin{align}\label{JTH2equ161}
\prod_{i=1}^g [\beta^{-1}_i, \alpha_i] 
= \prod_{i=1}^g \left(\delta^{-1}_i[\beta^{-1}_i, \gamma_i]\delta_i\right)\left(\delta^{-1}_i\gamma_i[\beta^{-1}_i, \alpha_i]\gamma^{-1}_i\delta_i\right)[\delta^{-1}_i,\gamma_i]\left(\gamma_i[\delta^{-1}_i,\alpha_i]\gamma^{-1}_i\right).
\end{align}

\medskip

\noindent Now $[\beta^{-1}_i, \gamma_i]\in K_{3+m}$, $[\delta^{-1}_i,\gamma_i]\in K_{3+2m}\subseteq K_{m+4}$ and $[\delta^{-1}_i,\alpha_i]\in K_{3+m}$. Therefore, by considering Equation (\ref{JTH2equ161}) modulo $K_{m+4}$ we obtain
\begin{align*}
\prod_{i=1}^g [\beta^{-1}_i, \alpha_i]& \equiv \prod_{i=1}^g \left(\delta^{-1}_i[\beta^{-1}_i, \gamma_i]\delta_i\right)\left(\delta^{-1}_i\gamma_i[\beta^{-1}_i, \alpha_i]\gamma^{-1}_i\delta_i\right)[\delta^{-1}_i,\gamma_i]\left(\gamma_i[\delta^{-1}_i,\alpha_i]\gamma^{-1}_i\right)\\
 & \equiv \left(\prod_{i=1}^g [\beta^{-1}_i, \alpha_i]\right)\left(\prod_{i=1}^g [\beta_i^{-1},\gamma_i][\delta^{-1}_i,\alpha_i]\right).
\end{align*}
Thus
\begin{equation}\label{JTH2equ16}
\prod_{i=1}^g [\beta_i^{-1},\gamma_i][\delta^{-1}_i,\alpha_i]\in K_{m+4}.
\end{equation}

\noindent From (\ref{JTH2equ16}), identification (\ref{JTH2equ8}) and (\ref{JTH2equ14}) we have
\begin{equation}\label{JTH2equ181}
\begin{split}
0 & = \Xi_m\left(\sum_{i=1}^{g}(-b_i)\otimes (\gamma_iK_{m+3}) + \sum^{g}_{i=1}a_i\otimes (\delta_iK_{m+2})\right)\\
 & = \sum_{i=1}^{g}[a_i, h_{\#}(\beta_i)\beta^{-1}_i K_{m+2}]-\sum_{i=1}^g [b_i,h_{\#}(\alpha_i)\alpha^{-1}_i K_{m+3}]\\
 & = \Xi_m\tau^{\mathfrak{a}}_m(h).
\end{split}
\end{equation}
In the second equality of (\ref{JTH2equ181}) we use
$\delta_iK_{m+2}=\beta_i\delta_i\beta^{-1}_iK_{m+2}= h_{\#}(\beta_i)\beta^{-1}_iK_{m+2}$.
\end{proof}

Let us  briefly recall the Johnson homomorphisms and the Johnson-Levine homomorphisms.

\para{Johnson homomorphisms}  To define the Johnson filtration we use the lower central series $(\Gamma_m\pi)_{m\geq 1}$ of $\pi$. The associated graded Lie algebra of this filtration is
\begin{equation}
\text{Gr}(\Gamma_{\bullet}\pi) = \bigoplus_{m\geq 1}\frac{\Gamma_m\pi}{\Gamma_{m+1}\pi}\cong \bigoplus_{m\geq 1}\mathfrak{Lie}_m(H)=\mathfrak{Lie}(H),
\end{equation}
where $\mathfrak{Lie}(H)$ is the graded free Lie algebra on $H$.  The $m$-th \emph{Johnson homomorphism}
\begin{equation}\label{JTH2equ19}
\tau_m: J_m\mathcal{M}\longrightarrow \text{Hom}(H,\Gamma_{m+1}\pi/\Gamma_{m+2}\pi)\cong H^*\otimes \Gamma_{m+1}\pi/\Gamma_{m+2}\pi \cong H\otimes \mathfrak{Lie}_{m+1}(H),
\end{equation}
\noindent sends the isotopy class  $h\in J_m\mathcal{M}$ to the map
$x\mapsto h_{\#}(\tilde{x}){\tilde{x}}^{-1}\Gamma_{m+2}\pi$ for all $x\in H$, where $\tilde{x}\in\pi$ is any lift of $x$. The second isomorphism in  (\ref{JTH2equ19}) is given by  the identification $H\stackrel{\sim}{\longrightarrow} H^*$ that maps $x$ to $\omega(x, \cdot)$.
These homomorphisms were introduced by D.~Johnson in~\cite{MR579103,MR718141} and extensively studied by S.~Morita in~\cite{MR1133875,MR1224104}. In particular S.~Morita proved in~\cite[Corollary 3.2]{MR1224104} that the $m$-th Johnson homomorphism takes values in the kernel $D_m(H)$ of the Lie bracket  $\left[\ ,\ \right]: H\otimes\mathfrak{Lie}_{m+1}(H)\rightarrow\mathfrak{Lie}_{m+2}(H)$. Compare this with Theorem \ref{JTH2thm1}. From the definition it follows that $\ker(\tau_m)=J_{m+1}\mathcal{M}$.

\para{Johnson-Levine homomorphisms} J. Levine  defined and studied in \cite{MR1823501,MR2265877} a version of the Johnson homomorphisms for the Johnson-Levine filtration.  Identify $H/A$ with $A^*$ by sending $x+A\in H/A$ to  $\omega(x,\cdot)\in A^*$ and  $H/A$ with $H'$ via  the isomorphism $\iota_*$. The $m$-th \emph{Johnson-Levine homomorphism}
$$
\tau_m^L:J_m^L\mathcal{M}\rightarrow\text{Hom}(A,\Gamma_{m+1}\pi'/\Gamma_{m+2}\pi')\cong A^*\otimes \Gamma_{m+1}\pi'/\Gamma_{m+2}\pi'\cong H'\otimes \mathfrak{Lie}_{m+1}(H'),$$

\noindent is the group homomorphism  that sends $h\in J_m^L\mathcal{M}$ to the map $a\in A\mapsto \iota_{\#}h_{\#}(\alpha)\Gamma_{m+2}\pi'$, where $\alpha\in \mathbb{A}$ is  any lift of $\alpha$. Notice that here we consider the graded free Lie algebra $\mathfrak{Lie}(H')$ generated by $H'$.  J. Levine showed in \cite[Proposition 4.3]{MR1823501} that $\tau_m^L$ takes values in the kernel $D_m(H')$ of the Lie bracket $[\ ,\ ]:H'\otimes\mathfrak{Lie}_{m+1}(H')\rightarrow\mathfrak{Lie}_{m+2}(H')$.  Compare this with Theorem \ref{JTH2thm1}. From the definition it follows that $\ker(\tau^L_m)=J^L_{m+1}\mathcal{M}$.

We refer to the alternative Johnson homomorphisms, the Johnson-Levine homomorphisms and the Johnson homomorphisms as \emph{Johnson-type homomorphisms}.

\para{Alternative Johnson homomorphisms and Johnson-levine homomorphisms} In view of Proposition \ref{JTHproposition5}, for $m\geq 1$ we have $J^{\mathfrak{a}}_m\mathcal{M}\subseteq J^{L}_{m+1}\mathcal{M}$. We show that for $J^{\mathfrak{a}}_m\mathcal{M}$ the $m$-th alternative Johnson homomorphism determines the $(m+1)$-st Johnson-Levine homomorphism. Recall that $B=H'$.

\begin{lemma} For $m\geq 1$, there is a well defined homomomorphism $$\iota_*: D_m(B;A)\longrightarrow D_{m+1}(H').$$
\end{lemma}
\begin{proof}
It follows from Lemma \ref{JTHlemma3} that for $m\geq 1$ the map $\iota_{\#}:\pi\rightarrow \pi'$ induces a well-defined homomorphism
$$\iota_{*}:\mathfrak{Lie}_{m+2}(B;A)\cong\frac{K_{m+2}}{K_{m+3}}\longrightarrow \frac{\Gamma_{m+2}\pi'}{\Gamma_{m+3}\pi'}\cong \mathfrak{Lie}_{m+2}(H'),$$

\noindent which sends $xK_{m+3}$ to $\iota_{\#}(x)\Gamma_{m+3}\pi'$ for all $x\in K_{m+2}$. This map is compatible with the Lie bracket, in particular, the following diagram is commutative
\begin{equation*}
\xymatrix{\left(A\otimes\mathfrak{Lie}_{m+1}(B;A)\right)\oplus\left(B\otimes\mathfrak{Lie}_{m+2}(B;A)\right)\ar[r]^{\ \ \ \ \ \ \ \ \ \ \ \ \ \ \ \ \ \ \left[\ ,\ \right]}\ar[d]_{\iota_*\otimes\iota_*} & \mathfrak{Lie}_{m+3}(B;A) \ar[d]^{\iota_*} \\
						H'\otimes\mathfrak{Lie}_{m+2}(H')\ar[r]^{\ \ \left[\ ,\ \right]} & \mathfrak{Lie}_{m+3}(H').}
\end{equation*}
Whence, we have a well-defined homomorphism $\iota_*: D_m(B;A)\rightarrow D_{m+1}(H').$
\end{proof}

\begin{proposition}\label{JTH2prop11}
 For $m\geq 1$, the diagram 
\begin{equation*}
\xymatrix{  J^{\mathfrak{a}}_m\mathcal{M}\ar[r]^{\subset}\ar[d]_{\tau^{\mathfrak{a}}_m} & J_{m+1}^L\mathcal{M} \ar[d]^{\tau_{m+1}^L} \\
						D_m(B;A)\ar[r]^{\iota_*} & D_{m+1}(H')}
\end{equation*}
is commutative. In other words, for $J^{\mathfrak{a}}_m\mathcal{M}$,  the homomorphism $\tau^L_{m+1}$ is determined by the homomorphism $\tau^{\mathfrak{a}}_m$.
\end{proposition}
\begin{proof} Let $h\in J^{\mathfrak{a}}_m\mathcal{M}$. By considering the free basis $\{\alpha_i, \beta_i\}$ of $\pi$ and the induced symplectic basis $\{a_i, b_i\}$ of $H$, the $(m+1)$-st Johnson-Levine homomorphism on $h$ is given by
\begin{equation}\label{JTH2equ201}
\tau_{m+1}^L(h)=-\sum_{i=1}^g \iota_*(b_i)\otimes\left(\iota_{\#}h_{\#}(\alpha_i)\Gamma_{m+3}\pi'\right). 
\end{equation}
\noindent Applying $\iota_*:D_m(B;A)\rightarrow D_{m+1}(H')$ to Equation (\ref{JTH2equ14}) we obtain exactly the left-hand side of (\ref{JTH2equ201}), that is,  $\iota_*\tau^{\mathfrak{a}}_m(h)=\tau_{m+1}^L(h)$.
\end{proof}

\begin{remark} In general it is not easy to compare the alternative Johnson homomorphisms and the Johnson homomorphisms. In Lemma \ref{JTHLMOlemma6.10} we carry out the comparison between $\tau^{\mathfrak{a}}_1(\psi)$ and  $\tau_1(\psi)$ for $\psi\in\mathcal{I}\cap\mathcal{I}^{\mathfrak{a}}$.
\end{remark}

\subsection{Alternative Johnson homomorphism on \texorpdfstring{$\mathcal{L}$}{L}}\label{subsection5.3} In \cite{MR3828784}, K. Habiro and G. Massuyeau defined, in a general context, a group homomorphism on $\mathcal{L}$. In this subsection we study in detail this homomorphism, which we shall call here the \emph{$0$-th alternative Johnson homomorphism}. 

An \emph{automorphism} $\phi$ of  $\mathfrak{Lie}(B;A)$ is a family $\phi=(\phi_i)_{i\geq 1}$ of group isomorphisms
$\phi_i:\mathfrak{Lie}_i(B;A)\rightarrow \mathfrak{Lie}_i(B;A),$ such that $\phi_{i+j}([x,y])=[\phi_i(x), \phi_j(y)]$ for $x\in\mathfrak{Lie}_i(B;A)$ and $y\in \mathfrak{Lie}_j(B;A)$. Let $\text{Aut}\left(\mathfrak{Lie}(B;A)\right)$ denote the group of automorphisms of $\mathfrak{Lie}(B;A)$.

Recall that for the $N$-series $(K_i)_{i\geq1}$ defined in (\ref{JTHequ11}), we have from Lemma \ref{JTHlemma2} that for $h\in\mathcal{L}$, $^h(K_i)\subseteq K_i$ for $i\geq 1$. Here $^hx=h_{\#}(x)$ for $x\in K_i$. 

\begin{definition}\label{definitiontau00} The \emph{$0$-th alternative Johnson homomorphism} is the group homomorphism
\begin{equation}\label{JTH2equ20}
\tau^{\mathfrak{a}}_0:\mathcal{L}\longrightarrow \text{Aut}\left(\mathfrak{Lie}(B;A)\right)
\end{equation} 
which sends $h\in\mathcal{L}$ to the family $\tau^{\mathfrak{a}}_0(h)=(\tau^{\mathfrak{a}}_0(h)_i)_{i\geq1}$ where 
$$\tau^{\mathfrak{a}}_0(h)_i:\mathfrak{Lie}_{i}(B;A)\cong \frac{K_i}{K_{i+1}}\longrightarrow \frac{K_i}{K_{i+1}}\cong \mathfrak{Lie}_i(B;A)$$
is defined by $\tau^{\mathfrak{a}}_0(h)_i(xK_{i+1})=h_{\#}(x)K_{i+1}$ for $x\in K_i$.
\end{definition}

From the definition it follows that $\text{ker}(\tau^{\mathfrak{a}}_0)=J^{\mathfrak{a}}_1\mathcal{M}=\mathcal{I}^{\mathfrak{a}}$. We refer to \cite[Proposition~6.1]{MR3828784} for a proof of the homomorphism property with the above definition. We will see an equivalent definition of $\tau^{\mathfrak{a}}_0$ in (\ref{JTH2equ28}) and we prove the homomorphism property with the equivalent definition in Proposition \ref{JTH2prop14}. 

 Let us see how $\tau^{\mathfrak{a}}_0$ is related to the other alternative Johnson homomorphisms. First, for $m\geq 1$ there is an action of $\mathcal{L}$ on $J^{\mathfrak{a}}_m\mathcal{M}$ by conjugation, that is, for $h\in \mathcal{L}$, and $f\in J^{\mathfrak{a}}_m\mathcal{M}$, we set $^hf=hfh^{-1}\in J^{\mathfrak{a}}_m\mathcal{M}$. On the other hand, there is an action of $\text{Aut}(\mathfrak{Lie}(B;A))$ on the group $\text{Der}_m(\mathfrak{Lie}(B;A))$ of derivations of degree $m$ of $\mathfrak{Lie}(B;A)$. Let $\phi\in \text{Aut}(\mathfrak{Lie}(B;A))$ and $d\in \text{Der}_m(\mathfrak{Lie}(B;A))$, set
\begin{equation} 
^{\phi}d={\phi}d{\phi}^{-1}.
\end{equation}
More precisely, $(^{\phi}d)_i(x)={\phi}_{m+i}d_i{\phi}^{-1}_i(x)$ for $x\in\mathfrak{Lie}_i(B;A)$ and $i\geq 1$. The following is an instance of a part of \cite[Theorem 6.4]{MR3828784}.

\begin{proposition} Let $m\geq 1$ and $h\in\mathcal{L}$. The $m$-th alternative Johnson homomorphism $\tau^{\mathfrak{a}}_m: J^{\mathfrak{a}}_m\mathcal{M}\rightarrow \text{\emph{Der}}_m(B;A)$ satisfies the following invariance property: for every $f\in J^{\mathfrak{a}}_m\mathcal{M}$ we have
$$\tau^{\mathfrak{a}}_m({}{^h}f)=\  ^{\tau^{\mathfrak{a}}_0(h)}\tau^{\mathfrak{a}}_m(f).$$
\end{proposition}

\para{Understanding the image of $\tau^{\mathfrak{a}}_0$} Recall that $\mathfrak{Lie}(B;A)$ is the free Lie algebra generated by $B$ in degree $1$ and $A$ in degree $2$. For $h\in\mathcal{L}$ the automorphism
$\tau^{\mathfrak{a}}_0(h):\mathfrak{Lie}(B;A)\rightarrow  \mathfrak{Lie}(B;A)$ is completely determined by its parts of degree $1$ and degree $2$:
\begin{equation}
\tau^{\mathfrak{a}}_0(h)_1: \mathfrak{Lie}_1(B;A) \rightarrow  \mathfrak{Lie}_1(B;A) \ \ \ \ \ \text{ and } \ \ \ \ \ \tau^{\mathfrak{a}}_0(h)_2: \mathfrak{Lie}_2(B;A) \rightarrow \mathfrak{Lie}_2(B;A),
\end{equation}

\noindent but $\mathfrak{Lie}_1(B;A)=B$ and $\mathfrak{Lie}_2(B;A)=\Lambda^2B \oplus A $. Consider the subgroup $\mathcal{P}$ of $\text{Aut}(B)\times \text{Aut}(\mathfrak{Lie}_2(B;A))$ defined by
\begin{equation*}
\mathcal{P}=\left\{(u,v)\in \text{Aut}(B)\times \text{Aut}(\mathfrak{Lie}_2(B;A))\ | \  v\left([x,y]\right)=[u(x),u(y)]\ \ \forall x,y\in B\right\}.
\end{equation*}
Besides, consider the set
\begin{equation*}
\mathcal{D} = \Big\{(u,v)\in \text{Aut}(B)\times\text{Hom}(A, \mathfrak{Lie}_2(B;A))
 \left|\begin{smallmatrix} \text{ the map }  [x,y]\   + \ a\  \longmapsto \ [u(x), u(y)]\ +\  v(a) \\ 
\text{is an automorphism of }  [B, B]\oplus A\  =\ \mathfrak{Lie_2}(B;A)  \end{smallmatrix}\right\}.
\end{equation*}

\noindent For $(u,v)\in\mathcal{D}$, define $\tilde{v}:\mathfrak{Lie}_2(B,A)\rightarrow \mathfrak{Lie}_2(B,A)$  on the summands by 
$$\tilde{v}\left([x,y]\right) = [u(x), u(y)]\ \ \  \text{ and } \ \ \    \tilde{v}(a)=v(a),$$
for $x,y\in B$ and $a\in A$. 

\begin{lemma} There is a bijective correspondence 
$\Phi:\mathcal{P}\rightarrow \mathcal{D}$,
which sends $(u,v)\in\mathcal{P}$ to $(u,v_{|A})$. This way, $\mathcal{D}$ inherits a group structure from $\mathcal{P}$, so that the  the product in $\mathcal{D}$ is given by   
\begin{equation}\label{JTH2equ25}
(u_1,v_1)(u_2,v_2)=\Phi\left((u_1,\tilde{v}_1)(u_2,\tilde{v}_2)\right) = \Phi(u_1u_2,\tilde{v}_1\tilde{v}_2) = \left(u_1u_2,(\tilde{v}_1\tilde{v}_2)_{|A}\right)
\end{equation}
for $(u_1,v_1)$, $(u_2,v_2)\in\mathcal{D}$.
\end{lemma}
\begin{proof}

The inverse of $\Phi$ is defined by $\Phi^{-1}(u,v)=(u,\tilde{v})$ for $(u,v)\in\mathcal{D}$.
\end{proof}

If $h\in\mathcal{L}$ we have $\left(\tau^{\mathfrak{a}}_0(h)_1, (\tau^{\mathfrak{a}}_0(h)_2)_{|A} \right)\in\mathcal{D}$. Hence we can see the $0$-th alternative Johnson homomorphism $\tau^{\mathfrak{a}}_0$ as taking values in $\mathcal{D}$. 

We can still improve the target of $\tau^{\mathfrak{a}}_0$. First, recall that if $h\in\mathcal{L}$, the induced map $h_*:H\rightarrow H$ is symplectic, see (\ref{JTHequ2}). Denote by $\hat{h}_*: H/A \rightarrow H/A$ the homomorphism induced by $h_*$. Hence  $\tau^{\mathfrak{a}}_0(h)_1=\iota_*\hat{h}_*\iota^{-1}_*:H'\rightarrow H'$. The symplectic condition on $h_*$ implies  that there is some information of $(\tau^{\mathfrak{a}}_0(h)_2)_{|A}$  which is already encoded in $\tau^{\mathfrak{a}}_0(h)_1$. More precisely, we have
\begin{equation}
\left(\tau^{\mathfrak{a}}_0(h)_2\right)_{|A}=\iota_*\left(\tau^{\mathfrak{a}}_0(h)_2\right)_{|A} + {h_*}_{|A}\in \text{Hom}(A, \Lambda^2B\oplus A)=\text{Hom}(A, \mathfrak{Lie}_2(B;A)),
\end{equation}
where $\iota_*: \Lambda^2B\oplus A \rightarrow \Lambda^2 B$ denotes the projection on $\Lambda^2 B$. For the moment we have that $\tau^{\mathfrak{a}}_0(h)$ is completely determined by the pair
$$\left(\tau^{\mathfrak{a}}_0(h)_1, \iota_*\left(\tau^{\mathfrak{a}}_0(h)_2\right)_{|A}\right)\in \text{Aut}(B)\times \text{Hom}(A, \Lambda^2 B),$$
because ${h_*}_{|A}$ and $\tau^{\mathfrak{a}}_0(h)_1=\iota_*\hat{h}_*\iota^{-1}_*$ determine each other by the symplectic condition on $h_*$. The set $\text{Aut}(B)\times \text{Hom}(A, \Lambda^2 B)$ inherits the group structure (\ref{JTH2equ25}) from $\mathcal{D}$, which can be described explicitly as follows. Let $h\in\text{Aut}(B)$. Using the identification $H'=B\cong A^*$, described in subsection \ref{sec5.1prel}, we obtain $h''\in\text{Aut}(A^*)$. Denote  by $h'$ the automorphism of $A$ such that $(h')^*=h''$. Let $\mu\in\text{Hom}(A,\Lambda^2 B)$. We consider the following actions of $\text{Aut}(B)$ on $\text{Hom}(A,\Lambda^2B)$
\begin{equation}\label{JTH2equa28}
\begin{split}
\text{right action: } & \ \ \ \mu\cdot h := \mu\circ h'\in \text{Hom}(A, \Lambda^2B),\\
\text{left action: }& \ \ \  h\cdot\mu := \Lambda^2 h\circ\mu\in \text{Hom}(A, \Lambda^2B).
\end{split}
\end{equation}
We have that the product in $\text{Aut}(B)\times \text{Hom}(A, \Lambda^2 B)$ inherited from (\ref{JTH2equ25}) is given by
\begin{equation}\label{JTH2equ27}
(h, \mu)(f, \nu) = (hf, h\cdot\nu + \mu\cdot f)
\end{equation}
for $h,f\in\text{Aut}(B)$ and $\mu,\nu\in\text{Hom}(A,\Lambda^2 B)$.

\begin{remark} There is a right split short exact sequence

\medskip

\begin{tikzcd}
    0\arrow{r} & \text{Hom}(A, \Lambda^2 B) \arrow{r}{j} & \text{Aut}(B)\times \text{Hom}(A, \Lambda^2 B) \arrow{r}{p} & \text{Aut}(B)\arrow{r}\arrow[bend left=20]{l}{s} & 1, 
\end{tikzcd}

\medskip

\noindent where $j(\mu)=(\text{Id}_B, \mu)$, $p(h,\mu)=h$ and $s(h)=(h,0)$ for $\mu\in\text{Hom}(A, \Lambda^2B)$ and $h\in\text{Aut}(B)$. Therefore, the group $\text{Aut}(B)\times \text{Hom}(A, \Lambda^2 B)$ is isomorphic to the semidirect product $\text{Aut}(B)\ltimes \text{Hom}(A, \Lambda^2 B)$, where the action of $\text{Aut}(B)$ on $\text{Hom}(A, \Lambda^2B)$ is given by $h*\mu=h\cdot\mu\cdot h^{-1}$ for $h\in\text{Aut}(B)$ and $\mu\in\text{Hom}(A, \Lambda^2 B)$. Here the $\cdot$ means the left and right actions defined in (\ref{JTH2equa28}). The explicit isomorphism 
$$\Theta: \text{Aut}(B)\times \text{Hom}(A, \Lambda^2 B)\longrightarrow \text{Aut}(B)\ltimes \text{Hom}(A, \Lambda^2 B), $$
is given by $\Theta(h,\mu)=(h, \mu\cdot h^{-1})$ for $h\in\text{Aut}(B)$ and $\mu\in\text{Hom}(A, \Lambda^2 B)$.
\end{remark}

From now on, by simplicity on the proofs, we consider the set $\text{Aut}(B)\times \text{Hom}(A, \Lambda^2 B)$ with the product given in~(\ref{JTH2equa28}).
This way, we  can see the $0$-th Johnson homomorphism~(\ref{JTH2equ20}) as the map
\begin{equation}\label{JTH2equ28}
\tau^{\mathfrak{a}}_0: \mathcal{L}\longrightarrow \text{Aut}(B)\times \text{Hom}(A, \Lambda^2 B), 
\end{equation}
which sends $h\in\mathcal{L}$ to the pair $\left(\tau^{\mathfrak{a}}_0(h)_1, \iota_*\left(\tau^{\mathfrak{a}}_0(h)_2\right)_{|A}\right)$. Moreover,  the homomorphism $\iota_*\left(\tau^{\mathfrak{a}}_0(h)_2\right)_{|A}$ takes $a\in A$ to $\iota_{\#}h_{\#}(\alpha)\Gamma_3\pi'$ where $\alpha\in\mathbb{A}$ is any lift of $a$, and we identify $\Gamma_2\pi'/\Gamma_3\pi'$ with $\Lambda^2H'=\Lambda^2B$. In this context we can show the homomorphism property of $\tau^{\mathfrak{a}}_0$.

\begin{proposition}\label{JTH2prop14}
The map $\tau^{\mathfrak{a}}_0: \mathcal{L}\rightarrow \text{\emph{Aut}}(B)\times \text{\emph{Hom}}(A, \Lambda^2 B)$ is a group homomorphism and its kernel is the second term $J^L_2\mathcal{M}$ of the Johnson-Levine filtration. In particular we have $\mathcal{I}^{\mathfrak{a}}=J^{\mathfrak{a}}_1\mathcal{M}=J^L_2\mathcal{M}$.
\end{proposition}
\begin{proof}
Let $h,f\in \mathcal{L}$. Clearly we have $\tau^{\mathfrak{a}}_0(hf)_1 = \tau^{\mathfrak{a}}_0(h)_1\tau^{\mathfrak{a}}_0(f)_1$. Identify $\Gamma_2\pi'/\Gamma_3\pi'$ with $\Lambda^2 B$. Set $\mu = \iota_*\left(\tau^{\mathfrak{a}}_0(h)_2\right)_{|A}$, $\nu = \iota_*\left(\tau^{\mathfrak{a}}_0(f)_2\right)_{|A}$ and $\kappa = \iota_*\left(\tau^{\mathfrak{a}}_0(hf)_2\right)_{|A}$. Let us see that 
$$\kappa = \tau^{\mathfrak{a}}_0(h)_1\cdot \nu + \mu\cdot  \tau^{\mathfrak{a}}_0(f)_1.$$

\noindent Let $a\in A$ and $\alpha\in\mathbb{A}$ with $\text{ab}(\alpha)=a$. By Lemma \ref{JTHlemma2} we can write $f_{\#}(\alpha)=\beta y$ with $\beta\in\mathbb{A}$ and $y\in\Gamma_2\pi$. We have
$$\text{ab}(\beta)=\text{ab}(\beta y)=\text{ab}(f_{\#}(\alpha))=f_*(\text{ab}(\alpha))=f_*(a).$$

\noindent Hence
\begin{equation*}
\begin{split}
\kappa & = \iota_{\#}(h_{\#}(f_{\#}(\alpha)))\Gamma_{3}\pi'\\
																 & = \iota_{\#}(h_{\#}(\beta))\Gamma_{3}\pi'+\iota_{\#}(h_{\#}(y))\Gamma_{3}\pi'\\
																 & = \mu(\text{ab}(\beta))+\iota_{\#}(h_{\#}(y))\Gamma_{3}\pi'\\
																 & = \mu(f_*(a)) + \Lambda^2h(\nu(a))\\
																 & = (\mu\cdot \tau^{\mathfrak{a}}_0(f)_1)(a) + (\tau_0(h)_1\cdot\nu)(a)\\
																 & = (\tau_0(h)_1\cdot\nu)(a) + (\mu\cdot \tau^{\mathfrak{a}}_0(f)_1)(a).
\end{split}
\end{equation*}

\medskip

\noindent Whence $\kappa = \tau^{\mathfrak{a}}_0(h)_1\cdot \nu + \mu\cdot  \tau^{\mathfrak{a}}_0(f)_1$. Thus $\tau^{\mathfrak{a}}_0: \mathcal{L}\rightarrow \text{Aut}(B)\times \text{Hom}(A, \Lambda^2 B)$ is a group homomorphism. Now, let $h\in\text{ker}(\tau^{\mathfrak{a}}_0)$, thus $\tau_0(h)_1=\text{Id}_{H'}$.  From the symplectic condition we have ${h_*}_{|A} = \text{Id}_A$, so $h\in \mathcal{I}^L$.  Let $\alpha\in\mathbb{A}$, hence
$$\iota_{\#}h_{\#}(\alpha)\Gamma_3\pi'=\iota_*\left(\tau^{\mathfrak{a}}_0(h)_2\right)(\text{ab}(\alpha))=\Gamma_3\pi',$$

\noindent that is,  $\iota_{\#}h_{\#}(\alpha)\in\Gamma_3\pi'$ for all $\alpha\in\mathbb{A}$, so that  $h\in J_2^L\mathcal{M}$.
\end{proof}

The next proposition follows from the definition (\ref{JTH2equ28}) of $\tau^{\mathfrak{a}}_0$ and shows that for the elements of $\mathcal{I}^L$, the $0$-th alternative Johnson homomorphism determines the first Johnson-Levine homomorphism.

\begin{proposition}\label{propositionTLTA} Let $q: \text{\emph{Aut}}(B)\times \text{\emph{Hom}}(A, \Lambda^2 B)\rightarrow \text{\emph{Hom}}(A, \Lambda^2 B)$ denote the cartesian projection (which is not a group homomorphism). Then, the diagram 
\begin{equation*}
\xymatrix{  \mathcal{I}^L\ar[r]^{\subset}\ar[d]_{\tau^{L}_1} & \mathcal{L} \ar[d]^{\tau^{\mathfrak{a}}_{0}} \\
						\text{\emph{Hom}}(A, \Lambda^2 B) &  \text{\emph{Aut}}(B)\times \text{\emph{Hom}}(A, \Lambda^2 B)\ar[l]^{q\ \ \ \ \ }}
\end{equation*}
is commutative.
\end{proposition}

We can do yet another refinement of the target of $\tau^{\mathfrak{a}}_{0}$. Considering the Definition \ref{definitiontau00} we have that $\tau^{\mathfrak{a}}_0(h)_3(\Omega')=\Omega'$ for $h\in\mathcal{L}$, where $\Omega'\in\mathfrak{Lie}_3(B;A)$ is determined by the intersection form (\ref{wthbasis2}).

Notice that a pair $(h,\kappa)\in \text{Aut}(B)\times \text{Hom}(A, \mathfrak{Lie}_2(B;A))$ uniquely determines  a morphism of Lie algebras $(h,\kappa):\mathfrak{Lie}(B;A)\rightarrow\mathfrak{Lie}(B;A)$.

\begin{lemma}\label{lemmaJTH2omega} Let $h\in\text{\emph{Aut}}(B)$ and let $h':A\rightarrow A\subseteq \mathfrak{Lie_2}(B;A)$ be  the automorphism of $A$ determined by $h$. Then $(h,h')(\Omega')=\Omega'$.
\end{lemma}
\begin{proof}
Consider the bases of $H$ and $B$ as in Proposition \ref{propwithbs} and identify $\iota_*(b_i)$ with $b_i$. Hence $\Omega'$ is given as in Equality (\ref{JTH2equ6}). If $R=(\epsilon_{kj})$ is the matrix of $h$ in the basis $\{\iota_*(b_i)\}$ and $P=(\lambda_{ij})$ is the the matrix of $h'$ in the basis $\{a_i\}$, then $P^TR=\text{Id}_g$. Thus
\begin{equation*}
\begin{split}
(h,h')(\Omega') & = \sum_{j=1}^g [h'(a_j), h(b_j)]
 = \sum_{j=1}^g \left[ \sum_{i=1}^g\lambda_{ij}a_i, \sum_{k=1}^g\epsilon_{kj}b_k\right]
  = \sum_{i=1}^g\sum_{k=1}^g\sum_{j=1}^g\lambda_{ij}\epsilon_{kj}[a_i,b_k]\\
 & = \sum_{j=1}^g [a_j, b_j]=\Omega'.
\end{split}
\end{equation*}
\end{proof}

Let $(h,\mu)\in \text{Aut}(B)\times \text{Hom}(A, \Lambda^2 B)$. Let $(h_i)_{i\geq 1}$ be the associated automorphism of~$\mathfrak{Lie}(B;A)$. Explicitly we have $h_1=h$ and $h_2= \mu+ h'$ where $h'\in\text{Aut}(A)$ is determined by $h$.

\begin{lemma}\label{lemmasympcond} The condition $h_3(\Omega')=\Omega'$ holds if and only if $(h,\mu)(\Omega')=0$. 
\end{lemma}
\begin{proof}
We use bases as in Lemma \ref{lemmaJTH2omega}. We have
\begin{equation*}
\begin{split}
h_3(\Omega') = \sum_{j=1}^g[h_2(a_j), h_1(b_j)]  & =  \sum_{j=1}^g[\mu(a_j), h(b_j)] + \sum_{j=1}^g[h'(a_j), h(b_j)]\\
& = (h,\mu)(\Omega') + (h,h')(\Omega')= (h,\mu)(\Omega') + \Omega'.
\end{split}
\end{equation*}
Last equality comes from Lemma \ref{lemmaJTH2omega}. Whence the desired result.
\end{proof}

Notice that a pair $(h,\mu)\in \text{Aut}(B)\times \text{Hom}(A, \Lambda^2 B)$, determines an  element $\mu_h\in B\otimes \Lambda^2 B$ through  the identification 
$$\text{Hom}(A,\Lambda^2B)\cong A^*\otimes\Lambda^2 B\cong B\otimes \Lambda^2 B\xrightarrow{\ h\otimes \text{Id}_{\Lambda^2 B}\ } B\otimes \Lambda^2 B,$$
Thus $\mu_h:=(h\otimes \text{Id}_{\Lambda^2 B})(\mu)$. Set 
\begin{equation}\mathcal{G}:=\left\{(h,\mu)\in \text{Aut}(B)\times \text{Hom}(A, \Lambda^2 B) \ \ | \ \  \Xi_3(\mu_h)=0 \right\},
\end{equation}
where $\Xi_3:B\otimes \mathfrak{Lie}_2(B)\rightarrow \mathfrak{Lie}_3(B)$ is the Lie bracket. Using bases as in Lemma \ref{lemmaJTH2omega} we have  
\begin{equation}\label{groupGwithbasis}
\mathcal{G}=\left\{(h,\mu)\in \text{Aut}(B)\times \text{Hom}(A, \Lambda^2 B) \ \ | \ \  \Xi_3\Big(\sum_{j=1}^{g}h(b_j)\otimes\mu(a_j)\Big)=0 \right\}.
\end{equation}

\begin{proposition}
The set $\mathcal{G}$ is a subgroup of $\text{\emph{Aut}}(B)\times \text{\emph{Hom}}(A, \Lambda^2 B)$.
\end{proposition}

\begin{proof}
The result can be deduced from Lemma \ref{lemmasympcond} or from the description of $\mathcal{G}$ given in (\ref{groupGwithbasis}) as follows. Let $(h,\mu), (f,\nu)\in \mathcal{G}$. Let us see that $(h,\mu)(f,\nu)=(hf, h\cdot\nu + \mu\cdot f)$ and $(h,\mu)^{-1}=(h^{-1}, -h^{-1}\cdot \mu \cdot h^{-1} )$  belong to $\mathcal{G}$. We have
\begin{equation*}
\begin{split}
\Xi_3\Big(\sum_{j=1}^ghf(b_j)\otimes (h\cdot\nu + \mu\cdot f)(a_j)\Big) & = \sum_{j=1}^g\Big[hf(b_j), \Lambda^2 h(\nu(a_j)) + \mu f'(a_j)\Big]\\
& = \mathfrak{Lie}_3(h)\Big(\sum_{j=1}^g[f(b_j), \nu(a_j)]\Big) \\
& \ \ \ \ \ \ \ + \sum_{j=1}^g [h(f(b_j)), \mu(f'(a_j))]\\
 & = \sum_{j=1}^g [h(b_j), \mu(a_j)]\\
 & = 0, 
\end{split}
\end{equation*}
where $\mathfrak{Lie}_3(h):\mathfrak{Lie}_3(B)\rightarrow\mathfrak{Lie}_3(B)$ is the isomorphism induced by $h$. The equality  $\sum_{j=1}^g [h(f(b_j)), \mu(f'(a_j))] = \sum_{j=1}^g [h(b_j), \mu(a_j)]$ is deduced in a similar way as we did in the proof of Lemma \ref{lemmaJTH2omega}. Therefore $(h,\mu)(f,\nu)\in\mathcal{G}$. On the other hand
\begin{equation*}
\begin{split}
\Xi_3\Big(\sum_{j=1}^g h^{-1}(b_j)\otimes (-h^{-1}\cdot\mu\cdot h^{-1})(a_j)\Big) & = \sum_{j=1}^g\Big[h^{-1}(b_j), -\Lambda^2 h^{-1}(\mu(h^{-1})'(a_j))\Big]\\
 & = -\mathfrak{Lie}_3(h^{-1})\Big(\sum_{j=1}^g [h (h^{-1}(b_j)),\mu (h^{-1})'(a_j)] \Big)\\
 & = -\mathfrak{Lie}_3(h^{-1})\Big(\sum_{j=1}^g [h (b_j),\mu (a_j)] \Big)\\
 & =0.
\end{split}
\end{equation*}
Hence $(h,\mu)^{-1}\in\mathcal{G}$.
\end{proof}

\begin{lemma} The $0$-th alternative Johnson homomorphism defined in (\ref{JTH2equ28}) takes its  values in $\mathcal{G}$.
\end{lemma}
\begin{proof} Let $h\in\mathcal{L}$. Set $\mu = \iota_*\left(\tau^{\mathfrak{a}}_0(h)_2\right)_{|A}$. Hence
\begin{equation*}\begin{split}
\Xi_3\Big(\sum_{j=1}^g\tau^{\mathfrak{a}}_0(h)_1(b_j)\otimes \mu(a_j)\Big) & = -\sum_{j=1}^{g}\Xi_3\Big(\iota_*h_*(-b_j)\otimes \iota_{\#}h_{\#}(\alpha_j)\Gamma_3\pi'\Big)\\
& = -\sum_{j=1}^g\Xi_3\Big(\iota_{\#}h_{\#}(\beta^{-1}_j)\Gamma_2\pi'\otimes\iota_{\#}h_{\#}(\alpha_j)\Gamma_3\pi'\Big)\\
& =\left(\iota_{\#}h_{\#}\Big(\prod_{j=1}^g[\beta^{-1}_j,\alpha_j]\Big)\Gamma_4\pi'\right)^{-1}\\
& = \Gamma_4\pi'\\
& =0\in\mathfrak{Lie}_3(B).
\end{split}
\end{equation*}
\end{proof}

To sum up, we can write 
\begin{equation}\label{JTH2equ30}
\tau^{\mathfrak{a}}_0: \mathcal{L}\longrightarrow \mathcal{G}.
\end{equation}

\begin{theorem} The $0$-th alternative Johnson homomorphism $\tau^{\mathfrak{a}}_0: \mathcal{L}\rightarrow \mathcal{G}$ is surjective.
\end{theorem}

\begin{proof}
Notice that the diagram
\begin{equation*}
\xymatrix{  \mathcal{I}\ar[r]^{\subset}\ar[d]_{\tau_1} & \mathcal{I}^L \ar[d]^{\tau_{1}^L} \\
						D_1(H)\ar[r]^{\iota_*} & D_{1}(H')}
\end{equation*}
is commutative, this can be shown by writing $\tau_1$ and $\tau^L_1$ by using a symplectic basis as we did for $\tau^L_{m+1}$ in Equation (\ref{JTH2equ201}), see \cite[Section 4]{MR2265877} for more details. The map $\iota_*:D_1(H)\rightarrow D_1(H')$ is induced by $\iota_*:H\rightarrow H'$. It is easy to show that $\iota_*:D_1(H)\rightarrow D_1(H')$ is surjective and it is well known that the first Johnson homomorphism $\tau_1$ is surjective \cite[Theorem 1]{MR579103}. Hence $\tau^{L}_1$ is surjective. Using the symplectic basis $\{a_i,b_i\}$ to identify $\text{Sp}(H)$ with $\text{Sp}(2g,\mathbb{Z})$ we have that the image of  $\mathcal{L}$ under the symplectic representation  (\ref{JTHequ2}) is
\begin{equation}\label{JTH2thm21equ1}
\sigma(\mathcal{L})= \left\{\left( \begin{smallmatrix} P &Q\\ 
0& (P^T)^{-1}\end{smallmatrix} \right) \  \bigg\rvert \ \ P^{-1}Q \text{\ is symmetric}  \right\}\subseteq\text{Sp}(2g,\mathbb{Z}).
\end{equation}
Let  $(f,\mu)\in\mathcal{G}$. From (\ref{JTH2thm21equ1}), it follows that there is $h\in\mathcal{L}$ such that $\tau^{\mathfrak{a}}_0(h)_1 = f$. Let $\nu= \iota_*\left(\tau^{\mathfrak{a}}_0(h)_2\right)_{|A}\in\text{Hom}(A,\Lambda^2 B)$. Hence $\tau^{\mathfrak{a}}_0(h)=(f,\nu)\in\mathcal{G}$.

\noindent Consider the element $$\mu'=f^{-1}\cdot (\mu-\nu)\in\text{Hom}(A, \Lambda^2 H').$$
 Recall that  $B=H'$. Let us see that $\mu'\in D_1(H')$. Indeed, if $\Xi: H'\otimes \mathfrak{Lie}_2(H')\rightarrow \mathfrak{Lie}_3(H')$ denotes the Lie bracket, we have
\begin{equation*}
\begin{split}
\Xi(\mu') & = -\Xi\Big(\sum_{j=1}^{g}b_j\otimes (f^{-1}\cdot (\mu-\nu))(a_j)\Big)\\
& = -\sum_{j=1}^g [b_j, \Lambda^2 f^{-1}\mu(a_j)] + \sum_{j=1}^g[b_j, \Lambda^2 f^{-1}\nu(a_j)]\\
& -\sum_{j=1}^g [f^{-1}(f(b_j)), \Lambda^2 f^{-1}\mu(a_j)] + \sum_{j=1}^g[f^{-1}(f(b_j)), \Lambda^2 f^{-1}\nu(a_j)]\\
& = -\mathfrak{Lie}_3(f^{-1})\Big(\sum_{j=1}[f(b_j), \mu(a_j)]\Big) + \mathfrak{Lie}_3(f^{-1})\Big(\sum_{j=1}^g[f(b_j),\nu(a_j)]\Big)\\
 & = 0.
\end{split}
\end{equation*} 
Whence $\mu'\in D_1(H')$. By the surjectivity of $\tau^L_1$ and Proposition \ref{propositionTLTA}, there exists $g\in\mathcal{I}^{L}$ such that 
$\tau^{\mathfrak{a}}_0(g)=(\text{Id}_{H'},\mu')$. Therefore
\begin{equation*}
\begin{split}
\tau^{\mathfrak{a}}_0(hg) & = \tau^{\mathfrak{a}}_0(h)\tau^{\mathfrak{a}}_0(g)
 = (f,\nu)(\text{Id}_{H'},\mu')
 = (f,\nu)(\text{Id}_{H'}, f^{-1}\cdot(\mu-\nu))\\
 & = (f, f\cdot (f^{-1}\cdot(\mu-\nu)) + \nu)= (f,\mu).
\end{split}
\end{equation*}
Hence we have the surjectivity of $\tau^{\mathfrak{a}}_0:\mathcal{L}\rightarrow\mathcal{G}$.
\end{proof}

\begin{corollary}\label{corotau0} We have the following  short exact sequence
$$1\longrightarrow \mathcal{I}^{\mathfrak{a}}\xrightarrow{\ \subset\ } \mathcal{L}\xrightarrow{ \ \tau_0^{\mathfrak{a}}\  } \mathcal{G}\longrightarrow 1.$$
\end{corollary}

\subsection{Diagrammatic versions of the Johnson-type homomorphisms}\label{subsection5.4} In subsection \ref{sub5.2} we have seen that for $m\geq 1$, the $m$-th  Johnson homomorphism, the $m$-th  Johnson-Levine homomorphism and the  $m$-th  alternative Johnson homomorphism take values in the abelian groups $D_m(H)$, $D_m(H')=D_m(B)$ and $D_m(B;A)$, respectively. These spaces were defined as
\begin{equation*}
\begin{split}
D_m(H) & = \text{ker}\big(H\otimes\mathfrak{Lie}_{m+1}(H)\xrightarrow{\ [\ , \ ]\ }\mathfrak{Lie}_{m+2}(H)\big),\\
D_m(H') & = \text{ker}\big(H'\otimes\mathfrak{Lie}_{m+1}(H')\xrightarrow{\ [\ ,\ ]\ }\mathfrak{Lie}_{m+2}(H')\big) \ \text{ and } \\
D_m(B;A) & = \text{ker}\big((A\otimes\mathfrak{Lie}_{m+1}(B;A))\oplus (B\otimes\mathfrak{Lie}_{m+2}(B;A))\xrightarrow{\ [\ ,\ ]\ }\mathfrak{Lie}_{m+3}(B;A)\big).
\end{split}
\end{equation*}

The rational versions $D_m(H)\otimes\mathbb{Q}$, $D_m(H')\otimes\mathbb{Q}$ and $D_m(B;A)\otimes\mathbb{Q}$  can be interpreted as subspaces of the spaces of connected tree-like Jacobi diagrams $\mathcal{A}^{t,c}(H)$, $\mathcal{A}^{t,c}(H')$ and $\mathcal{A}^{t,c}(B\oplus A)$, respectively. See Example \ref{ejemplo2JD} for the definitions.  Recall that these spaces are graded by the internal degree.  Notice that as spaces $\mathcal{A}^{t,c}(H) = \mathcal{A}^{t,c}(B\oplus A)$ but we would like to give a special role to $A$ in the latter space, which will be reflected in a different grading of the space  $\mathcal{A}^{t,c}(B\oplus A)$. Let us start by recalling this interpretation.

For a connected tree-like Jacobi diagram $T$ in $\mathcal{A}^{t,c}(H)= \mathcal{A}^{t,c}(B\oplus A)$ or in $\mathcal{A}^{t,c}(H')$, set
\begin{equation}\label{JTH2equ31}
\eta(T)=\sum_v \text{color}(v)\otimes (T \text{ rooted at } v),
\end{equation}
where the sum ranges over the set of legs (univalent vertices) of $T$ and we interpret a rooted tree as a Lie commutator.

\begin{example}\label{examplet0}
Consider the tree
\begin{center}
\includegraphics[scale=0.88]{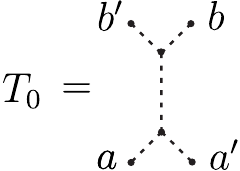}		
\end{center}

\noindent where  $a,a'\in A$ and $b,b'\in B$. Hence,  
\medskip
\begin{center}
\includegraphics[scale=0.85]{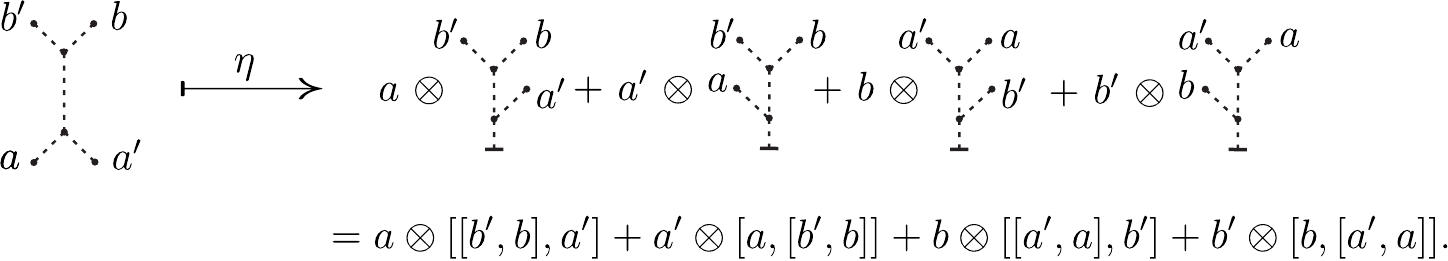}		
\end{center}
\medskip
We have that $\eta(T_0)\in H\otimes\mathfrak{Lie}_3(H)$ and $\eta(T_0)\in \big(A\otimes\mathfrak{Lie}_4(B;A)\big)\oplus\big(B\otimes\mathfrak{Lie}_5(B;A)\big)$. Moreover, by  the Jacobi identity, if we apply the Lie bracket $\Xi$ to $\eta(T_0)$ we obtain $\Xi\eta(T_0)=0$. Therefore $\eta(T_0)\in D_2(H)$ and $\eta(T_0)\in D_3(B;A)$. 
\end{example}

Denote by $\mathcal{A}^{t,c}_m(H)$ the subspace of $\mathcal{A}^{t,c}(H)$ generated by diagrams of internal degree $m$. So if $T\in \mathcal{A}^{t,c}_m(H)$, then $T$ has $m+2$ legs and therefore by rooting $T$ at one of its legs we obtain a rooted tree with $m+1$ leafs. To sum up, $\eta(T)\in H\otimes\mathfrak{Lie}_{m+1}(H)$. Moreover if we apply the Lie bracket $\Xi:H\otimes\mathfrak{Lie}_{m+1}(H)\rightarrow\mathfrak{Lie}_{m+2}(H)$ to $\eta(T)$, we obtain $\Xi\eta(T)=0$.  This way $\eta(T)\in D_m(H)$, see \cite[Lemma 3.1]{MR1943338} for the proof in the general case. The following result is well known.

\begin{theorem}\label{JTH2thm23}
For $m\geq 1$ the map
\begin{equation}
\eta: \mathcal{A}^{t,c}_m(H)\longrightarrow D_m(H)\otimes \mathbb{Q},
\end{equation}
defined as in Equality (\ref{JTH2equ31}), is an isomorphism of $\mathbb{Q}$-vector spaces.
\end{theorem}
\noindent We refer to \cite[Corollary 3.2]{MR1943338} or \cite[Theorem 1]{MR1983089} for a proof of Theorem \ref{JTH2thm23}. 

In particular we have an isomorphism of graded $\mathbb{Q}$-vector spaces
\begin{equation}\label{JTH2cor24}
\eta:\bigoplus_{m\geq 1}\mathcal{A}^{t,c}_m(H)\longrightarrow \bigoplus_{m\geq 1} D_m(H)\otimes\mathbb{Q}.
\end{equation}
 
The same statements hold replacing $H$ by $H'$. We define  a degree for connected tree-like Jacobi diagrams, which we call \emph{alternative degree} and denote by $\mathfrak{a}$-deg, such that if $T\in \mathcal{A}^{t,c}(B\oplus A)$ is such that $\mathfrak{a}$-deg$(T)=m$ then $\eta(T)\in D_m(B;A)\otimes \mathbb{Q}$. In Example \ref{examplet0}, $\eta(T_0)\in D_3(B;A)$, so we want $\mathfrak{a}$-deg$(T_0)=3$. 

\begin{definition} Let $T$ be a connected tree-like Jacobi diagram with legs colored by $B\oplus A$. The \emph{alternative degree} of $T$, denoted $\mathfrak{a}$-deg$(T)$, is defined as 
$$\mathfrak{a}\text{-deg}(T)=2\#\{A\text{-colored legs of } T\} + \#\{B\text{-colored legs of } T\}-3.$$
Here $\#S$ denotes the cardinal of the set $S.$
\end{definition}

In Table  \ref{table1} we show some examples of tree-like Jacobi diagrams organized by their internal degree  in the columns and by the alternative degree in the rows. The legs colored by $+$ (respectively by $-$) in the diagrams represent legs colored by elements of $B$ (respectively of $A$). Notice that a strut diagram $D$ whose both legs are colored by elements of $B$ is such that $\mathfrak{a}$-deg$(D)=-1$.

 \begin{table}
\begin{center}
\includegraphics[scale=0.75]{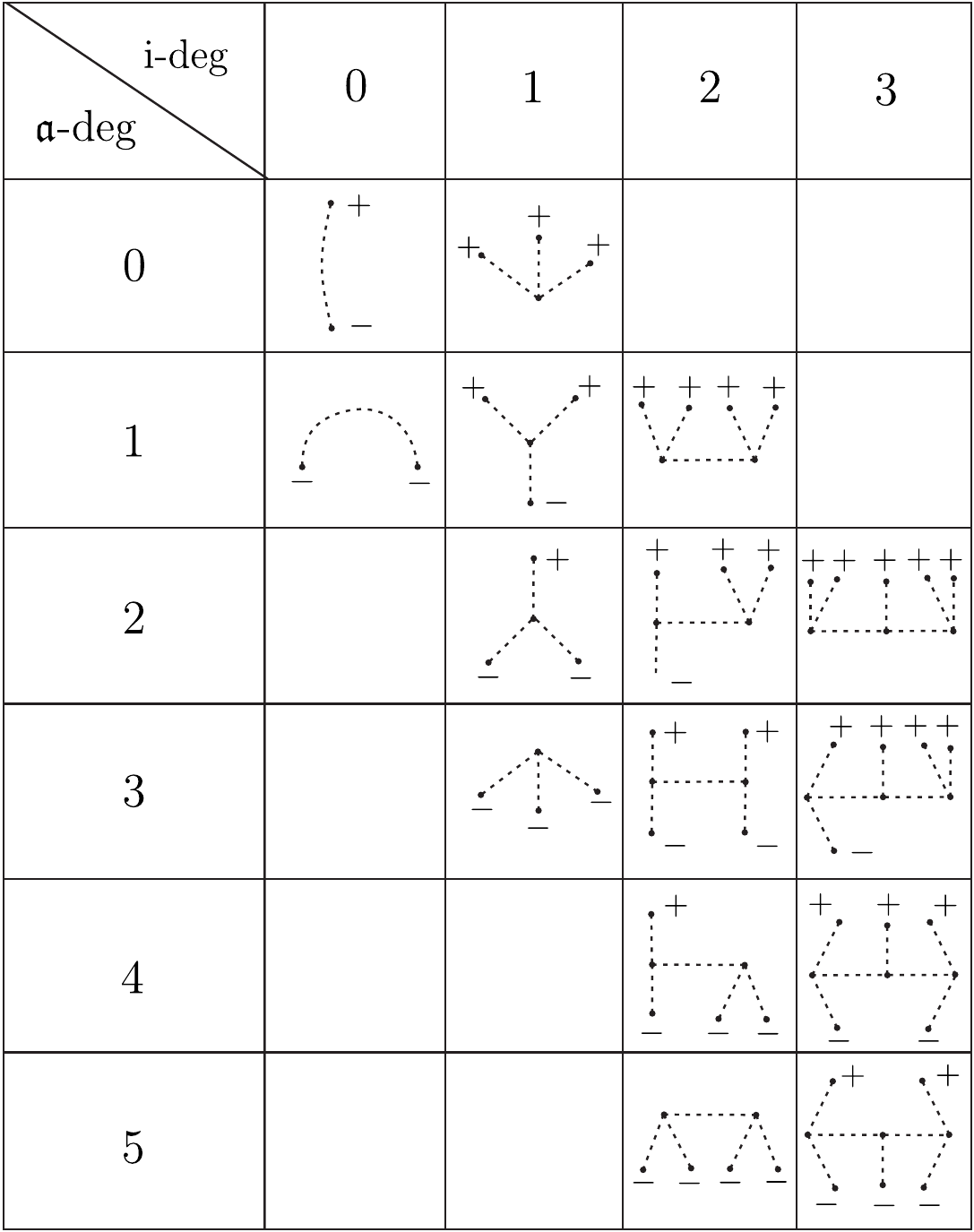}		
\end{center}
\caption{}
\label{table1}
\end{table}

For $m\geq 1$, let $\mathcal{T}^{\mathfrak{a}}_m(B\oplus A)$ denote the subspace of $\mathcal{A}^{t,c}(B\oplus A)$ generated by diagrams of alternative degree $m$. 
\begin{proposition}\label{JTH2prop26}
For $m\geq 1$ the map $\eta$ defined in (\ref{JTH2equ31}) induces an isomorphism 
\begin{equation}\label{JTH2equ33}
\eta: \mathcal{T}^{\mathfrak{a}}_m(B\oplus A)\longrightarrow D_m(B;A)\otimes \mathbb{Q}
\end{equation}
of $\mathbb{Q}$-vector spaces.
\end{proposition}

\begin{proof}
Let $T$ be a  $(B\oplus A)$-colored connected tree-like Jacobi diagram with $\mathfrak{a}$-deg$(T)=m$.  Let $v$ be a leg of $T$ and denote by $T_v$ the Lie commutator obtained from $T$ rooted at $v$. If $v$ is colored by an element of $B$, then $\text{deg}(T_v) = (\mathfrak{a}\text{-deg}(T) + 3) - 1 = m+2$. Hence 
$$\text{color}(v)\otimes T_v \in B\otimes \mathfrak{Lie}_{m+2}(B;A).$$ 
On the other hand, if $v$ is colored by an element of $A$ then $\text{deg}(T_v) = (\mathfrak{a}\text{-deg}(T) + 3) - 2 = m+1$. Therefore 
$$\text{color}(v)\otimes T_v \in A\otimes \mathfrak{Lie}_{m+1}(B;A).$$ 
To sum up 
$$\eta(T)\in (A\otimes\mathfrak{Lie}_{m+1}(B;A))\oplus (B\otimes\mathfrak{Lie}_{m+2}(B;A)).$$
The argument \cite[Lemma 3.1]{MR1943338} used in the proof of  Theorem \ref{JTH2thm23} to show that $\Xi\eta(T)=0$ is still valid. The only caveat is when $\mathfrak{a}$-deg$(T)=1$ and i-deg$(T)=0$. In this case $T$ is a strut whose both legs are colored by elements of $A$, say $a_i$, $a_j$. Then $\eta(T) = a_i\otimes a_j + a_j \otimes a_i$, so $\Xi\eta(T)= 0$ by the antisymmetry relation. This way for $m\geq 1$ we have $\eta\big(\mathcal{T}^{\mathfrak{a}}_m(B\oplus A)\big)\subseteq D_m(B;A)\otimes \mathbb{Q}$ and by  (\ref{JTH2cor24}), the map $\eta|_{\mathcal{T}^{\mathfrak{a}}_m(B\oplus A)}$ is injective (it is again  necessary to consider the case $m = 1$ separately). For the surjectivity, first consider the case $m=1$. The elements in $(A\otimes A)\cap D_1(B;A)$ are linear combinations of elements of the form $a_i\otimes a_i$ and $a_i\otimes a_j + a_j\otimes a_i$. Now if $T$ is the strut whose both legs are colored by $a_i$, then $(1/2)\eta(T)=a_i\otimes a_i $ and if $T$ is the strut whose legs are colored by $a_i$ and $a_j$, then $\eta(T)=a_i\otimes a_j + a_j\otimes a_i$. 

Let $m\geq 1$ and consider $y\in D_m(B;A)$. In the case $m=1$, by the previous paragraph, we can suppose that there are no elements of $(A\otimes A)\cap D_1(B;A)$ appearing in $y$. This way we can see $y\in\bigoplus_{m\geq 1} D_m(H)\otimes \mathbb{Q}$.  By (\ref{JTH2cor24}), there exists $T\in\bigoplus_{m\geq1}\mathcal{A}_m^{t,c}(H)$ such that $\eta(T)=y$. Consider the decomposition of $T$ by the alternative degree  $T = \sum T_i$ with $T_i\in\mathcal{T}^{\mathfrak{a}}_i(B\oplus A)$. Thus $\eta(T)=\sum \eta(T_i)= y$, but for $i\not= m$ we know  that $\eta(T_i)\not\in D_m(B;A)\otimes\mathbb{Q}$. Hence  $\eta(T_i)=0$ for $i\not=m$. By the injectivity of $\eta$, we obtain $T_i=0$ for $i\not=m$. Therefore $T=T_m\in\mathcal{T}^{\mathfrak{a}}_m(B\oplus A)$ and $\eta(T)=y$.
\end{proof}

Theorem \ref{JTH2thm23} and Proposition \ref{JTH2prop26} allow  to define diagrammatic versions of the Johnson-type homomorphisms.

\begin{definition}\label{defdiag} Let $m\geq 1$.  The \emph{diagrammatic version} of the $m$-th alternative Johnson homomorphism is defined as the composition
\begin{equation}
 J^{\mathfrak{a}}_m\mathcal{M} \xrightarrow{\ \tau^{\mathfrak{a}}_m \ } D_m(B;A)\otimes \mathbb{Q} \xrightarrow{\ \eta^{-1} \ } \mathcal{T}^{\mathfrak{a}}_m(B\oplus A).
\end{equation} 
Similarly, the \emph{diagrammatic versions}  of the $m$-th Johnson homomorphism and of the $m$-th Johnson-Levine homomorphism are defined as the compositions
\begin{equation}
J_m\mathcal{M} \xrightarrow{\ \tau_m \ } D_m(H)\otimes \mathbb{Q} \xrightarrow{\ \eta^{-1} \ } \mathcal{A}^{t,c}_m(H)
\end{equation}
and
\begin{equation}
 J^L_m\mathcal{M} \xrightarrow{\ \tau^L_m \ } D_m(H')\otimes \mathbb{Q} \xrightarrow{\ \eta^{-1} \ } \mathcal{A}^{t,c}_m(H'),
\end{equation}
respectively.
\end{definition}

\begin{example}\label{JTH2ejemplo5} In Example \ref{JTH2example1}  we calculated $\tau^{\mathfrak{a}}_1(t_{\alpha_i})= - a_i\otimes a_i$, for the Dehn twist $t_{\alpha_i}$ from Example \ref{ejemploJTH1}. Therefore 

\medskip

\begin{center}
\includegraphics[scale=1]{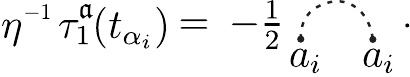}		
\end{center}
\end{example}

\begin{example}\label{JTH2ejemplo6}  In Example \ref{JTH2example2}  we calculated
 $$\tau^{\mathfrak{a}}_1(t_{\alpha_{kl}})= - (a_k\otimes a_k) - (a_l\otimes a_l) - (a_k\otimes a_l) - (a_l\otimes a_k),$$
 for the Dehn twist $t_{\alpha_{kl}}$ from Example \ref{ejemploJTH1}. Hence 
 
\medskip

\begin{center}
\includegraphics[scale=1]{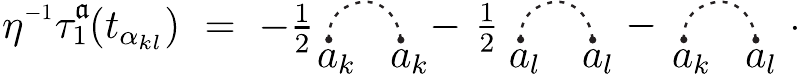}		
\end{center}
\end{example}

Comparing these results with the low degree values of the LMO functor on the cobordisms $c(t_{\alpha_i})$ and $c(t_{\alpha_{kl}})$   computed in Examples \ref{ejemploLMO6}, \ref{ejemploLMO7} and \ref{ejemploLMO8}, we can see that, for these examples, the diagrammatic version of the first alternative Johnson homomorphism appears in the LMO functor with an opposite sign.  This is a more general fact which we develop in  next section.

\section{Alternative Johnson homomorphisms and the LMO functor}\label{section6} In this section we establish the relation between the LMO functor and the alternative Johnson homomorphisms. From Proposition \ref{JTHproposition5}, we know that for $m\geq 2$, $J^{\mathfrak{a}}_m\mathcal{M}\subseteq \mathcal{I}$. Hence, we can use some known results involving the Torelli group. Therefore we carry this out in two stages separately. First we establish this relation for the alternative Johnson homomorphism $\tau^{\mathfrak{a}}_1$. Then we consider $\tau^{\mathfrak{a}}_m$ for $m\geq 2$. First of all let us start by defining the filtration on cobordisms induced by the alternative degree.

\subsection{The filtration on  Lagrangian cobordisms induced by the alternative degree}

Recall from subsection \ref{subsectionLMOfunctor} that $\mathcal{C}=\mathcal{C}_{g,1}$ denotes the monoid of homology cobordisms of $\Sigma=\Sigma_{g,1}$. If $(M,m)\in\mathcal{C}$ then, by Stallings' theorem \cite[Theorem 3.4]{MR0175956}, the maps $m_{\pm,*}:\pi/\Gamma_k\pi\rightarrow \pi_1(M,*)/\Gamma_k\pi_1(M,*)$ are isomorphisms for $k\geq 2$. We can then define the nilpotent version of the Dehn-Nielsen-Baer representation (\ref{JTHequ16}) for the monoid of homology cobordisms as the monoid homomorphism
\begin{equation}\label{JTHLMOequ1}
\rho_k:\mathcal{C}\longrightarrow \text{Aut}(\pi/\Gamma_{k+1}\pi),
\end{equation}
that sends $(M,m)\in \mathcal{C}$ to the automorphism $m_{-,*}^{-1}\circ m_{+,*}$. Consider the following submonoids of $\mathcal{C}$. 

The monoid $\mathcal{IC}$ of \emph{homology cylinders} of $\Sigma$ is defined as $\mathcal{IC}=\text{ker}(\rho_1)$. The monoid $\mathcal{LC}$ of \emph{Lagrangian homology cobordisms}  of $\Sigma$ is defined as
\begin{equation}\label{JTHLMOequ2}
\mathcal{LC}=\{(M,m)\in\mathcal{C}\ \ |\ \ \rho_1(M)(A)\subseteq A\}=\{(M,m)\in\mathcal{C}\ \ |\ \ m_{+,*}(A)\subseteq m_{-,*}(A)\},
\end{equation}
and the monoid $\mathcal{IC}^L$ of \emph{strongly Lagrangian homology cobordisms} of $\Sigma$ is defined as
\begin{equation}\label{JTHLMOequ3}
\mathcal{IC}^L =\{(M,m)\in\mathcal{LC}\ \ |\ \ \rho_1(M)|_A=\text{Id}_A\}=\{(M,m)\in\mathcal{LC}\ \ |\ \ m_{+,*}|_A = m_{-,*}|_A\}.
\end{equation}

The monoids $\mathcal{LC}$, $\mathcal{IC}^L$ and $\mathcal{IC}$ are characterized in terms of the linking matrix as follows.

\begin{lemma}\label{JTHLMOlemma1} Let $M\in\mathcal{C}_{g,1}$ and let $(B,\gamma)$ be its bottom-top tangle presentation. Then
\begin{enumerate}
\item[\emph{(i)}] $M$ belongs to $\mathcal{LC}_{g,1}$ if and only if $B$ is a homology cube and $\text{\emph{Lk}}(M)=\left( \begin{smallmatrix} 0 & \Lambda \\ {\Lambda}^{T} & \Delta \end{smallmatrix} \right)$, 

\medskip

\item[\emph{(ii)}] $M$ belongs to $\mathcal{IC}^L_{g,1}$ if and only if $B$ is a  homology cube and   $\text{\emph{Lk}}(M)=\left( \begin{smallmatrix} 0 & {\text{\emph{Id}}_g}\\ {\text{\emph{Id}}_g} & {\Delta} \end{smallmatrix} \right)$,

\medskip

\item[\emph{(iii)}] $M$ belongs to $\mathcal{IC}_{g,1}$ if and only if $B$ is a homology cube and   $\text{\emph{Lk}}(M)=\left( \begin{smallmatrix} 0 & {\text{\emph{Id}}_g}\\ {\text{\emph{Id}}_g} & 0 \end{smallmatrix} \right)$,
\end{enumerate}
where  $\Lambda$ and $\Delta$ are $g\times g$ matrices and $\Delta$ is symmetric.
\end{lemma}

We refer to \cite[Lemma 3.7]{vera1} or \cite[Lemma 2.12]{MR2403806} for a proof.

\begin{definition}
The monoid $\mathcal{IC}^{\mathfrak{a}}$ of \emph{alternative homology cylinders} of $\Sigma$ is defined as
$$\mathcal{IC}^{\mathfrak{a}} = \{(M,m)\in\mathcal{IC}^L\ \ |\ \ \forall\alpha\in\mathbb{A}:\ \iota_{\#}\rho_2(M)(\alpha\Gamma_3\pi)= 1\in \pi'/\Gamma_3\pi'\}.$$
Here $\iota_{\#}:\pi/\Gamma_3\pi\rightarrow \pi'/\Gamma_3\pi'$ is induced by $\iota_{\#}:\pi\rightarrow \pi'$.
\end{definition}
Notice that the given definition of $\mathcal{IC}^{\mathfrak{a}}$ is motivated by the definition of $J^L_2\mathcal{M}$. There is an equivalent definition motivated by the definition of $\mathcal{I}^{\mathfrak{a}}$, see Proposition \ref{JTH2prop14}.
\begin{example} If $c:\mathcal{M}\rightarrow \mathcal{C}$ is the mapping cylinder monoid homomorphism, then $c(\mathcal{L})\subseteq \mathcal{LC}$, $c(\mathcal{I})\subseteq \mathcal{IC}$, $c(\mathcal{I}^L)\subseteq \mathcal{IC}^L$ and $c(\mathcal{I}^{\mathfrak{a}})\subseteq \mathcal{IC}^{\mathfrak{a}}$.
\end{example}

Recall that for $M$ a Lagrangian cobordism,  $\widetilde{Z}^{Y,t}(M)$ denotes the reduction of the value $\widetilde{Z}(M)$ modulo  struts and looped diagrams.

\begin{definition} The \emph{alternative tree filtration} $\{\mathcal{F}^{\mathfrak{a}}_m\mathcal{C}\}_{m\geq 1}$ of $\mathcal{C}$ is defined by
$$\mathcal{F}^{\mathfrak{a}}_m\mathcal{C}=\{(M,m)\in\mathcal{IC}^{\mathfrak{a}}\ \ |\ \ \widetilde{Z}^{Y,t}(M)=\varnothing + (\text{terms of } \mathfrak{a}\text{-deg}\geq m)\}.$$
\end{definition}

Let $\mathcal{T}^{Y,\mathfrak{a}}_m(\left\lfloor g\right\rceil^+\sqcup \left\lfloor g\right\rceil^-)$ denote the subspace of $\mathcal{A}^{Y,t}(\left\lfloor g\right\rceil^+\sqcup \left\lfloor g\right\rceil^-)$ generated by diagrams of $\mathfrak{a}$-\emph{deg} $= m$.
\begin{theorem}\label{JTHLMOprop1}
For $m\geq 1$, the set $\mathcal{F}^{\mathfrak{a}}_m\mathcal{C}$ is a submonoid of $\mathcal{C}$. Consider the map
$$\widetilde{Z}^{Y,\mathfrak{a}}_m : \mathcal{F}^{\mathfrak{a}}_m\mathcal{C}\longrightarrow \mathcal{T}^{Y,\mathfrak{a}}_m(\left\lfloor g\right\rceil^+\sqcup \left\lfloor g\right\rceil^-),$$
where $\widetilde{Z}^{Y,\mathfrak{a}}_m(M)$ is defined as the sum of the terms of $\mathfrak{a}$-\emph{deg} $= m$ in $\widetilde{Z}^{Y,t}(M)$ for $M\in \mathcal{F}^{\mathfrak{a}}_m\mathcal{C}$.  Then $\widetilde{Z}^{Y,\mathfrak{a}}_m$ is a monoid homomorphism.
\end{theorem}

\begin{proof}
 Let $M,N\in\mathcal{F}^{\mathfrak{a}}_m\mathcal{C}$ and write $\widetilde{Z}^{Y,t}(M)=\varnothing + D_M+ (\mathfrak{a}\text{-deg}>m)$ and $\widetilde{Z}^{Y,t}(N)=\varnothing + D_N + (\mathfrak{a}\text{-deg}>m)$, where $D_M$ and $D_N$ are linear combinations of connected Jacobi diagrams in $\mathcal{A}^{Y,t}(\left\lfloor g\right\rceil^+\sqcup \left\lfloor g\right\rceil^-)$ of $\mathfrak{a}$-deg $=m$. We have to show that $M\circ N \in\mathcal{F}^{\mathfrak{a}}_m\mathcal{C}$ and that
\begin{equation}\label{JTHLMOequ4}
\widetilde{Z}^{Y,t}(M\circ N)=\varnothing + (D_M + D_N) + (\mathfrak{a}\text{-deg} > m).
\end{equation}

Suppose that
$\text{Lk}(M)=\left( \begin{smallmatrix} 0 & {\text{Id}_g}\\ {\text{Id}_g} & \Lambda \end{smallmatrix} \right)$ and $\text{Lk}(N)=\left( \begin{smallmatrix} 0 & {\text{Id}_g}\\ {\text{Id}_g} & \Delta \end{smallmatrix} \right)$, where $\Lambda=(m_{ij})$ and $\Delta=(n_{ij})$ are symmetric $g\times g$ matrices. We have
\begin{equation}
\widetilde{Z}(M\circ N)= \widetilde{Z}(M)\circ \widetilde{Z}(N)=\left\langle \widetilde{Z}(M)_{|j^+\mapsto j^*}, \ \widetilde{Z}(N)_{|j^-\mapsto j^*}\right\rangle_{\left\lfloor g\right\rceil^*}.
\end{equation}
By Lemma \ref{lemmasplitstruty} we can write
\begin{equation}\label{JTHLMOequ6}
\centerline{\includegraphics[scale=0.87]{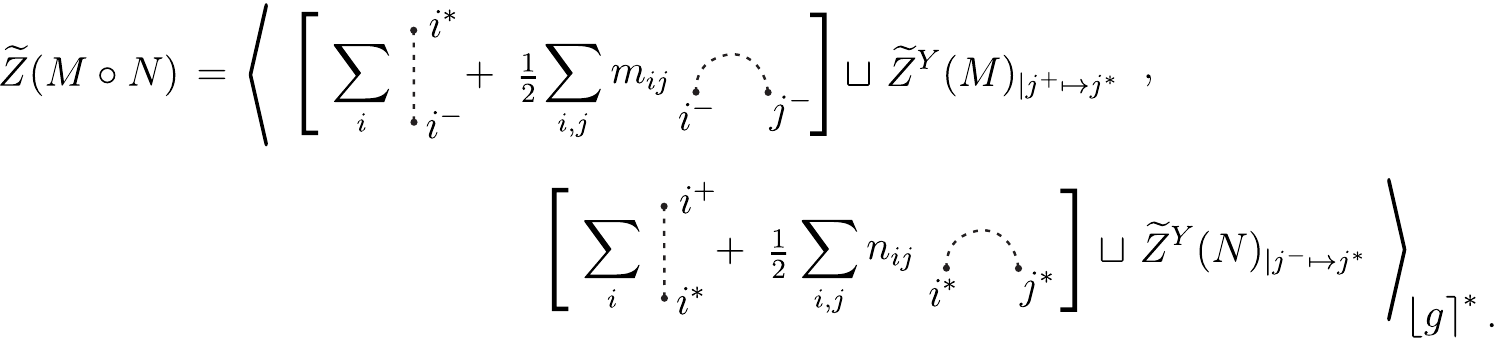}}
\end{equation}
Recall that the square brackets denote an exponential. Let us write
$$
\widetilde{Z}^{Y,t}(M\circ N) = \varnothing + E + (\mathfrak{a}\text{-deg}> m),
$$
where $E$ is a linear combination of connected Jacobi diagrams in $\mathcal{A}^{Y,t}(\left\lfloor g\right\rceil^+\cup \left\lfloor g\right\rceil^-)$ with $\mathfrak{a}$-deg $\leq m$. We want to show that  $E$  only has diagrams with $\mathfrak{a}$-deg $= m$ and moreover that $E = D_M + D_N$. Since $\widetilde{Z}^{Y,t}(M\circ N)$  is obtained by considering the reduction of $\widetilde{Z}(M\circ N)$ modulo struts and looped diagrams, we need to carefully analyse  the pairing (\ref{JTHLMOequ6}).
\begin{itemize}[leftmargin=*]
\item It is possible for $\widetilde{Z}^Y(M)$ or $\widetilde{Z}^Y(N)$ to have diagrams with loops. A diagram in $\widetilde{Z}(M\circ N)$ coming from the pairing of a diagram with loops, in $\widetilde{Z}^Y(M) $ or in $\widetilde{Z}^Y(N)$, with any other diagram will still have loops. Hence the diagrams with loops in $\widetilde{Z}^Y(M)$ or in $\widetilde{Z}^Y(N)$ do no contribute any term to $\widetilde{Z}^{Y,t}(M\circ N)$.

\item The diagrams  of  type

\centerline{\includegraphics[scale=0.87]{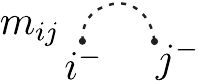}}

\medskip

\noindent do not contribute any connected term to $\widetilde{Z}^{Y,t}(M\circ N)$. Therefore, the diagrams  (which are no struts)  with the lowest alternative degree contributed by the diagrams
\begin{equation}\label{JTHLMOequ7}
\centerline{\includegraphics[scale=0.87]{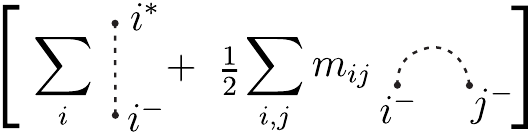}}
\end{equation}

\noindent after the pairing (\ref{JTHLMOequ6}) are exactly the diagrams appearing in $D_N$.

\item The diagrams of  type
\begin{equation}\label{JTHLMOequ8}
\centerline{\includegraphics[scale=0.87]{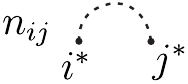}}
\end{equation}
\noindent can contribute  looped diagrams to $\widetilde{Z}^Y(M\circ N)$ when we consider their pairing with \emph{connected} diagrams in $\widetilde{Z}^{Y,t}(M)$, so at the end they do not appear in $\widetilde{Z}^{Y,t}(M\circ N)$. Or they can also contribute connected diagrams to $\widetilde{Z}^{Y}(M\circ N)$ after their pairing with \emph{disconnected} diagrams $T$ of $\widetilde{Z}^t(M)$, where at least one of the connected components of $T$ has at least one trivalent vertex. In Figure \ref{figuraJTHLMO4} we illustrate this situation with three examples of  such a $T$.

\begin{figure}[ht!]
\centering
\includegraphics[scale=0.87]{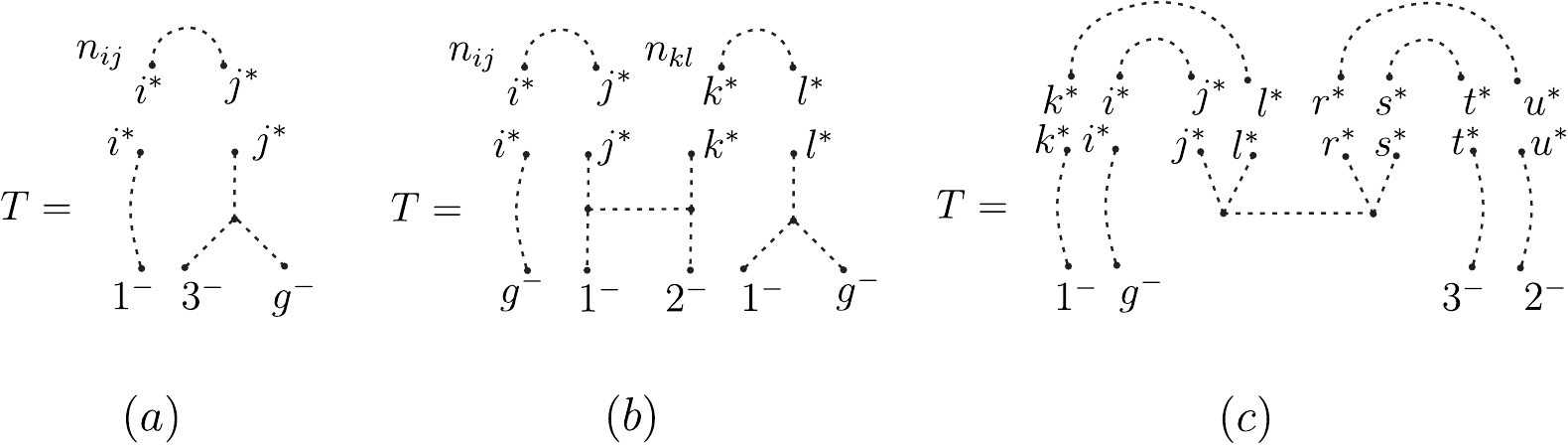}
\caption{}
\label{figuraJTHLMO4}
\end{figure}

\noindent Let us see that in this case, the obtained connected tree-like Jacobi diagrams are of $\mathfrak{a}$-deg $> m$. Let $T$ be a disconnected diagram in $\widetilde{Z}^t(M)$. Then the connected components of $T$ can be struts or diagrams with at least one trivalent vertex. If there are diagrams in $T$  whose all legs are colored by elements of  $\left\lfloor g\right\rceil^-$, then it is not possible to obtain a connected diagram from $T$ after the pairing (\ref{JTHLMOequ6}), hence we can suppose that there is not this type of diagrams in $T$. Also note that if all the legs of all the connected components of $T$ are colored by elements of  $\left\lfloor g\right\rceil^+$, then after the pairing (\ref{JTHLMOequ6}) with the struts of type (\ref{JTHLMOequ8}),  we will obtain looped diagrams. This way we can suppose that there is at least one connected component of $T$ which has one leg colored by an element of $\left\lfloor g\right\rceil^-$, and moreover that all the struts appearing in $T$ have one leg colored by $\left\lfloor g\right\rceil^+$ and the other by $\left\lfloor g\right\rceil^-$. Let $T_1$ be a connected component of $T$ with at least one trivalent vertex, then $\mathfrak{a}$-deg$(T_1)\geq m$. Now, $T_1$ has legs colored by $\left\lfloor g\right\rceil^+$ and when we do the pairing with the struts of the type (\ref{JTHLMOequ8}), we connect such legs either with struts which have legs colored by $\left\lfloor g\right\rceil^-$ or with other trees appearing in $T$. In either case, we strictly increase the alternative degree of $T_1$. To sum up, the connected tree-like Jacobi diagrams which can  appear in this case are of $\mathfrak{a}$-deg $> m$. 

Therefore, the diagrams  (which are not struts)  with the lowest alternative degree contributed by the diagrams
\begin{equation}\label{JTHLMOequ9}
\centerline{\includegraphics[scale=0.87]{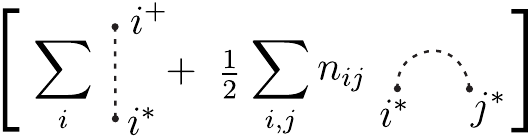}}
\end{equation}

\noindent after the pairing (\ref{JTHLMOequ6}), are exactly the diagrams appearing in $D_M$. 

\item Let $S$ be a diagram appearing in $\widetilde{Z}^{t}(N)$ and $T$ be a diagram appearing in $\widetilde{Z}^{t}(M)$. If all the legs of every connected component of $S$ are colored by $\left\lfloor g\right\rceil^+$, then $S$ does not intervene in the pairing (\ref{JTHLMOequ6}), except with the empty diagram. Similarly when all the legs of every connected component of $T$ are colored by $\left\lfloor g\right\rceil^-$. Hence we can suppose that there is at least a connected component  of $S$ (respectively in $T$) with at least one trivalent vertex and with at least one $\left\lfloor g\right\rceil^-$-colored leg (respectively one $\left\lfloor g\right\rceil^+$-colored leg).  As in the previous case, the connected diagrams without loops obtained from the pairing of $S$ and $T$ strictly increase the alternative degree. The alternative degree only remains stable when we do the pairing with diagrams coming from 

\medskip

\centerline{\includegraphics[scale=0.87]{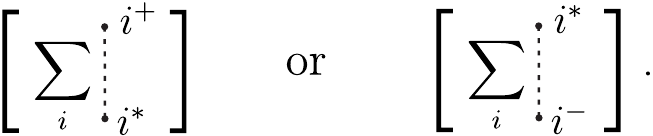}}

\medskip

In conclusion, the lower alternative degree terms appearing in $\widetilde{Z}^{Y,t}(M\circ N)$ are exactly $D_M + D_N$, that is, $E=D_M+ D_N$.
\end{itemize}
\end{proof}

\subsection{First alternative Johnson homomorphism and the LMO functor}

From Definition \ref{defdiag}, the diagrammatic version of the first alternative Johnson homomorphism of an element in $\mathcal{I}^{\mathfrak{a}}$ is given by a linear combination of the diagrams shown in Figure \ref{figuraJTHLMO10}. Recall that by a $-$ (respectively by a $+$) we mean that the color of the leg belongs to $A$ (respectively to $B$).

\begin{figure}[ht!]
\centering
\includegraphics[scale=0.85]{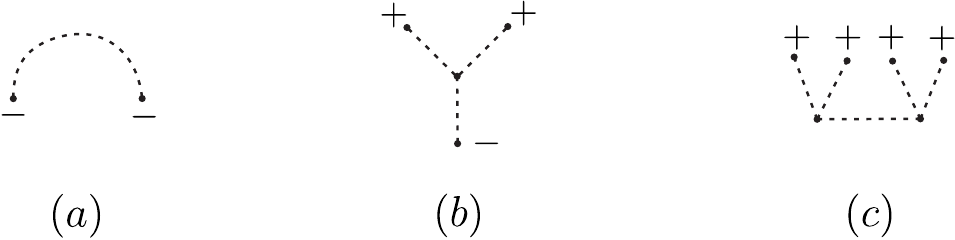}
\caption{Tree-like Jacobi diagrams of $\mathfrak{a}$-deg $= 1$.}
\label{figuraJTHLMO10}
\end{figure}
\noindent Besides, the diagrammatic version of the first Johnson homomorphism of an element in $\mathcal{I}$ is given by a linear combination of the diagrams shown in Figure \ref{figuraJTHLMO11} and the diagrammatic version of the second Johnson-Levine homomorphism is given by a linear combination of diagrams of type $(c)$ in Figure \ref{figuraJTHLMO10}.

\begin{figure}[ht!]
\centering
\includegraphics[scale=0.85]{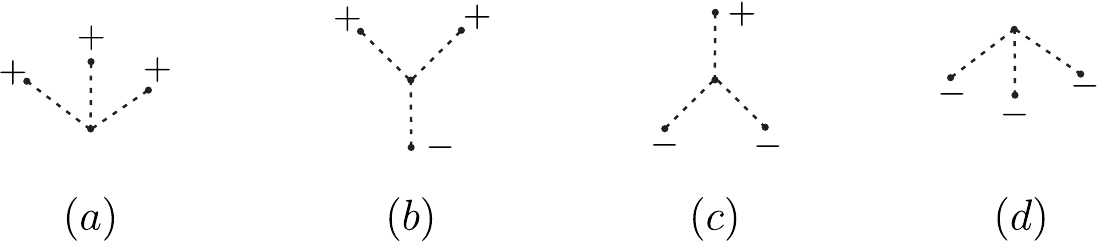}
\caption{Tree-like Jacobi diagrams of $i$-deg $= 1$.}
\label{figuraJTHLMO11}
\end{figure}

Let us start by identifying the elements in $\mathcal{I}^{\mathfrak{a}}$ whose first alternative Johnson homomorphism only contains diagrams of the type $(a)$ in Figure \ref{figuraJTHLMO10}. Let $\mathcal{N}$ be the subgroup of $\mathcal{I}^{\mathfrak{a}}$ generated by the Dehn twists $t_{\alpha_i}$ and $t_{\alpha_{kl}}$ with $1\leq i\leq g$ and $1\leq k < l \leq g$. Here $\alpha_i$ denotes the $i$-th meridional curve as in Figure \ref{figuraLMO2artalt}; and $\alpha_{kl}$ is as shown in Figure~\ref{figuraJTHLMO8artalt}~$(a)$.

Using the symplectic basis $\{a_i,b_i\}$ to identify $\text{Sp}(H)$ with $\text{Sp}(2g,\mathbb{Z})$ we  have that the image of $\mathcal{N}$ under the symplectic representation (\ref{JTHequ2}) is
\begin{equation}\label{JTHLMOequ10}
\sigma(\mathcal{N})= \left\{\left( \begin{smallmatrix} \text{Id}_g &\Delta\\ 0& \text{Id}_g\end{smallmatrix} \right) \  \bigg\rvert \ \ \Delta \text{\ is symmetric}  \right\}\subseteq\text{Sp}(2g,\mathbb{Z}).
\end{equation}
Equality (\ref{JTHLMOequ10}) is precisely \cite[Lemma 6.3]{MR2265877}. It also follows from the computations made  in Examples \ref{JTH2example1} and \ref{JTH2example2}. Notice that $\mathcal{N}$ is contained in the handlebody group $\mathcal{H}$ defined in (\ref{handlebodygroup}).

\begin{lemma}\label{lemmaJTHLMO2} We have the equality
$\mathcal{I}^{\mathfrak{a}}=\mathcal{N}\cdot(\mathcal{I}\cap\mathcal{I}^{\mathfrak{a}}).$
\end{lemma}

\begin{proof}
Let $h\in \mathcal{I}^{\mathfrak{a}}$. In the symplectic basis $\{a_i, b_i\}$, the matrix of $\sigma(h)=h_*$ is  $\left( \begin{smallmatrix} \text{Id}_g &\Delta\\ 0& \text{Id}_g\end{smallmatrix} \right)$ with $\Delta$  a symmetric matrix. From (\ref{JTHLMOequ10}) there exists $f\in\mathcal{N}$ such that the matrix of $\sigma(f)=f_*$ is  $\left( \begin{smallmatrix} \text{Id}_g &-\Delta\\ 0& \text{Id}_g\end{smallmatrix} \right)$. Therefore $f\circ h\in\mathcal{I}$. Let $\psi = f\circ h$, we also have $\psi=f\circ h\in \mathcal{I}^{\mathfrak{a}}$. Thus $h=f^{-1}\circ \psi\in\mathcal{N}\cdot (\mathcal{I}\cap \mathcal{I}^{\mathfrak{a}})$.
\end{proof}

We have already computed in Examples \ref{JTH2example1} and \ref{JTH2example2} the first alternative Johnson homomorphism for the generators of $\mathcal{N}$. These computations imply the following.

\begin{proposition}\label{JTHLMOprop2} For $h\in\mathcal{N}$ the first alternative Johnson homomorphism $\tau^{\mathfrak{a}}_1(h)$ can be computed from the action of $h$ in homology and reciprocally. More precisely, if the matrix of $\sigma(h)=h_*:H\rightarrow H$ in the symplectic basis $\{a_i, b_i\}$ is $\left( \begin{smallmatrix} \text{\emph{Id}}_g &\Delta\\ 0& \text{\emph{Id}}_g\end{smallmatrix} \right)$ with $\Delta=(n_{ij})$ a symmetric matrix, then
$$ \tau^{\mathfrak{a}}_1(h) = \sum_{1\leq i,j\leq g} n_{ij} a_i\otimes a_j.$$
\end{proposition}

\begin{proof}
By Examples \ref{JTH2example1} and \ref{JTH2example2}, the result holds for the generators $t_{\alpha_i}$ and $t_{\alpha_{kl}}$ of $\mathcal{N}$. The general result follows from the homomorphism property of $\tau^{\mathfrak{}a}_1$ and the equality
$$\left(\begin{smallmatrix} \text{Id}_g &\Delta\\ 0& \text{Id}_g\end{smallmatrix} \right)\left(\begin{smallmatrix} \text{Id}_g &\Delta'\\ 0& \text{Id}_g\end{smallmatrix} \right) = \left(\begin{smallmatrix} \text{Id}_g &\Delta + \Delta'\\ 0& \text{Id}_g\end{smallmatrix} \right), $$
for all matrices $\Delta$ and $\Delta'$ of size $g\times g$.
\end{proof}

\begin{proposition}\label{JTHLMOcor2} For $h\in\mathcal{N}$ we have
\begin{center}
\includegraphics[scale=0.74]{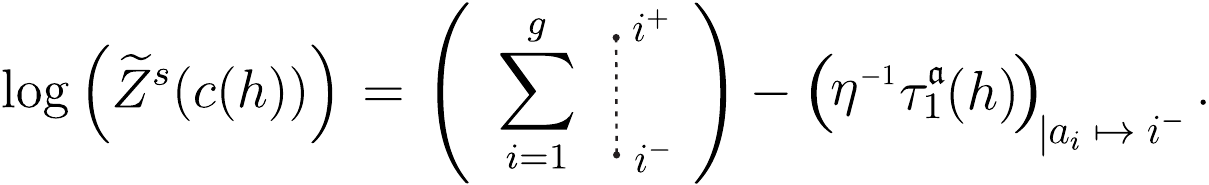}	
\end{center}
\end{proposition}
\begin{proof}
From Examples \ref{ejemploLMO6}, \ref{ejemploLMO7}, \ref{ejemploLMO8} and Examples \ref{JTH2ejemplo5} and \ref{JTH2ejemplo6}, we already have the result for the generators of $\mathcal{N}$. By Lemma \ref{lemmasplitstruty}, we know that the strut part in the LMO functor is encoded in the linking matrix. Thus we need to know how the strut part, or equivalently  the linking matrix, behaves with respect to the composition of cobordisms, this was done in \cite[Lemma 4.5]{MR2403806}. For $M,N\in\mathcal{IC}^L$ with $\text{Lk}(M)=\left( \begin{smallmatrix} 0 &\text{Id}_g\\ \text{Id}_g& \Delta\end{smallmatrix} \right)$ and $\text{Lk}(N)=\left( \begin{smallmatrix} 0 &\text{Id}_g\\ \text{Id}_g& \Delta'\end{smallmatrix} \right)$ we have
\begin{equation}\label{equcomplk}
\text{Lk}(M\circ N)=\left( \begin{smallmatrix} 0 &\text{Id}_g\\ \text{Id}_g& \Delta + \Delta'\end{smallmatrix} \right).
\end{equation}
Whence we have the desired result from the homomorphism property of $\tau^{\mathfrak{a}}_1$ and $\eta^{-1}$. (Equality (\ref{equcomplk}) can also be deduced from the description of the composition of cobordisms in terms of their bottom-top tangle presentations).
\end{proof}

At this point, we understand the first alternative Johnson homomorphism for the elements of $\mathcal{N}$ and moreover we know that the diagrammatic version only contains diagrams of  type $(a)$ in Figure \ref{figuraJTHLMO10}.  By Lemma \ref{lemmaJTHLMO2}, in order to understand $\tau_1^{\mathfrak{a}}$ for all $\mathcal{I}^{\mathfrak{a}}$ we need to understand it for the elements of $\mathcal{I}\cap \mathcal{I}^{\mathfrak{a}}$. Recall that  $B=H'$.

Consider the projection
\begin{equation}
p: (A\otimes \mathfrak{Lie}_2(B;A))\oplus (B\otimes \mathfrak{Lie}_3(B;A))\longrightarrow (A\otimes \Lambda^2 B)\oplus \left(B\otimes \frac{\mathfrak{Lie}_3(B;A)}{\mathfrak{Lie}_3(B)}\right),
\end{equation}

\noindent defined as $\text{Id}_A\otimes \iota_*$ on the first direct summand and by sending $b\otimes z \in B\otimes\mathfrak{Lie}_3(B;A)$ to $b\otimes (z \mathfrak{Lie}_3(B))$. 

\noindent Let $\psi\in\mathcal{I}\cap \mathcal{I}^{\mathfrak{a}}$. Then, $\tau_1^{\mathfrak{a}}(\psi)\in(A\otimes \mathfrak{Lie}_2(B;A))\oplus (B\otimes \mathfrak{Lie}_3(B;A))$. Thus, we can apply $p$ to $\tau_1^{\mathfrak{a}}(\psi)$. Diagrammatically, when we apply $p$ we kill the diagrams in $\eta^{-1}(\tau^{\mathfrak{a}}_1(\psi))$ of type $(c)$ in Figure \ref{figuraJTHLMO10}. Besides, by  Proposition  \ref{JTH2prop11}, we have that $\iota_*(\tau^{\mathfrak{a}}_1(\psi))= \tau_2^L(\psi)$. Here the map $\iota_* $ is the map appearing in Proposition \ref{JTH2prop11}. Diagrammatically we are killing all the diagrams in $\eta^{-1}(\tau^{\mathfrak{a}}_1(\psi))$ with at least one leg colored by an element of $A$. This way, we have
\begin{equation}\label{JTHLMOequ12}
\tau_1^{\mathfrak{a}}(\psi) = p(\tau_1^{\mathfrak{a}}(\psi)) + \iota_*(\tau^{\mathfrak{a}}_1(\psi)) = p(\tau_1^{\mathfrak{a}}(\psi)) + \tau_2^{L}(\psi).
\end{equation}

On the other hand we  can  also consider the first Johnson homomorphism $\tau_1(\psi)$ of $\psi$. Hence,  the diagrammatic version  of  $\tau_1(\psi)$ is a linear combination of the diagrams shown in Figure~\ref{figuraJTHLMO11}. We want to compare $\tau_1(\psi)$ and $\tau^{\mathfrak{a}}_1(\psi)$. Thus, we need to kill the diagrams of type  $(a)$, $(c)$ and $(d)$ in Figure \ref{figuraJTHLMO11} from $\eta^{-1}\tau_1(\psi)$. For this, consider the projection
\begin{equation}
q: (B\otimes \Lambda^2 H)\oplus (A\otimes \Lambda^2H)\longrightarrow \left(B\otimes \frac{\Lambda^2H}{\langle \Lambda^2A +  \Lambda^2B\rangle}\right)\oplus \left(A\otimes \Lambda^2 H'\right),
\end{equation} 
which in the first direct summand is given by the $\text{Id}_B$ tensored with the projection, and in the second direct summand is given by $\text{Id}_A\otimes \Lambda^2\iota_*$, where $\Lambda^2\iota_*:\Lambda^2H\rightarrow \Lambda^2 H'$  is induced by  $\iota_*:H\rightarrow H'$.

\begin{lemma}\label{JTHLMOlemma6.10} Via the canonical isomorphism 
$$\frac{\mathfrak{Lie}_3(B;A)}{\mathfrak{Lie}_3(B)}\cong A\otimes B\cong \frac{\Lambda^2 H}{\langle \Lambda^2A + \Lambda^2B\rangle},$$
we have  $p(\tau^{\mathfrak{a}}_1(\psi)) = q(\tau_1(\psi))$ for every $\psi\in\mathcal{I}\cap\mathcal{I}^{\mathfrak{a}}$.
\end{lemma}

\begin{proof}
Let $\psi\in \mathcal{I}\cap\mathcal{I}^{\mathfrak{a}}$. From equation (\ref{JTH2equ14}) we can write
\begin{equation}\label{JTHLMOequ14}
\tau^{\mathfrak{a}}_1(\psi)  = \sum_{i=1}^g a_i\otimes \left(\psi_{\#}(\beta_i)\beta^{-1}_iK_{3}\right) - \sum_{i=1}^g b_i\otimes \left(\psi_{\#}(\alpha_i)\alpha^{-1}_iK_{4}\right),
\end{equation}
besides
\begin{equation}\label{JTHLMOequ15}
\tau_1(\psi) = \sum_{i=1}^g a_i\otimes \left(\psi_{\#}(\beta_i)\beta^{-1}_i\Gamma_{3}\pi\right) - \sum_{i=1}^g b_i\otimes \left(\psi_{\#}(\alpha_i)\alpha^{-1}_i\Gamma_{3}\pi\right).
\end{equation}
By applying $p$ to (\ref{JTHLMOequ14}) and $q$ to (\ref{JTHLMOequ15}) we obtain 
\begin{equation*}\label{JTHLMOequ16}
p\tau^{\mathfrak{a}}_1(\psi)  = \sum_{i=1}^g a_i\otimes \left(\iota_{\#}(\psi_{\#}(\beta_i)\beta^{-1}_i)\Gamma_{3}\pi'\right) - \sum_{i=1}^g b_i\otimes \left(\left(\psi_{\#}(\alpha_i)\alpha^{-1}_iK_{4}\right) \text{ mod } \mathfrak{Lie}_3(B)\right),
\end{equation*}
and
\begin{equation*}\label{JTHLMOequ17}
q\tau_1(\psi) = \sum_{i=1}^g a_i\otimes \left(\iota_{\#}(\psi_{\#}(\beta_i)\beta^{-1}_i)\Gamma_{3}\pi'\right) - \sum_{i=1}^g b_i\otimes \left(\left(\psi_{\#}(\alpha_i)\alpha^{-1}_i\Gamma_{3}\pi\right)\text{ mod }(\Lambda^2A+\Lambda^2 B)\right).
\end{equation*}
Thus we need to show that $$(\psi_{\#}(\alpha_i)\alpha^{-1}_i K_4) \text{ mod } \mathfrak{Lie}_3(B)$$ and $$(\psi_{\#}(\alpha_i)\alpha^{-1}_i \Gamma_3\pi) \text{ mod } (\Lambda^2A+\Lambda^2 B)$$ define the same element in $A\otimes B$. By Lemma \ref{JTHlemma3}, we can write $\psi_{\#}(\alpha_i)\alpha^{-1}_i=y_in_i$ with $y_i\in\Gamma_3\pi\subseteq K_3$ and $n_i\in\mathbb{A}$. Since $\psi_{\#}(\alpha_i)\alpha^{-1}_i\in \Gamma_2\pi$, then $n_i\in \Gamma_2\pi\cap \mathbb{A}= [\pi, \mathbb{A}]\subseteq K_3$. 
Therefore
\begin{equation*}
\psi_{\#}(\alpha_i)\alpha^{-1}_i\Gamma_3\pi = n_i\Gamma_3\pi = \sum_{k<l}\lambda^i_{kl}(a_k\wedge a_l) + \sum_{k<l}\epsilon^i_{kl}(b_k\wedge b_l) + \sum_{k,l}\delta^i_{kl}(a_k\wedge b_l) \in \Lambda^2 H,
\end{equation*}
where $\lambda^i_{kl},\epsilon^i_{kl},\delta^i_{kl}\in\mathbb{Z}$.  Thus
$$(\psi_{\#}(\alpha_i)\alpha^{-1}_i \Gamma_3\pi) \text{ mod } (\Lambda^2A+\Lambda^2 B) = \sum_{k,l}\delta^i_{kl}(a_k\otimes b_l) \in A\otimes B.$$
On the other hand $\psi_{\#}(\alpha_i)\alpha^{-1}_iK_4 = y_iK_4 + n_iK_4$, but $y_i\in \Gamma_3\pi$, hence  $y_iK_4\in\mathfrak{Lie}_3(B)$. Therefore we also have
\begin{equation*}
\begin{split}
(\psi_{\#}(\alpha_i)\alpha^{-1}_i K_4) \text{ mod } \mathfrak{Lie}_3(B) & = (y_iK_4) \text{ mod } \mathfrak{Lie}_3(B) +  (n_iK_4) \text{ mod } \mathfrak{Lie}_3(B)\\
& = \sum_{k,l}\delta^i_{kl}(a_k\otimes b_l) \in A\otimes B.
\end{split}
\end{equation*}
\end{proof}

\para{The first alternative Johnson homomorphism and the LMO functor} In order to relate $\tau^{\mathfrak{a}}_1$ with the LMO functor we need particular cases of the two theorems that say how   the Johnson homomorphisms and the Johnson-Levine homomorphisms are related to the LMO functor. For $h\in\mathcal{L}$ denote by $\widetilde{Z}^{Y,t,+}(c(h))$ the element in $\mathcal{A}^{Y,t}(\left\lfloor g\right\rceil^+)$ obtained from $\widetilde{Z}^{Y}(c(h))$ by sending all terms with loops or with $i^{-}$-colored legs to $0$. 

\begin{theorem}\cite[Corollary 5.11]{MR2903772}\label{JTHLMOthmCHM}  Let $m\geq 1$. If $h\in J_m\mathcal{M}$ then
\begin{equation*}\label{JTHLMOequ18}
\widetilde{Z}^{Y,t}(c(h))=\varnothing - \left(\eta^{-1}\tau_m(c(h))\right)_{|a_j\mapsto j^-, b_j\mapsto j^+} + (\text{\emph{i-deg}}>m).
\end{equation*}
\end{theorem}

\begin{theorem}\cite[Theorem 5.4]{vera1}\label{JTHLMOthmvera}  Let $m\geq 1$. If $h\in J^L_m\mathcal{M}$ then
\begin{equation*}\label{JTHLMOequ19}
\widetilde{Z}^{Y,t,+}(c(h))=\varnothing - \left(\eta^{-1}\tau^L_m(c(h))\right)_{|b_j\mapsto j^+} + (\text{\emph{i-deg}}>m).
\end{equation*}
\end{theorem}

\begin{remark} We state Theorems \ref{JTHLMOthmCHM} and \ref{JTHLMOthmvera}  in the context of the mapping class group, but the original versions are stated in the context of homology cobordisms.
\end{remark}

\begin{lemma}\label{JTHLMOlemma6.14} Let $h\in\mathcal{N}\subseteq J^{\mathfrak{a}}_1\mathcal{M}$. Then 
$$\widetilde{Z}^{Y,t}(c(h))= \varnothing + (\mathfrak{a}\text{\emph{-deg}} >1),$$
equivalently, $\widetilde{Z}^{Y,\mathfrak{a}}_1(c(h)) = 0$.
\end{lemma}
\begin{proof}
The diagrams with $\mathfrak{a}$-deg $=1$ have i-deg between $0$ and $2$ included. In Examples \ref{ejemploLMO6}, \ref{ejemploLMO7} and \ref{ejemploLMO8} we have seen  that for the generators  $h\in\mathcal{N}$ there are no diagrams  of $\mathfrak{a}$-deg $=1$ and  i-deg $=1$ in $\widetilde{Z}^{Y,t}(c(h))$. Now, the diagrams of $\mathfrak{a}$-deg $=1$ and i-deg $=2$ are of  type $(c)$ in Figure \ref{figuraJTHLMO10}. Since $\mathcal{N}$ is included in the handlebody group, this kind of diagrams do not appear in $\widetilde{Z}^{Y,t}$, see \cite[Corollary 5.4]{MR2403806}. Therefore we have the stated result for the generators of $\mathcal{N}$ and the general statement follows by Theorem \ref{JTHLMOprop1}.
\end{proof}

\begin{theorem}\label{JTHLMOthm6.15} Let $f\in J^{\mathfrak{a}}_1\mathcal{M}$. Then
\begin{center}
\includegraphics[scale=0.74]{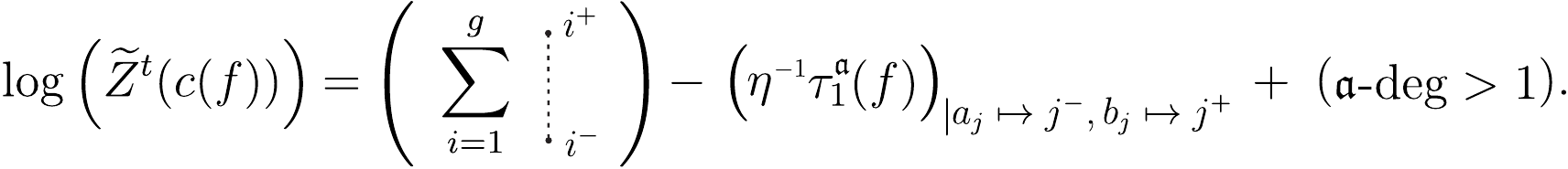}	
\end{center}

\noindent In particular $\widetilde{Z}^{Y,\mathfrak{a}}_1(c(f)) = -\left(\eta^{-1}\tau^{\mathfrak{a}}_1(f)\right)^Y_{|a_j\mapsto j^-, b_j\mapsto j^+}$, where $\left(\eta^{-1}\tau^{\mathfrak{a}}_1(f)\right)^Y$ is the reduction of $\eta^{-1}\tau^{\mathfrak{a}}_1(f)$ modulo struts. 
\end{theorem}

\begin{proof} Let $f\in J^{\mathfrak{a}}_1\mathcal{M}$. By Lemma \ref{lemmaJTHLMO2} we can write $f=h\psi$ with $h\in\mathcal{N}$ and $\psi\in \mathcal{I}\cap\mathcal{I}^\mathfrak{a}$. From (\ref{JTHLMOequ12}) and Lemma \ref{JTHLMOlemma6.10} we have
$$\tau^{\mathfrak{a}}_1(f) = \tau^{\mathfrak{a}}_1(h) + \tau^{\mathfrak{a}}_1(\psi) =    \tau^{\mathfrak{a}}_1(h) + p( \tau^{\mathfrak{a}}_1(\psi)) + \tau^L_2(\psi) =  \tau^{\mathfrak{a}}_1(h) + q( \tau_1(\psi)) + \tau^L_2(\psi) .$$
Therefore 
\begin{equation}\label{JTHLMOthmequ06.15}
\eta^{-1}\tau^{\mathfrak{a}}_1(f) = \eta^{-1}\tau^{\mathfrak{a}}_1(h) + \eta^{-1}q(\tau_1(\psi)) + \eta^{-1}\tau^L_2(\psi).
\end{equation}
Notice that $\eta^{-1}\tau^{\mathfrak{a}}_1(h)$ corresponds to the diagrams of i-deg $= 0$, $\eta^{-1}p( \tau^{\mathfrak{a}}_1(\psi))$ corresponds to the diagrams of i-deg $= 1$ and $\eta^{-1}\tau^L_2(\psi)$ corresponds to the diagrams of i-deg $= 2$ appearing in $\eta^{-1}\tau^{\mathfrak{a}}_1(f)$. By  Lemma \ref{lemmasplitstruty} and Proposition \ref{JTHLMOcor2} we have 
\begin{center}
\includegraphics[scale=0.74]{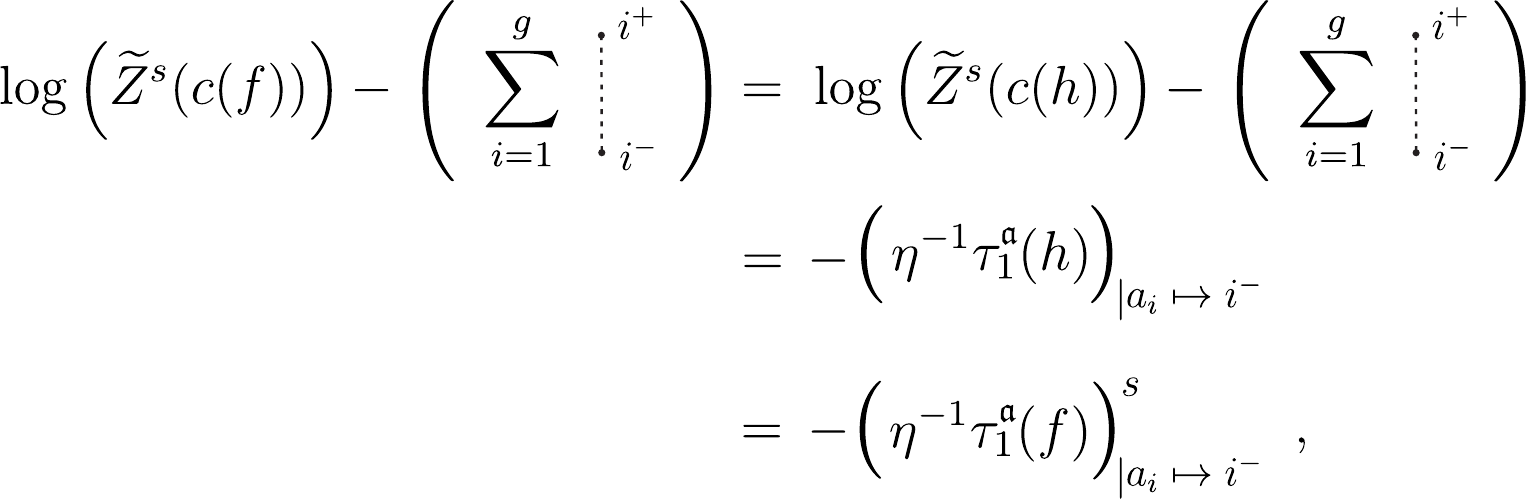}	
\end{center}
\noindent where $\left(\eta^{-1}\tau^{\mathfrak{a}}_1(f)\right)^s$ is the reduction of $\eta^{-1}\tau^{\mathfrak{a}}_1(f)$ modulo diagrams with i-deg $\geq 1$. 

On the other hand, by Theorem \ref{JTHLMOthmCHM}
\begin{equation}\label{JTHLMOequ1thm6.15}
\begin{split}
-\left(\eta^{-1}q(\tau_1(\psi))\right)_{|a_j\mapsto j^-, b_j\mapsto j^+} & = -\left(\left(\eta^{-1}\tau_1(\psi)\right)_{|a_j\mapsto j^-, b_j\mapsto j^+}\right)_{\mathfrak{a}\text{-deg} = 1}\\
& = \left(\widetilde{Z}^{Y,t}(c(\psi))\right)_{\mathfrak{a}\text{-deg} = 1, \text{i-deg} =1},
\end{split}
\end{equation}
and by Theorem \ref{JTHLMOthmvera}
\begin{equation}\label{JTHLMOequ2thm6.15}
-\left(\eta^{-1}\tau^L_2(\psi)\right)_{b_j\mapsto j^+} = \left(\widetilde{Z}^{Y,t,+}(c(\psi))\right)_{|\text{i-deg} = 2} = \left(\widetilde{Z}^{Y,t}(c(\psi))\right)_{|\mathfrak{a}\text{-deg} = 1, \text{i-deg} = 2}.
\end{equation}

\noindent Therefore

\begin{equation*}
\begin{split}
\widetilde{Z}^{Y,\mathfrak{a}}_1(c(f)) & = \widetilde{Z}^{Y,\mathfrak{a}}_1(c(h)) + \widetilde{Z}^{Y,\mathfrak{a}}_1(c(\psi))\\
& = \left(\widetilde{Z}^{Y,t}(c(\psi))\right)_{\mathfrak{a}\text{-deg} = 1, \text{i-deg} =1}+  \left(\widetilde{Z}^{Y,t}(c(\psi))\right)_{|\mathfrak{a}\text{-deg} = 1, \text{i-deg} = 2}\\
& = - \left(\eta^{-1}q(\tau_1(\psi))\right)_{|a_j\mapsto j^-, b_j\mapsto j^+}  -  \left(\eta^{-1}\tau^L_2(\psi)\right)_{b_j\mapsto j^+} \\
& = - \left(\eta^{-1}\big(p\tau^{\mathfrak{a}}_1(\psi) + \tau^L_2(\psi)\big)\right)_{|a_j\mapsto j^-, b_j\mapsto j^+}\\
& = - \left(\eta^{-1}\tau^{\mathfrak{a}}_1(\psi)\right)_{|a_j\mapsto j^-, b_j\mapsto j^+}\\
& = - \left(\eta^{-1}\tau^{\mathfrak{a}}_1(f)\right)^Y_{|a_j\mapsto j^-, b_j\mapsto j^+}
\end{split}
\end{equation*}

\noindent In the first equality we use Theorem \ref{JTHLMOprop1}, in the second we use Lemma \ref{JTHLMOlemma6.14}, in the third we use (\ref{JTHLMOequ1thm6.15}) and (\ref{JTHLMOequ2thm6.15}), in the fourth we use Lemma \ref{JTHLMOlemma6.10} and the homomorphism property of $\eta^{-1}$, and in the fifth  we use (\ref{JTHLMOequ12}). Finally in the sixth equality we use~(\ref{JTHLMOthmequ06.15}).
\end{proof}

\begin{remark} Theorems \ref{JTHLMOthmCHM} and \ref{JTHLMOthmvera} are valid in the context of homology cobordisms. This suggests that the first alternative Johnson homomorphism could be defined on $\mathcal{IC}^{\mathfrak{a}}$ and that the statement of  Theorem \ref{JTHLMOthm6.15} could be generalized to $\mathcal{IC}^{\mathfrak{a}}$. It is very likely possible to read the $0$-th alternative Johnson homomorphism $\tau^{\mathfrak{a}}_0(h)$ for $h\in\mathcal{L}$ in the $\mathfrak{a}$-deg $=0$  part of $\widetilde{Z}^t(c(h))$.
\end{remark}

\subsection{Higher alternative Johnson homomorphisms and the LMO functor}\label{section6.3} The aim of this subsection is to prove an analogue of Theorem \ref{JTHLMOthm6.15} for $\tau^{\mathfrak{a}}_m$ with $m\geq 2$. That is, we want to prove the following.

\begin{theorem}\label{JTHLMOthmvera2} Let $m\geq 2$. If $f\in J^{\mathfrak{a}}_m\mathcal{M}$, then
$$\widetilde{Z}^{Y,t}(c(f)) = \varnothing -  \left(\eta^{-1}\tau^{\mathfrak{a}}_m(f)\right)_{|a_j\mapsto j^-, b_j\mapsto j^+} + (\mathfrak{a}\text{\emph{-deg}} > m).$$

\noindent Or equivalently $\widetilde{Z}^{Y,\mathfrak{a}}_m(c(f))= - \left(\eta^{-1}\tau^{\mathfrak{a}}_m(f)\right)_{|a_j\mapsto j^-, b_j\mapsto j^+}$.
\end{theorem}

An immediate consequence of the above theorem is the following.

\begin{corollary} For $f\in J^{\mathfrak{a}}_m\mathcal{M}$ the value
$$\widetilde{Z}^{Y,\mathfrak{a}}_m(c(f))_{j^+\mapsto b_j, j^-\mapsto a_j}\in \mathcal{T}^{\mathfrak{a}}_m(B\oplus A)$$
is independent of the choice of a Drinfeld associator. 
\end{corollary}

One of the key points in the proof of Theorem \ref{JTHLMOthmvera2} is to show that the LMO functor defines  an \emph{alternative symplectic expansion} of $\pi$. We use several results and definitions from  \cite{MR2903772} and \cite{MR3828784}. 

\para{Alternative symplectic expansions}  Let $H_{\mathbb{Q}}$ be the $\mathbb{Q}$-module $H_1(\Sigma,\mathbb{Q})= H\otimes \mathbb{Q}$. Denote by $T(H_{\mathbb{Q}})$ the free associative  $\mathbb{Q}$-algebra generated by $H_{\mathbb{Q}}$ in degree $1$, that is, $T(H_{\mathbb{Q}})$ is the tensor algebra of $H_{\mathbb{Q}}$ and let $\widehat{T}(H_{\mathbb{Q}})$ denote its degree completion. Let $A_{\mathbb{Q}}= A\otimes \mathbb{Q}$ and $B_{\mathbb{Q}}=B\otimes\mathbb{Q}$. Let $T(B_\mathbb{Q};A_{\mathbb{Q}})$ be the free associative $\mathbb{Q}$-algebra generated by $B_{\mathbb{Q}}$ in degree $1$ and $A_{\mathbb{Q}}$ in degree $2$. We call the induced degree in $T(B_\mathbb{Q},A_{\mathbb{Q}})$ the \emph{alternative degree}. Hence
$$T(B_\mathbb{Q};A_{\mathbb{Q}}) = \mathbb{Q} \oplus  B_{\mathbb{Q}} \oplus (A_{\mathbb{Q}} \oplus (B_{\mathbb{Q}}\otimes B_{\mathbb{Q}})) \oplus ((A_{\mathbb{Q}}\otimes B_{\mathbb{Q}})\oplus  (B_{\mathbb{Q}}\otimes A_{\mathbb{Q}}) \oplus (B_{\mathbb{Q}}\otimes B_{\mathbb{Q}}\otimes B_{\mathbb{Q}})) \oplus\cdots$$ 
Denote by $\widehat{T}(B_{\mathbb{Q}};A_{\mathbb{Q}})$ the completion of $T(B_\mathbb{Q};A_{\mathbb{Q}})$ with respect to the alternative degree. Notice that $\widehat{T}(H_{\mathbb{Q}})$ and $\widehat{T}(B_{\mathbb{Q}};A_{\mathbb{Q}})$ are complete Hopf algebras.

\begin{definition}\label{defexpansion} An \emph{expansion} of $\pi$ is a multiplicative map $\theta: \pi\rightarrow \widehat{T}(H_{\mathbb{Q}})$ such that $\theta(x) = 1 + \{x\} + (\text{deg} > 1)$ for all $x\in\pi$. Here $\{x\}$ denotes $x\Gamma_2\pi\otimes 1 \in H_{\mathbb{Q}}$. Moreover, we say that an expansion $\theta$ is \emph{group-like} if it takes values in the group of group-like elements of $\widehat{T}(H_{\mathbb{Q}})$. 
\end{definition}

\begin{definition}\label{defaltexpansion} An \emph{alternative expansion} of $\pi$ relative to $\mathbb{A}$ is a multiplicative map $\theta: \pi\rightarrow \widehat{T}(B_{\mathbb{Q}}, A_{\mathbb{Q}})$ which takes values in the group of group-like elements of $\widehat{T}(B_{\mathbb{Q}}; A_{\mathbb{Q}})$ and such that:
\begin{itemize}
\item $\theta(x) = 1 + \{x\} + (\mathfrak{a}\text{-deg} > 1)$ for all $x\in\pi$ and
\item $\theta(\alpha) = 1 + \{\alpha\} + (\mathfrak{a}\text{-deg} > 2)$ for all $\alpha\in\mathbb{A}$.
\end{itemize}
Here $\{x\}$ denotes $x K_2\otimes 1\in B_{\mathbb{Q}}$ for $x\in\pi$, and $\{\alpha\}$ denotes $\alpha K_3\otimes 1\in (A_{\mathbb{Q}} \oplus (B_{\mathbb{Q}}\otimes B_{\mathbb{Q}}))$ for $\alpha\in\mathbb{A}$.
\end{definition}

\begin{remark} Definition \ref{defaltexpansion} implies that for $i\geq 1$ and $x\in K_i$, $\theta(x)= 1 + (xK_{i+1})\otimes 1 + (\mathfrak{a}\text{-deg} > i)$. Hence an alternative expansion of $\pi$ relative to $\mathbb{A}$ is an \emph{expansion of the $N$-series $(K_i)_{i\geq 1}$} in the sense of \cite[Section 12]{MR3828784}.
\end{remark}

Notice that $T(H_{\mathbb{Q}})= T(B_\mathbb{Q};A_{\mathbb{Q}})$ as $\mathbb{Q}$-algebras as soon as we have chosen a section of $\iota_*:H\rightarrow H'$ with isotropic image and identified $B=H'$ to the image of this section. Thus if $\theta$ is an alternative expansion of $\pi$ relative to $\mathbb{A}$, then, in particular, $\theta$ is a group-like expansion of $\pi$. We denote this group-like expansion by $\theta'$.

\begin{example}\label{examplealtexpansion} Consider the free basis $\{\alpha_i,\beta_i\}$ of $\pi$ defined by the system of meridians and parallels of Figure \ref{figuraLMO1artalt} and let $\{a_i,b_i\}$ be the induced symplectic basis of $H$. We identify $H'=B$ with the subgroup of $H$ generated by $\{b_1,\ldots,b_g\}$. Define $\theta:\pi\rightarrow \widehat{T}(B;A)$ by $\theta(\alpha_i)= \text{exp}(a_i)$ and $\theta(\beta_i)=\text{exp}(b_i)$. Then $\theta$ is an alternative expansion of~$\pi$ relative to $\mathbb{A}$.
\end{example}

Recall that the intersection form $\omega'$ of the handlebody, see (\ref{wthbasis2}), determines an element $\Omega'\in\mathfrak{Lie}_3(B;A)$. The element $\Omega'$ was described explicitly in  (\ref{JTH2equ6}).

\begin{definition}\cite[Definition 2.15]{MR2903772} An expansion $\theta$ of $\pi$  is said to be \emph{symplectic} if it is group-like and $\theta([\partial\Sigma])= \text{exp}(-\Omega')$.
\end{definition}

\begin{definition} An alternative expansion $\theta$ of $\pi$  relative to $\mathbb{A}$ is said to be \emph{symplectic} if  $\theta([\partial\Sigma])= \text{exp}(-\Omega')$.
\end{definition}

In \cite[Lemma 2.16]{MR2903772} Massuyeau shows that symplectic expansions exist by ``deforming'' the expansion of Example \ref{examplealtexpansion}. It can be verified that the constructed symplectic expansion in that lemma is actually an alternative symplectic expansion of $\pi$ relative to~$\mathbb{A}$. We will see that the LMO functor also gives an example of an alternative symplectic expansion of $\pi$ relative to $\mathbb{A}$.

\para{Completions of the group algebra $\mathbb{Q}[\pi]$} Let $\mathbb{Q}[\pi]$ be the group $\mathbb{Q}$-algebra of $\pi$. We have considered two $N$-series of $\pi$; the lower central series $\Gamma_+ = (\Gamma_m\pi)_{m\geq 1}$ and the $N$-series $K_+ = (K_m)_{m\geq 1}$ defined in (\ref{JTHequ11}). Each one of these defines a \emph{filtration} of $\mathbb{Q}[\pi]$, that is, a decreasing sequence $\mathbb{Q}[\pi]=F_0\supseteq F_1\supseteq F_2 \supseteq\cdots$ of additive subgroups of $\mathbb{Q}[\pi]$ indexed by non-negative integers such that $F_mF_n\subseteq F_{m+n}$ for $m,n\geq 0$. Let $I$ be the augmentation ideal of $\mathbb{Q}[\pi]$. For $m\geq 1$ it is well known that
$$ I^m = \big\langle (x_1-1)\cdots (x_p-1)\ |\ p\geq 1, x_i\in\Gamma_{m_i}\pi \text{ and } m_1+\cdots+m_p\geq m\big\rangle,$$
where the angle brackets stand for the generated subspace of $\mathbb{Q}[\pi]$. For $m\geq 1$, set
$$R_m = \big\langle (x_1-1)\cdots (x_p-1)\ |\ p\geq 1, x_i\in K_{m_i} \text{ and } m_1+\cdots+m_p\geq m\big\rangle.$$
This way we have the filtrations $\{I^m\}_{m\geq 0}$ and $\{R_m\}_{m\geq 0}$ of $\mathbb{Q}[\pi]$ where we set $I^0=R_0=\mathbb{Q}[\pi]$. These filtrations   define inverse systems $\{\mathbb{Q}[\pi]/I^m\}_m$ and $\{\mathbb{Q}[\pi]/R_m\}_m$. 

Consider the $I$-adic and $R$-adic completions of $\mathbb{Q}[\pi]$, that is, the inverse limits
\begin{equation*}
\widehat{\mathbb{Q}[\Gamma_+]} = \varprojlim_m (\mathbb{Q}[\pi]/I^m)   \text{\ \ \ \ \ and \ \ \ \ \ } \widehat{\mathbb{Q}[K_+]} = \varprojlim_m (\mathbb{Q}[\pi]/R_m).
\end{equation*}
Notice that $\widehat{\mathbb{Q}[\Gamma_+]}$ and $\widehat{\mathbb{Q}[K_+]}$ are filtered complete Hopf algebras, with filtrations $\{\widehat{I}^m\}_{m\geq 0}$ and $\{\widehat{R}_m\}_{m\geq 0}$ defined by

\begin{equation*}
\widehat{I}^m = \varprojlim_l (I^m/I^l)   \text{\ \ \ \ \ and \ \ \ \ \ } \widehat{R}_m = \varprojlim_l (R_m/R_l),
\end{equation*}
for $m\geq 0$. From now on, let $\theta:\pi\rightarrow \widehat{T}(B_{\mathbb{Q}};A_{\mathbb{Q}})$ be an alternative expansion of $\pi$ relative to $\mathbb{A}$ and denote by $\theta':\pi\rightarrow \widehat{T}(H_{\mathbb{Q}})$ the associated group-like expansion of $\pi$. The Quillen's description \cite{MR0231919} of the associated graded of the filtered ring $\mathbb{Q}[\pi]$ with respect to  $\{I^m\}_m$, can be generalized \cite[Theorem 11.2]{MR3828784} to describe the associated graded of $\mathbb{Q}[\pi]$ with respect to  $\{R_m\}_m$. Moreover, we have the following.

\begin{proposition}\cite[Proposition 12.2]{MR3828784}\cite[Proposition 2.10]{MR2903772}\label{JTHLMO6.3prop1} The maps $\theta$ and $\theta'$ extend uniquely to complete Hopf algebra isomorphisms 
\begin{equation*}
\hat{\theta}:  \widehat{\mathbb{Q}[K_+]}\longrightarrow \widehat{T}(B_{\mathbb{Q}};A_{\mathbb{Q}})   \text{\ \ \ \ \ and \ \ \ \ \ } \hat{\theta}':  \widehat{\mathbb{Q}[\Gamma_+]}\longrightarrow \widehat{T}(H_{\mathbb{Q}}),
\end{equation*}
which are the identity at the graded level.
\end{proposition}

Since $\Gamma_m\pi\subseteq K_m\subseteq \Gamma_{\left\lceil m/2\right\rceil}\pi$ for $m\geq 1$, see (\ref{JTHequ11}), then $I^m\subseteq R_m$ and $R_m\subseteq I^{\left\lceil m/2\right\rceil}$. Hence, the identity automorphism of $\pi$ induces a morphism of inverse systems $$\{u_m:\mathbb{Q}[\pi]/I^m\longrightarrow \mathbb{Q}[\pi]/R_m\}$$
which induces an isomorphism 
$$U:\widehat{\mathbb{Q}[\Gamma_+]}\longrightarrow \widehat{\mathbb{Q}[K_+]}$$
of complete Hopf algebras. The  following two results are straightforward.

\begin{proposition}\label{JTHLMO6.3prop2}
The diagram
\begin{equation*}
\xymatrix{  \widehat{\mathbb{Q}[\Gamma_+]}\ar[r]^{U}_{\cong}\ar[d]_{\hat{\theta}'}^{\cong} & \widehat{\mathbb{Q}[K_+]} \ar[d]^{\hat{\theta}}_{\cong} \\
						\widehat{T}(H)\ar@{=}[r] & \widehat{T}(B;A)}
\end{equation*}
is commutative.
\end{proposition}

For a complete Hopf algebra $F$ denote by $\text{Aut}(F)$ the group of automorphisms of $F$.

\begin{corollary}
The diagram
\begin{equation*}
\xymatrixcolsep{5pc}\xymatrix{\text{\emph{Aut}}(\widehat{\mathbb{Q}[\Gamma_+]})\ar[r]^{U(\_\_) U^{-1}}\ar[d]_{\hat{\theta}'(\_\_)\hat{\theta}'^{-1}} & \text{\emph{Aut}}(\widehat{\mathbb{Q}[K_+]}) \ar[d]^{\hat{\theta}(\_\_)\hat{\theta}^{-1}} \\
						\text{\emph{Aut}}(\widehat{T}(H))\ar@{=}[r] & \text{\emph{Aut}}(\widehat{T}(B;A))}
\end{equation*}
is commutative.
\end{corollary}

Recall that we can restrict the Dehn-Nielsen representation (\ref{JTHequ1}) to the Lagrangian mapping class group $\mathcal{L}$. Now, if $h\in\mathcal{L}$, we have  $\rho(h)(\Gamma_m\pi)=h_{\#}(\Gamma_m\pi)\subseteq \Gamma_m\pi$ and $\rho(h)(K_m)=h_{\#}(K_m)\subseteq K_m$, see Lemma \ref{JTHlemma2}. This way, we obtain representations
$$\hat{\rho}': \mathcal{L}\longrightarrow \text{Aut}(\widehat{\mathbb{Q}[\Gamma_+]})$$
and
$$\hat{\rho}: \mathcal{L}\longrightarrow \text{Aut}(\widehat{\mathbb{Q}[K_+]}).$$

\noindent Notice that for $h\in\mathcal{L}$, the diagram 
\begin{equation}\label{JTHLMOequ6.19}
\xymatrix{  \widehat{\mathbb{Q}[\Gamma_+]}\ar[r]^{U}\ar[d]_{\hat{\rho}'(h)} & \widehat{\mathbb{Q}[K_+]} \ar[d]^{\hat{\rho}(h)} \\
						 \widehat{\mathbb{Q}[\Gamma_+]}\ar[r]^{U} & \widehat{\mathbb{Q}[K_+]}}
\end{equation}
is commutative. Besides, using Proposition \ref{JTHLMO6.3prop1}, we define
\begin{equation*}
\rho^{\hat{\theta}'}:\mathcal{L}\longrightarrow \text{Aut}(\widehat{T}(H_{\mathbb{Q}}))   \text{\ \ \ \ \ and \ \ \ \ \ } \rho^{\hat{\theta}}: \mathcal{L}\longrightarrow \text{Aut}(\widehat{T}(B_{\mathbb{Q}}; A_{\mathbb{Q}})),
\end{equation*}

\noindent by $\rho^{\hat{\theta'}}(h) = \hat{\theta'}\hat{\rho}'(h)\hat{\theta}'^{-1}$  and $\rho^{\hat{\theta}}(h) = \hat{\theta}\hat{\rho}(h)\hat{\theta}^{-1}$ for $h\in\mathcal{L}$.

\begin{proposition}\label{JTHLMO6.3prop6.26} The diagram 
\begin{equation}
\xymatrix{  \mathcal{L}\ar[r]^{\rho^{\hat{\theta}'}}\ar[d]_{\rho^{\hat{\theta}}} & \text{\emph{Aut}}\big(\widehat{T}(H_{\mathbb{Q}})\big) \ar@{=}[ld] \\
						 \text{\emph{Aut}}\big(\widehat{T}(B_{\mathbb{Q}}; A_{\mathbb{Q}})\big) & }
\end{equation}
is commutative.
\end{proposition}
\begin{proof} Let $h\in\mathcal{L}$, then
\begin{equation*}
\begin{split}
\rho^{\hat{\theta}}(h) & = \hat{\theta}\hat{\rho}(h)\hat{\theta}^{-1}\\
 & = \hat{\theta}U\hat{\rho}'(h)U^{-1}\hat{\theta}^{-1}\\
 & = \hat{\theta}' \hat{\rho}'(h)\hat{\theta}'^{-1}\\
 & = \rho^{\hat{\theta}'}(h).
\end{split}
\end{equation*}
In the second equality we use (\ref{JTHLMOequ6.19}), and in the third equality we use Proposition~\ref{JTHLMO6.3prop2}.
\end{proof}

From Proposition \ref{JTHprop3} we have $J^{\mathfrak{a}}_2\mathcal{M}\subseteq J_1\mathcal{M}=\mathcal{I}$. Hence, by considering the restriction of $\rho^{\hat{\theta}'}$ to $\mathcal{I}$ and the restriction of $\rho^{\hat{\theta}}$ to $J^{\mathfrak{a}}_2\mathcal{M}$ and using Proposition \ref{JTHLMO6.3prop6.26} we obtain the following. 

\begin{corollary}\label{JTHLMO6.3cor6.27}
The diagram 
\begin{equation*}
\xymatrix{J^{\mathfrak{a}}_2\mathcal{M}\ar[r]^{\subset}\ar[d]_{\rho^{\hat{\theta}}} & \mathcal{I} \ar[d]^{\rho^{\hat{\theta}'}} \\
\text{\emph{Aut}}\big(\widehat{T}(B_{\mathbb{Q}}, A_{\mathbb{Q}})\big)\ar@{=}[r] & \text{\emph{Aut}}\big(\widehat{T}(H_{\mathbb{Q}})\big)}
\end{equation*}
is commutative.
\end{corollary}

Notice that the maps $\rho^{\hat{\theta}'}$ and $\rho^{\hat{\theta}}$ can be defined on $\mathcal{M}$ and $\mathcal{L}$, respectively. Hence, Corollary \ref{JTHLMO6.3cor6.27} still holds when replacing $\mathcal{I}$ by $\mathcal{M}$  and $J_2^{\mathfrak{a}}\mathcal{M}$ by $\mathcal{L}$,   but we are interested in taking the logarithm on the images of  $\rho^{\hat{\theta}'}$ and $\rho^{\hat{\theta}}$ and we can do this only if we restrict to the Torelli-type groups $\mathcal{I}$ and $\mathcal{I}^{\mathfrak{a}}$, see \cite[Lemma 12.5]{MR3828784}. The following is an instance of one of the main results in \cite{MR3828784}.

\begin{theorem}\cite[Theorem 3.5]{MR2903772}\cite[Theorem 12.6, Remark 12.8]{MR3828784}\label{JTHLMOthm6.28} Consider the filtration-preserving maps
\begin{equation*}
\varrho^{\theta}:\mathcal{I}^{\mathfrak{a}}\longrightarrow \widehat{\text{\emph{Der}}}^+\big(\mathfrak{Lie}(B_{\mathbb{Q}}; A_{\mathbb{Q}})\big)   \text{\ \ \ \ \ and \ \ \ \ \ } \varrho^{\theta'}: \mathcal{I}\longrightarrow \widehat{\text{\emph{Der}}}^+\big(\mathfrak{Lie}(H_\mathbb{Q})\big),
\end{equation*}
defined by $\varrho^{\theta}(h) = \text{\emph{log}}\big(\rho^{\hat{\theta}}(h)\big)$ and $\varrho^{\theta'}(f) = \text{\emph{log}}\big(\rho^{\hat{\theta}'}(f)\big)$ for $h\in\mathcal{I}^{\mathfrak{a}}$ and $f\in\mathcal{I}$. Then, $\varrho^{\theta}$ determines all the alternative Johnson homomorphisms and $\varrho^{\theta'}$ determines all the Johnson homomorphisms. That is, for $m\geq 1$, $h\in J^{\mathfrak{a}}_m\mathcal{M}$ and $f\in J_m\mathcal{M}$, we have
\begin{equation}\label{JTHLMOthm6.28equ1}
\tau^{\mathfrak{a}}_m(h) = \big(\varrho^{\theta}(h)_{|B\oplus A}\big)_m \in D_m(\mathfrak{Lie}(B_{\mathbb{Q}}, A_{\mathbb{Q}})),
\end{equation}
and
\begin{equation}\label{JTHLMOthm6.28equ2}
\tau_m(f) = \big(\varrho^{\theta'}(f)_{|H}\big)_m \in \text{\emph{Hom}}(H,\mathfrak{Lie}_{m+1}(H_{\mathbb{Q}})),
\end{equation}
where $D_m(\mathfrak{Lie}(B_{\mathbb{Q}}, A_{\mathbb{Q}}))$ is defined by considering the rational version of Equation (\ref{JTH2equ7}). The subscripts $m$ in the right-hand side of the above equations denote the terms of degree $m$ in $\varrho^{\theta}(h)_{|B\oplus A}$ and $\varrho^{\theta'}(f)_{|H}$, respectively.
\end{theorem}

\begin{remark} Notice that in the left-hand sides of (\ref{JTHLMOthm6.28equ1}) and (\ref{JTHLMOthm6.28equ2}) we are actually considering the rationalization of the Johnson-type homomorphims, but there is no loss of information by doing this. 
\end{remark}

Furthermore, if $\theta$ is symplectic, then the maps $\varrho^{\theta}$ and $\varrho^{\theta'}$ take values in the completions  $\widehat{\text{Der}}^{+,\omega}\big(\mathfrak{Lie}(B_{\mathbb{Q}}; A_{\mathbb{Q}})\big)$ and $\widehat{\text{Der}}^{+,\omega}\big(\mathfrak{Lie}(H_{\mathbb{Q}})\big)$ of positive symplectic derivations of $\mathfrak{Lie}(B_{\mathbb{Q}}; A_{\mathbb{Q}})$ and $\mathfrak{Lie}(H_{\mathbb{Q}})$, respectively. Now
$$\widehat{\text{Der}}^{+,\omega}\big(\mathfrak{Lie}(B_{\mathbb{Q}}; A_{\mathbb{Q}})\big)\cong \prod_ m D_m(B_{\mathbb{Q}}; A_{\mathbb{Q}})\cong \mathcal{T}^{\mathfrak{a}}(B\oplus A),$$
where the first isomorphism is given by Proposition \ref{JTH2prop1} and the second by Proposition~\ref{JTH2prop26}. Similarly
$$\widehat{\text{Der}}^{+,\omega}\big(\mathfrak{Lie}(H_{\mathbb{Q}})\big)\cong \prod_m D_m(H_{\mathbb{Q}}) \cong \mathcal{A}^{t,c}(H),$$
where the second isomorphism is given by Theorem \ref{JTH2thm23}. This way, with the hypothesis that $\theta$  is symplectic, Theorem \ref{JTHLMOthm6.28} can be restated by saying that for $m\geq 1$ the diagrams 
\begin{equation}\label{JTHLMOequ23}
\xymatrixrowsep{3pc}\xymatrix{J^{\mathfrak{a}}_m\mathcal{M}\ar[r]^{\subset}\ar[d]_{\eta^{-1}\tau^{\mathfrak{a}}_m} & \mathcal{I}^{\mathfrak{a}} \ar[d]^{\eta^{-1}\varrho^{\theta}} \\
\mathcal{T}^{\mathfrak{a}}_m(B\oplus A)\ar@{<<-}[r] & \mathcal{T}^{\mathfrak{a}}(B\oplus A)} \text{\ \ \ \ \ and\ \ \ \ \ \ \ \ \ \ } \xymatrixrowsep{3pc}\xymatrix{J_m\mathcal{M}\ar[r]^{\subset}\ar[d]_{\eta^{-1}\tau_m} & \mathcal{I} \ar[d]^{\eta^{-1}\varrho^{\theta'}} \\
\mathcal{A}^{t,c}_m(H)\ar@{<<-}[r] & \mathcal{A}^{t,c}(H)}
\end{equation} 
are commutative. In the same way, if $\theta$ is symplectic, Corollary \ref{JTHLMO6.3cor6.27} can be restated by saying that the diagram
\begin{equation}\label{JTHLMOequ24}
\xymatrixrowsep{3pc}\xymatrix{J^{\mathfrak{a}}_2\mathcal{M}\ar[r]^{\subset}\ar[d]_{\eta^{-1}\varrho^{\theta}} & \mathcal{I} \ar[d]^{\eta^{-1}\varrho^{\theta'}} \\
\mathcal{T}^{\mathfrak{a}}(B\oplus A)\ar@{=}[r] & \mathcal{A}^{t,c}(H)}
\end{equation}
is commutative.

\para{The LMO functor defines a symplectic alternative expansion} We want to apply diagrams (\ref{JTHLMOequ23}) and (\ref{JTHLMOequ24}) with a particular symplectic alternative expansion of $\pi$ relative to $\mathbb{A}$. In \cite[Section 5]{MR2903772}, G. Massuyeau constructed a symplectic expansion of $\pi$ from the LMO functor. It turns out that this expansion is  a symplectic alternative expansion of $\pi$ relative to $\mathbb{A}$. Let us recall this construction.

Fix two points $p,q\in\text{int}(\Sigma)$.  A \emph{bottom knot} in $\Sigma\times [-1,1]$ is the isotopy class (relative to the boundary) of a connected framed oriented tangle starting at $q\times \{-1\}$ and ending at $p\times \{-1\}$, see Figure \ref{figuraJTHLMO14artalt} $(a)$ for an example. Let $\mathcal{B}$ denote the set of bottom knots in $\Sigma\times[-1,1]$. There is a monoid structure in $\mathcal{B}$. If $K,L\in \mathcal{B}$, then $K\cdot L$ is the bottom knot obtained from $K$ and $L$ by joining $K$ and $L$ as shown in Figure \ref{figuraJTHLMO14artalt} $(b)$.
\begin{figure}[ht!] 
										\centering
                        \includegraphics[scale=0.7]{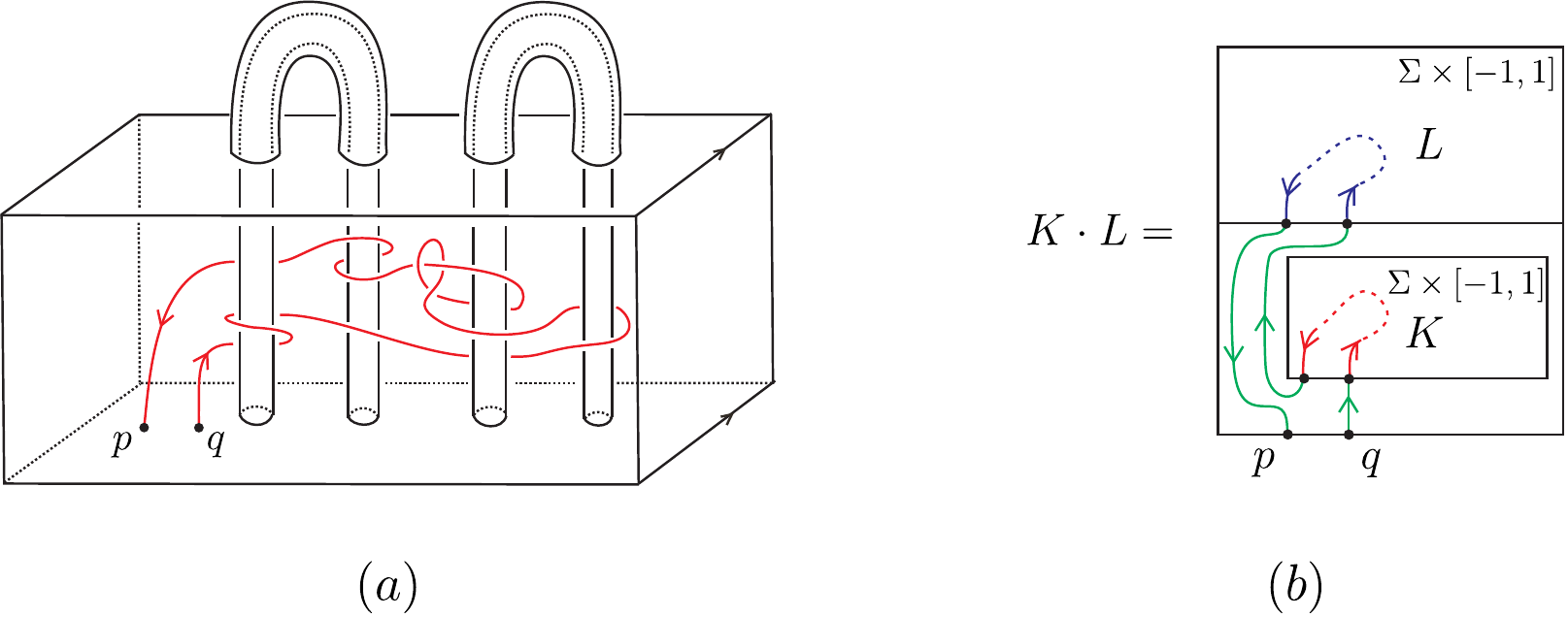}												\caption{$(a)$ Bottom knot in $\Sigma\times [-1,1]$ and  $(b)$ monoid structure in $\mathcal{B}$.}
\label{figuraJTHLMO14artalt}
\end{figure}

Two bottom knots $K,K'\in\mathcal{B}$ are said to be \emph{homotopic}, denoted $K\simeq K'$, if $K'$ can be obtained from $K$ by framing changes and a finite number of crossing changes. This relation is compatible with the monoid structure, that is, if $K,K',L,L'\in\mathcal{B}$ are such that $K\simeq K'$ and $L\simeq L'$, then $K\cdot L \simeq K'\cdot L'$. There is a canonical monoid morphism 
\begin{equation}
\ell: \mathcal{B}/_{\simeq}\longrightarrow  \pi,
\end{equation} 
which assigns to the homotopy class of $K\in\mathcal{B}$ the based loop in $\Sigma\times[-1,1]$ as shown in Figure \ref{figuraJTHLMO15artalt}. Then identify $\pi$ with $\pi_1(\Sigma\times[-1,1],*)$, see \cite[Lemma 5.3]{MR2903772}.
\begin{figure}[ht!] 
										\centering
                        \includegraphics[scale=0.7]{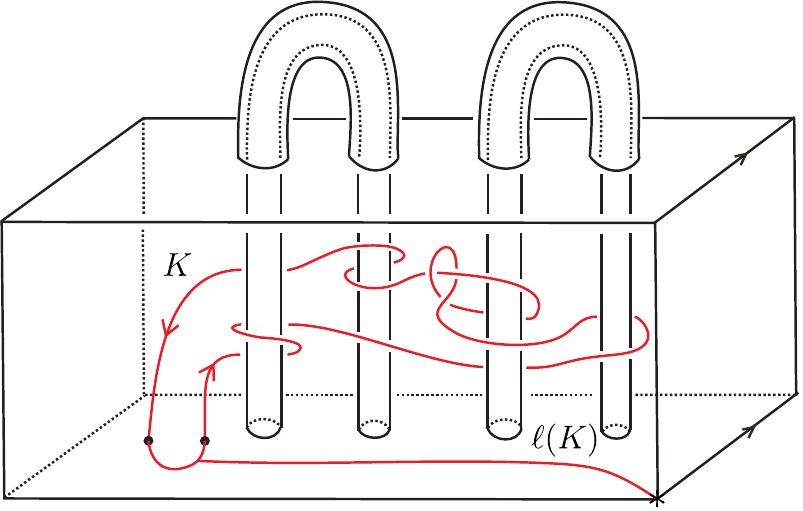}												\caption{The based loop $\ell(K)$.}
\label{figuraJTHLMO15artalt}
\end{figure}

Now, a bottom knot $K\in\mathcal{B}$ gives rise to an element in $\mathcal{LC}ob(g,g+1)$ by ``digging" along $K$, more precisely, let $(M_K, m)$ be the cobordism obtained from $\Sigma\times[-1,1]$ by removing a tubular neighborhood of $K$ and  define the parametrization $m$ on the first handle of the bottom surface of $M_K$ by using the framing of $K$ and as the identity elsewhere. We continue to denote the cobordism $(M_K,m)$ by $K$ and we endow its top and bottom with the right-handed non-associative words as in Convention \ref{convention1}.

 This way we have $K\in\mathcal{LC}ob_q(g,g+1)$ and we can apply the LMO functor to it,  to obtain $\widetilde{Z}(K)\in{}^{ts}\!\!\mathcal{A}(g,g+1)\subseteq \mathcal{A}(\left\lfloor g\right\rceil^+\sqcup \left\lfloor g+1\right\rceil^-)$.  Change the colors in $\widetilde{Z}(K)$ as follows
$$1^-\mapsto r \text{\ \ \ \ and\ \ \ \ } i^-\mapsto (i-1)^-, \ \ \ \forall i=2,\ldots,g+1;$$
so that, the variable $r$ refers to the bottom knot. Thus $\widetilde{Z}(K)\in\mathcal{A}(\left\lfloor g\right\rceil^+\sqcup \left\lfloor g\right\rceil^-\sqcup\{r\})$.

\begin{example}\label{ejemploJTHLMON} In Example \ref{ejemploLMO5} we considered a cobordism $N_i\in{}^s\!\mathcal{LC}ob(g,g+1)$, for $i=1,\ldots,g$, with bottom-top tangle presentation shown in Figure \ref{figuraLMO7artalt}. Analyzing carefully  Figure \ref{figuraLMO7artalt} we see that the cobordism $N_i$ corresponds to the bottom knot, also denoted  $N_i$, in $\Sigma\times[-1,1]$ such that $\ell(N_i)$ is the homotopy class of the meridian $\alpha_i$ in~$\Sigma$.
\end{example}

Recall the space $\mathcal{H}(r)$ defined in Example \ref{ejemploJDhomotopy}. Hence, $\widetilde{Z}(K) \text{ mod } \mathcal{H}(r)$ is an exponential series of tree-like Jacobi diagrams with at most one $r$-colored leg and which only depends on the homotopy class of $K$, see \cite[Lemma 5.5]{MR2903772}. Moreover, the series consisting of the terms without $r$-colored legs in $\widetilde{Z}(K) \text{ mod } \mathcal{H}(r)$ is exactly the identity morphism in ${}^{ts}\!\!\mathcal{A}(g,g)$. To sum up

\begin{center}
\includegraphics[scale=0.74]{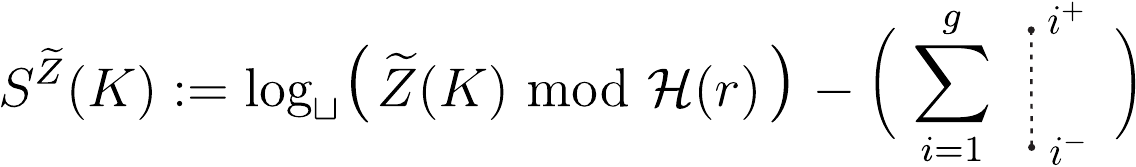}	
\end{center}

\noindent is a series of  connected tree-like Jacobi diagrams with legs colored by $\left\lfloor g\right\rceil^+\sqcup \left\lfloor g\right\rceil^-\sqcup\{r\}$  and with exactly one $r$-colored leg. Hence, by considering the $r$-colored leg as a root, we can see  $S^{\widetilde{Z}}(K)$ as an element in $\widehat{\mathfrak{Lie}}(B_{\mathbb{Q}}, A_{\mathbb{Q}})$ after the replacement of colors $i^+\mapsto b_i$ and $i^-\mapsto a_i$ for $i=1,\ldots,g$. The map
\begin{equation}
\theta^{\widetilde{Z}}:\pi\longrightarrow \widehat{T}(B_{\mathbb{Q}}; A_{\mathbb{Q}}),
\end{equation} 
defined by $\theta^{\widetilde{Z}}(\ell(K))=\text{exp}_{\otimes} \big(S^{\widetilde{Z}}(K)\big)$ is a symplectic expansion of $\pi$, see \cite[Proposition~5.6]{MR2903772}.

\begin{proposition} The symplectic expansion $\theta^{\widetilde{Z}}:\pi\rightarrow \widehat{T}(B_{\mathbb{Q}}; A_{\mathbb{Q}})$ satisfies 
$$\theta^{\widetilde{Z}}(\alpha) = 1 + \{\alpha\} + (\mathfrak{a}\text{\emph{-deg}} > 2)$$
for all $\alpha\in \mathbb{A}$. Therefore $\theta^{\widetilde{Z}}$ is a symplectic alternative expansion of $\pi$ relative to $\mathbb{A}$.
\end{proposition}

\begin{proof}
By \cite[Proposition 5.6]{MR2903772} we know that for every $\alpha\in\mathbb{A}$, 
$$\theta^{\widetilde{Z}}(\alpha) = 1 + \{\alpha\} + (\text{deg} > 1).$$

First, notice that a tree of i-deg $> 1$ with one leg colored by $r$ (the root) and the other legs colored by elements of $B\oplus A$ gives rise to a Lie commutator in $\mathfrak{Lie}(B_{\mathbb{Q}}, A_{\mathbb{Q}})$ of $\mathfrak{a}$-deg $> 2$. Therefore, we only need to calculate the terms of i-deg $= 1$ in $S^{\widetilde{Z}}(K)$ for bottom knots $K$ such that $\ell(K)\in\mathbb{A}$.

Now $\mathbb{A}$ is the normal closure of the subgroup $\langle \alpha_i \ | \ i=1,\ldots,g \rangle$ of $\pi$ generated by the homotopy classes of the meridians and by Example \ref{ejemploJTHLMON}, the cobordism $N_i$ is such that $\ell(N_i)$ is the homotopy class of $\alpha_i$. Therefore, by the homomorphism property of $\widetilde{Z}_{\text{i-deg }=1}$,  it is enough to calculate the terms of i-deg $= 1$ in $\widetilde{Z}(N_i) \text{ mod } \mathcal{H}(r)$ and see whether they give rise to  Lie commutators in $\mathfrak{Lie}(B_{\mathbb{Q}}, A_{\mathbb{Q}})$ of $\mathfrak{a}$-deg $= 1$. In Example \ref{ejemploLMO9} we see that each of the terms with i-deg $=1$ in  $\widetilde{Z}(N_i) \text{ mod } \mathcal{H}(r)$ has one $r$-colored leg and  one $i^-$-colored leg, thus the associated commutator has $\mathfrak{a}$-deg $> 2$ which completes the proof.
\end{proof}

\para{Proof of Theorem \ref{JTHLMOthmvera2}} Let $\theta$ denote the symplectic alternative expansion of $\pi$ relative to $\mathbb{A}$ defined by the LMO functor and denote by $\theta'$ the associated symplectic expansion of $\pi$. In \cite[Theorem 5.13]{MR2903772} G. Massuyeau proved that the diagram
\begin{equation}\label{JTHLMOequ27}
\xymatrix{\mathcal{I}\ar[r]^-{c}\ar[d]_-{\eta^{-1}\varrho^{\theta'}} & \mathcal{LC}ob_q(g,g) \ar[ld]^-{-\text{log } \widetilde{Z}^{Y,t}} \\
\mathcal{A}^{t,c}(H) & }
\end{equation}
is commutative, where $c$ denotes the cylinder map and for $h\in\mathcal{I}$ we endow the top and bottom of $c(h)$ with the right-handed non-associative words as in Convention \ref{convention1}. In order to see $\text{log}\big( \widetilde{Z}(c(h))\big)$ as an element of $\mathcal{A}^{t,c}(H)$ we consider the change of colors $i^+\mapsto b_i$ and $i^-\mapsto a_i$ for $i=1,\ldots,g$. We obtain the desired result by putting together the commutative diagrams (\ref{JTHLMOequ23}), (\ref{JTHLMOequ24}) and (\ref{JTHLMOequ27}).\qed

\begin{remark} In fact \cite[Theorem 5.13]{MR2903772} is more general than the commutativity of diagram (\ref{JTHLMOequ27}); it  says that for every homology cylinder $M\in\mathcal{IC}$ we have $\eta^{-1}\varrho^{\theta'}(M) = - \text{log} \big(\widetilde{Z}^{Y,t}(M)\big)$. Besides, Theorem \ref{JTHLMOprop1} is proved in the setting of homology cobordisms. This suggests that our results   could be  generalized to the setting of homology cobordisms. More precisely, we expect that the alternative Johnson filtration and the alternative Johnson homomorphisms extend to homology cobordisms and that the diagrammatic version of such alternative Johnson homomorphisms can be read in the tree reduction of the LMO functor.
\end{remark}

\bibliographystyle{plain} 
\bibliography{bibliography_alt} 
\end{document}